\documentclass[10pt,reqno]{amsart}
\usepackage{hyperref}
\usepackage{float}
\usepackage{mathtools}
\usepackage[all]{xy}
\usepackage{amscd,amssymb,amsmath,latexsym,bm}
\usepackage[mathcal,mathscr]{euscript}
\usepackage{lipsum}
\usepackage{amsfonts}
\usepackage{graphicx}
\usepackage{epstopdf}
\usepackage{algorithmic}
\usepackage{hyperref}
\usepackage{xcolor}
\usepackage{upgreek}
\usepackage{enumerate}
\usepackage{ulem}
\usepackage{stmaryrd}
\usepackage{gensymb}			
\usepackage[utf8]{inputenc}
\usepackage[small,nohug]{diagrams}
\diagramstyle[labelstyle=\scriptstyle]
 
\newtheorem{theorem}{Theorem}[section]
\newtheorem{lemma}[theorem]{Lemma}
\newtheorem{corollary}[theorem]{Corollary}
\newtheorem{proposition}[theorem]{Proposition}
\newtheorem{remark}[theorem]{Remark}
\newtheorem{definition}[theorem]{Definition}
\newtheorem{example}[theorem]{Example}

\numberwithin{equation}{section}
\numberwithin{figure}{section}

\newcommand{\BM}{{\mathbb B}}
\newcommand{\CM}{{\mathbb C}}
\newcommand{\NM}{{\mathbb N}}
\newcommand{\QM}{{\mathbb Q}}
\newcommand{\RM}{{\mathbb R}}
\newcommand{\SM}{{\mathbb S}}

\newcommand{\ZM}{{\mathbb Z}}

\newcommand{\KM}{{\mathbb K}}

\newcommand{\Aa}{{\mathcal A}}

\newcommand{\Ii}{{\mathcal I}}
\newcommand{\Pp}{{\mathcal P}}

\newcommand{\Bb}{{\mathcal B}}
\newcommand{\Dd}{{\mathcal D}}
\newcommand{\Ff}{{\mathcal F}}
\newcommand{\Gg}{{\mathcal G}}

\newcommand{\Ss}{{\mathcal S}}

\newcommand{\Tt}{{\mathcal T}}
\newcommand{\Rr}{{\mathcal R}}
\newcommand{\Nn}{{\mathcal N}}
\newcommand{\Mm}{{\mathcal M}}
\newcommand{\Cc}{{\mathcal C}}

\newcommand{\Jj}{{\mathcal J}}

\newcommand{\Ll}{{\mathcal L}}
\newcommand{\Qq}{{\mathcal Q}}
\newcommand{\Kk}{{\mathcal K}}
\newcommand{\Hh}{{\mathcal H}}

\begin{document}
\title{A groupoid approach to interacting fermions}

\author{Bram Mesland}

\address{Mathematisch Instituut \\ Universiteit Leiden  \\ Niels Bohrweg 1 \\ 2333CA Leiden, Netherlands
\\
\href{mailto:b.mesland@math.leidenuniv.nl}{b.mesland@math.leidenuniv.nl}}

\author{Emil Prodan}

\address{Department of Physics and
\\ Department of Mathematical Sciences 
\\Yeshiva University 
\\New York, NY 10016, USA \\
\href{mailto:prodan@yu.edu}{prodan@yu.edu}}

\date{\today}

\begin{abstract} 
We consider the algebra $\dot\Sigma(\Ll)$ generated by the inner-limit derivations over the ${\rm GICAR}$ algebra of a fermion gas populating an aperiodic Delone set $\Ll$. Under standard physical assumptions such as finite interaction range, Galilean invariance of the theories and continuity with respect to the deformations of the aperiodic lattices, we demonstrate that the image of $\dot \Sigma(\Ll)$ through the Fock representation can be completed to a groupoid-solvable pro-$C^\ast$-algebra. Our result is the first step towards unlocking the $K$-theoretic tools available for separable $C^\ast$-algebras for applications in the context of interacting fermions.
\end{abstract}

\thanks{Emil Prodan was supported by the U.S. National Science Foundation through the grants DMR-1823800 and CMMI-2131760.}

\maketitle

{\scriptsize \tableofcontents}

\setcounter{tocdepth}{1}

\section{Introduction and Main Statements}
\label{Sec:Introduction}

The physical observables of a self-interacting Fermi gas populating a uniformly discrete lattice $\Ll \subset \RM^d$ generate the $C^\ast$-algebra of canonical anti-commutation relations (CAR) over the Hilbert space $\ell^2(\Ll)$ \cite{BratteliBook2}, denoted here by ${\rm CAR}(\Ll)$. Our interest is in the dynamics of these observables, formalized as a point-wise continuous one-parameter group 
\begin{equation*}
\bm \alpha : \RM \to {\rm Aut}\big({\rm CAR}(\Ll)\big).
\end{equation*} 
We will refer to $\bm \alpha$ as the time evolution. Let $\Dd_{\bm \alpha}$ be the set of elements $A\in{\rm CAR}(\Ll)$ for which $\bm \alpha_t(A)$ is first order differentiable w.r.t. $t\in \RM$. Then $\Dd_{\bm \alpha}$ is stable under addition and multiplication, hence it is actually a sub-algebra of ${\rm CAR}(\Ll)$. Furthermore,
\begin{equation*}
\delta_{\bm \alpha} (A) : = \lim_{t \rightarrow 0} \tfrac{1}{t}\big (\bm \alpha_t(A)-A \big )
\end{equation*}
defines a ${\rm CAR}(\Ll)$-valued linear map over $\Dd_{\bm \alpha}$, which obeys Leibniz's rule 
\begin{equation*}
\delta_{\bm \alpha}(AA') = \delta_{\bm \alpha}(A) A' + A \delta_{\bm \alpha}(A'), \quad A,A' \in \Dd_{\bm \alpha},
\end{equation*}
hence $\delta_{\bm \alpha}$ is a (possibly unbounded) derivation \cite{BratteliBook1}. We will refer to $\delta_{\bm \alpha}$ as the generator of ${\bm \alpha}$. Consider now a family of time evolutions whose $\Dd$'s contain a fixed dense $\ast$-subalgebra $\Dd(\Ll) \subset {\rm CAR}(\Ll)$, which is invariant under the action of all the generators of the family. Under such conditions, these generators can be composed and they form a sub-algebra of the algebra of linear maps over $\Dd(\Ll)$, which we refer to as the {\it core} algebra associated with the family of time evolutions. In this work, we investigate the core algebra associated with a large family of time evolutions, which we claim it contains {\it all} physically reasonable time evolutions over a Delone set $\Ll$. The goal is to complete this and an associated core algebra of physical Hamiltonians to pro-$C^\ast$-algebras (in the sense of \cite{PhillipsJOT1988}), and to supply fine characterizations of these completions in a manner that unlocks the $K$-theoretic tools available for separable $C^\ast$-algebras. To our knowledge, such tools are not yet available in the context of interacting many-fermion systems.

Before elaborating any further, let us acknowledge that, for the 1-fermion sector, the program outlined above has been in place for almost three decades. Indeed, a fundamental result by Bellissard and Kellendonk \cite{Bellissard1986,Bellissard1995,Bellissard2000,
Bellissard2003,KellendonkRMP95} states that, for Galilean invariant theories, the generators of the 1-fermion dynamics over a uniformly discrete pattern form a separable groupoid $C^\ast$-algebra, canonically associated to the given pattern. By identifying this separable $C^\ast$-algebra, Bellissard and Kellendonk essentially completed the stably-homotopic classification of all the available Galilean invariant gapped 1-fermion Hamiltonians over a given pattern, which now can be enumerated via the $K$-theories of the Bellissard-Kellendonk groupoid $C^\ast$-algebra.  In \cite{BourneRMP2016,ProdanBook} and \cite{BourneAHP2020}, the reader can find models for implementing this program in the context of disordered and, respectively, generic lattices in arbitrary dimensions. At the practical level, the program pioneered by Bellissard and Kellendonk spurred new directions in materials science, resulting in a high-throughput of novel topological materials and meta-materials \cite{ApigoPRM2018,ApigoPRL2019,ChenNatComm2021,ChengPre2020,ChengArxiv2020,
LuxArxiv2021,
NiCP2019,RosaPRL2019,RosaNC2020,XiaPRL2021}, based on the sampling of different classes of Delone sets. The results in the present can facilitate similar
developments for arbitrary $N$-fermion sectors. 

We will be dealing exclusively with time evolutions that conserve the number of fermions, hence with the ones that drop to time evolutions over the ${\rm GICAR}$-subalgebra of gauge invariant physical observables. We will enforce two natural physical conditions: 1) the generators of the time evolutions depend continuously on the pattern $\Ll$ when the latter is deformed inside a fixed class of Delone sets and 2) the theories are Galilean invariant (see section~\ref{SubSec:GInvariance} and below). A large portion of our work is dedicated to formulating these two properties in a rigorous and natural framework and to identifying the most general expression of the generators that deliver such time evolutions. Prototypical examples are supplied by Hamiltonians generated from many-body potentials, such as
\begin{equation}\label{Eq:WW0}
\begin{aligned}
H_\Ll  = \sum_{k \in \NM^\times}\sum_{\{x_i,x'_i\} \in \Ll} \ w_k(x_1,\ldots,x_k;x'_1,\ldots,x'_k)  \, a^\ast_{x_1} \cdots a^\ast_{x_k} \, a_{x'_k} \cdots a_{x'_1},
\end{aligned}
\end{equation}
where the $k$-body potentials $w_k : (\RM^d)^k \times (\RM^d)^k \to \CM$ obey several constraints, such as continuity, anti-symmetry against permutations and invariance under global shifts
$$
(x_1,\ldots,x_k;x'_1,\ldots,x'_k) \mapsto (x_1-x,\ldots,x_k-x;x'_1-x,\ldots,x'_k-x).
$$ 
The $a$'s in Eq.~\eqref{Eq:WW0} are the generators of ${\rm CAR}(\Ll)$. If the $k$-body potentials have finite interaction range, {\it i.e.} $w_k$'s vanish whenever the diameter of the sets $\{x_i\} \cup \{x'_i\}$ exceed a fixed value, then $H_\Ll$ defines a generator for a dynamics, which comes in the form of a closable unbounded derivation
\begin{equation*}
{\rm ad}_{H_\Ll}(A) = \imath [A,H_\Ll], \quad A \in \Dd(\Ll) \subset {\rm CAR}(\Ll),
\end{equation*}
with a fixed dense and invariant domain (see section~\ref{Sec:FIRHam}). Note that a many-body potential defines an entire correspondence $\Ll \mapsto H_\Ll$, which is tacitly assumed to be continuous w.r.t. $\Ll$. Now, if $\Ll$ and $\Ll'$ happen to enter a relation like $\Ll' = \Ll-x$ for some $x \in \RM^d$, then
\begin{equation}\label{Eq:GInv}
{\rm Ad}_{H_{\Ll'}} = S_x \circ {\rm Ad}_{H_\Ll} \circ S_x^{-1},
\end{equation}
where $S_x : {\rm CAR}(\Ll) \to {\rm CAR}(\Ll-x)$ are the obvious $C^{*}$-algebra isomorphisms (which properly map the domains).  Eq.~\eqref{Eq:GInv} reflects the Galilean invariance of the theory, which was already encoded in the many-body potentials. 

At first sight, formalizing such Hamiltonians seem straightforward, but a fundamental difficulty presents itself. As we do not have a topology on the set of our models, the only way to define the continuity w.r.t. $\Ll$ is through the coefficients of the Hamiltonians, such as the $w_k$'s in Eq.~\eqref{Eq:WW0}. Therefore, from both the physical and mathematical points of view, identifying the correct domain of these coefficients as a topological space is paramount. Note that, besides the lattice $\Ll$, the orderings of the sets $\{x_i\}$ and $\{x'_i\}$ in Eq.~\eqref{Eq:WW0} are part of the data defining the coefficients and, for a generic Delone set, there is no canonical choice for these orderings. Furthermore, under continuous cyclic deformations of the pattern $\Ll$, the orderings of the sets $\{x_i\}$ and $\{x'_i\}$ can change and this can happen even when the deformations occur inside the transversal of a single pattern (see Definition~\ref{Def:Transversal}), as it is the case for topological lattice defects \cite{ProdanJPA2021} or for patterns like in our Example~\ref{Ex:CircleXi}. Such phenomena present a challenge because, while the Hamiltonian~\eqref{Eq:WW0} returns to itself under such cyclic deformations, the coefficients {\it do not}.

In Section~\ref{Sec:FermionsDynamics}, we construct the natural domains of the Hamiltonian coefficients as covering spaces of the topological space of all Delone sets. First, we introduce the set of pairs $(\Ll,V)$, where $\Ll$ is a Delone set and $V$ is a subset of $\Ll$ of cardinal $|V|=k$, and we topologize this set such that it becomes an (infinite) cover of the space of Delone sets. Then we introduce a second, finite cover over the topological space of pairs $(\Ll,V)$, which we call the {\it order cover}, such that the group of deck transformations of this covering space is the full permutation group $\Ss_k$ on $k$ elements. A point in this cover is a triple $\xi=(\Ll,V_\xi,\chi_\xi)$, where $\chi_\xi:\{1,\ldots,k\} \to V_\xi$ is a bijection, hence an ordering of the subset $V_\xi \subset \Ll$ (see section~\ref{Sec:Cover}). There is one such space for each $k >1$ and we call $\mathfrak a_k(\xi) :=\Ll$ the $k$-body covering map of the space of Delone sets.

A generic finite interaction range Hamiltonian evaluated at $\Ll$ then takes the form
\begin{equation}\label{Eq:H1}
H_\Ll = \sum_{k \in \NM^\times} \tfrac{1}{k!}  \sum_{\xi,\zeta \in \mathfrak a^{-1}_{k}(\Ll)}
  h_{k} (\xi,\zeta) \, \bm a^\ast(\xi) \bm a(\zeta),
\end{equation}
where
$$
\bm a(\xi) : = a_{\chi_{\xi}(k)} \cdots a_{\chi_{\xi}(1)} \in {\rm CAR}(\Ll),
$$
and the $h_{k}$ are bi-equivariant coefficients w.r.t. to the deck transformations of the order cover, invariant under the simultaneous translation of the arguments, continuous in the arguments and supported on pairs $(\xi,\zeta)$ for which the diameter of $V_\xi \cup V_\zeta$ is smaller than a fixed positive number (see subsection~\ref{Sec:AFRGI} for complete details). The derivations ${\rm ad}_{H_\Ll}$ corresponding to Hamiltonians of the type~\eqref{Eq:H1} generate the core algebra $\dot \Sigma(\Ll)$ we mentioned in the first paragraph.

Of course, the Hamiltonian coefficients can be also defined over a twofold orientation cover of the space of Delone sets. However, it turns out that the full order covers over the space of Delone sets not only supply natural domains for the Hamiltonian coefficients, but also hold the key to our generalization of the Bellissard-Kellendonk groupoid algebras to the $N$-fermion sectors.  As a topological space, for each $N$, the groupoid $\mathcal{G}_{N}$ appears naturally as a transversal to the translation action of $\mathbb{R}^{d}$ on the $N$-order cover. The ordering $\chi_{\xi}$ give us a handle to equip these transversals with the algebraic structure of a groupoid. Indeed, viewing each subset $V_\xi$ as anchored at $\chi_\xi(1)$ allows us to pull $V_\xi$ to the origin of the physical space and also gives us the means to define the set of composable pairs (see section~\ref{SubSec:AlgGN} for details). The permutation group $\mathcal{S}_{N}$ acts on $\mathcal{G}_{N}$ by bisections and this actions are non-trivial as they involve permutations {\it and} translations of the points (see sections \ref{Sec:EquiGg} and \ref{Sec:2Action}).

We now describe the process that enables us to characterize the algebra $\dot\Sigma (\Ll)$. The key observation is that these are inner-limit derivations \cite{BratteliBook1} and such derivations have a special relation with the double sided ideals: They leave them invariant and, as such, the inner-limit derivations descend to derivations on quotients of double sided ideals. The lattice of double sided ideals of ${\rm GICAR}$ algebra was worked out by Bratteli, who uncover the following structure (see section~\ref{Sec:SolvableCAR}):

\begin{theorem}[\cite{BratteliTAMS1972}]\label{Th:1} The ${\rm GICAR}(\Ll)$ algebra is a solvable $C^\ast$-algebra (of infinite length), in the sense of Dynin  \cite{DyninPNAS1978}. Specifically, it accepts a filtration by primitive ideals
\begin{equation}\label{Eq:Fil0}
   \cdots \rightarrowtail {\rm GI}_N(\Ll) \rightarrowtail \cdots \rightarrowtail {\rm GI}_1(\Ll) \rightarrowtail {\rm GI}_0(\Ll) = {\rm GICAR}(\Ll) 
 \end{equation}
 and
\begin{equation}\label{Eq:Quo0}
{\rm GI}_{N}(\Ll) / {\rm GI}_{N+1}(\Ll) \simeq \KM,
\end{equation}
for all $N \in \NM$, where $\KM$ is the algebra of compact operators.
\end{theorem}

Since the derivations descend on the quotients~\eqref{Eq:Quo0}, we obtain a sequence of representations of $\dot\Sigma(\Ll)$ by  (bounded) derivations on the algebra of compact operators. The bounded derivations over the algebra of compact operators have been already completely characterized, cf. \cite[Example~1.6.4]{BratteliBook1}, specifically, they always take the form of commutators with bounded operators over the underlying Hilbert space, which in the present context happens to be the $N$-fermion sector $\Ff_N^{(-)}(\Ll)$ of the Fock space. As a result, there are sub-algebras ${\mathfrak H}_N$ of the algebras of bounded operators over $\Ff_N^{(-)}(\Ll)$ as well as a tower of representations of $\dot \Sigma(\Ll)$ inside ${\mathfrak H}_N \otimes {\mathfrak H}_N^{\rm op}$. Furthermore, if $\eta$ is the vacuum state and $\pi_\eta$ is the associated Fock representation of ${\rm CAR}(\Ll)$, then the latter induces a representation $\dot \pi_\eta$ of $\dot \Sigma_\Ll$ on a dense linear sub-space of $\Ff_N^{(-)}(\Ll)$ and  $\dot \Sigma(\Ll) / {\rm ker} \, \dot \pi_\eta \simeq \varprojlim \, \dot\Gamma_N(\Ll)$, where $\dot \Gamma_N(\Ll) = \bigoplus_{n=0}^N \mathfrak H_N$. Note that $\dot \Sigma(\Ll) / {\rm ker} \, \dot \pi_\eta$ represents the core algebra of physical Hamiltonians (see section~\ref{SubSec:FockH} for details and justification). 

In a followup step, we compute the $C^\ast$-envelopes of $\mathfrak H_N$'s and promote the tower $\{\dot \Gamma_N(\Ll)\}$ to a projective tower of $C^\ast$-algebras $\{\Gamma_N(\Ll)\}$, whose inverse limit $\mathfrak H(\Ll)$ is the pro-$C^\ast$-algebra mentioned in the first paragraph of the section. This brings us to our first major result, stated in Theorem~\ref{Th:M1}, saying that the core algebra of physical Hamiltonians embeds into the pro-$C^\ast$-algebra $\mathfrak H(\Ll)$. A similar statement holds for the algebra $\dot \Sigma(\Ll)$ (see Theorem~\ref{Th:SigmaFinal}). Our second major result states that $\mathfrak H(\Ll)$ affords the following explicit characterization:

\begin{theorem}\label{Th:0} Let $\mathfrak p_N : {\mathfrak H}(\Ll) \twoheadrightarrow \Gamma_N(\Ll)$ be the epimorphisms associated with the inverse limit and define the closed double sided ideals $\Jj_N(\Ll) := {\rm ker}\, \mathfrak p_N$. Then $\{\Jj_N(\Ll)\}_N$ supplies a filtration of $\mathfrak H(\Ll)$ and
\begin{equation}\label{Eq:FAlg}
\Jj_{N-1}(\Ll)/\Jj_{N}(\Ll) = \pi_\Ll \Big (\Mm \big (C^{\ast}_{r,\Ss_N}(\Gg_N,\CM) \big ) \Big ).
\end{equation}
Here, $\Gg_N$ is the groupoid canonically associated to $\Ll$ generalizing the Bellissard-Kellendonk groupoid from the single to the $N$-fermion sector, $C^{\ast}_{r,\mathcal{S}_{N}}(\Gg_N,\CM)$ is the $\Ss_N$-bi-equivariant groupoid $C^\ast$-algebra, $\Mm$ indicates the essential extension to the multiplier algebra and $\pi_\Ll$ is a left regular representation of the groupoid $C^{*}$-algebra.
\end{theorem}

We will show that $\Jj_{N-1}(\Ll)/\Jj_{N}(\Ll)$ are isomorphic to the images of $\mathfrak H(\Ll)$ through its Fock representations on the $N$-fermion sectors. As such, the algebra of Galilean invariant Hamiltonians of $N$ fermions over the lattice $\Ll$ coincides with multiplier algebra $\Mm \big (C^{\ast}_{r,\Ss_N}(\Gg_N,\CM) \big )$. This reduces to the (unital) Bellissard-Kellendonk groupoid algebra for $N=1$. Our third major result states that the groupoid $\Gg_N$ continue to be second countable, Hausdorff and \'etale for $N>1$, hence its groupoid $C^\ast$-algebra is separable (see section~\ref{SubSec:AlgGN}). Furthermore, as stated already, this groupoid accepts a (highly non-trivial) 2-action of the permutation group $\mathcal{S}_{N}$ (see Definition~\ref{Def:2ActionGroupoid} and section~\ref{Sec:2Action}). 

In section~\ref{Sec:SBvsS}, we argue that the $C^\ast$-algebra of the groupoid $\Gg_N$ generates self-binding dynamics, in the sense that under the evolution generated by a Hamiltonian from $C^{\ast}_{r,\Ss_N}(\Gg_N,\CM)$, the $N$ fermions evolve in one single cluster. If the Hamiltonian comes from the corona of the multiplier extension, then the dynamics is scattered, in the sense that the $N$ fermions evolve in two or more separate clusters. When the many-body covers are disconnected, as it is the case for periodic and quasi-periodic lattices, we show in section~\ref{Sec:SpecialCases} that $C^{\ast}_{r,\mathcal{S}_{N}}(\Gg_N,\CM)$ is Morita equivalent to the 1-fermion Bellissard-Kellendonk $C^\ast$-algebra $C^{\ast}_{r}(\Gg_1,\CM)$. Thus, in these special cases, the stably-homotopic classification of the self-binding states of $N$ fermions is no more complicated than that of single fermion states. We point out that this has been already observed in numerical experiments (see \cite{LiuPRB2022}).

\vspace{0.2cm}

We follow up with pointed remarks in order to guide the reader through the significance of the above results, its relation to other works and to highlight important aspects that were left to future investigations:

\begin{remark}\label{Re:AFAlg}{\rm A good portion of our program can be repeated for AF-algebras other than the ${\rm GICAR}$ algebra, provided they have a rich lattice of ideals and the descended derivations on the quotients are by commutators with elements from their multiplier algebras. It is important to note that the pro-$C^\ast$-algebra completions generated by the mechanism explained above are not universal, in the sense that they depend in an essential way on the chosen filtration. In the present context, the filtration mentioned in Theorem~\ref{Th:1} makes the most physical sense and the associated pro-$C^\ast$-algebra gives a satisfactory framework for the dynamics of fermions. Nevertheless, it will be very interesting to investigate the range of pro-$C^\ast$-completions that can be generated from other filtrations or even from other AF-algebras, and to study the relations among these different completions. Lastly, it will be interesting to see how all this fits into or benefits from the existing frameworks on representations of $\ast$-algebras by unbounded operators \cite{MeyerDM2017,SavchukART2013}.
}$\Diamond$
\end{remark}

\begin{remark}{\rm In \cite{AndroulakisJSP2012}, the algebras ${\rm CAR}(\Ll)$ of local observables were formalized as a continuous field of $C^\ast$-algebras over the space of Delone sets. Here, we skip entirely this step and we deal directly with the algebra of derivations.
}$\Diamond$
\end{remark}

\begin{remark}\label{Re:CMor}{\rm The spaces of units of the groupoids $\Gg_N$ are only locally compact, except for $N=1$. Hence, the groupoid algebras for $N>1$ do not have a unit and the extensions to the multiplier algebras are non-trivial. Since $C^\ast_r(\Gg_N)$ does not contain the compacts, neither does its multiplier extension. As such, this extension is not the whole algebra of bounded operators over the $N$-sector of the Fock space and we expect $\Mm\big(C^\ast_r(\Gg_N)\big )$ to have a very interesting structure. A glimpse into this structure is supplied by Proposition~\ref{Pro:Mor10}, where a continuous groupoid morphism $\Gg_{N+M} \to \Gg_N \times \Gg_M$ is defined. This groupoid morphism is neither continuous nor does it preserve the Haar system. Nonetheless, its pullback map supplies a $\CM$-module map $C^\ast_r(\Gg_N) \otimes C^\ast_r(\Gg_M) \to \Mm\big (C^\ast_r(\Gg_{N+M})\big )$, which actually lands in the corona (see Remark~\ref{Re:Pullback}). We conjecture that the multiplier extensions can be fully characterized by such maps.
}$\Diamond$
\end{remark}

\begin{remark}\label{Re:BlowUp}{\rm In Example~2.37 of \cite{WilliamsBook}, Williams describes how to generate new groupoids by blowing up the unit space of an existing groupoid. As we shall see in Sec.~\ref{Sec:BlowUp1}, the groupoids $\Gg_N$, $N>1$, are all blow-ups of $\Gg_1$, which automatically implies that the associated $C^\ast$-algebras are Morita equivalent (see Th.~2.52 in \cite{WilliamsBook}). The $\Ss_N$-bi-equivariant sub-algebras are, however, {\it not} expected to be Morita equivalent in general. In fact, the 2-action of the permutation group does not admit a straightforward interpretation when the groupoids $\Gg_N$ are presented as blow-ups of $\Gg_1$. Nevertheless, as we already announced, there are special yet important cases when the $\Ss_N$-bi-equivariant sub-algebra is in fact Morita equivalent to $C^\ast_r(\Gg_1)$. As a result, for these special cases, we have precise $K$-theoretic statements.
}
\end{remark}

\begin{remark}\label{Re:TDLimit}{\rm In the light of the previous two remarks, we believe that our formalism supplies new tools and a fresh framework for the problem of the thermodynamic limit ($N\to \infty$). These aspects are put in a perspective in Section~\ref{Sec:Discussion}. Let us specify here that the formalism can be straightforwardly adapted to quasi-free states (hence situations with $N=\infty$), by using the Fock and anti-Fock primitive ideals of the ${\rm GICAR}$ algebra (see section~\ref{Sec:SolvableCAR}). The finite $N$-sectors are equally interesting, from both a theoretical and a practical point of view. For example, \cite{LiuPRB2022} supplied numerical evidence that the $K$-theories of the finite $N$-sectors can become extremely complex even for $N=3$. As explained in \cite{LiuPRB2022}, this can be exploited to generate new topological dynamics in meta-materials and new manifestations of the bulk-boundary principle.
}$\Diamond$
\end{remark}

\begin{remark}{\rm Any groupoid-invariant open subset of the unit space generates a two-sided ideal inside the groupoid $C^\ast$-algebra \cite[Sec.~5.1]{WilliamsBook}. In \cite{ProdanJPA2021}, the bulk-defect correspondence principle for the 1-fermion sector was re-formulated in terms of the exact sequences of $C^\ast$-algebras corresponding to such ideals. In the light of this re-formulation, the bulk-boundary correspondence principle can now be straightforwardly studied for any $N$-fermion sector. See also \cite{BourneAHP2020} for equivalent formulations of the bulk-boundary correspondence using groupoids and their $C^\ast$-algebras.
}$\Diamond$
\end{remark}

\begin{remark}{\rm In \cite{Upmeier1991}, Upmeier introduced the notion of a groupoid-solvable $C^\ast$-algebra (see also \cite{UpmeierBook}), which generalizes the notion of solvable $C^\ast$-algebras of Dynin. In the light of Theorem~\ref{Th:0}, we can say that $\mathfrak H$ is essentially a groupoid-solvable pro-$C^\ast$-algebra. Furthermore, the reader can find in \cite{Upmeier1991} (see also \cite{DyninPNAS1978}) models of $K$- and index-theoretic analyses in the context of groupoid-solvable $C^\ast$-algebras, specifically, for multivariable Toeplitz operators on domains of $\CM^n$. We, however, will have to leave the $K$-theoretic aspects to future investigations, although some important remarks will be made in section~\ref{Sec:SpecialCases}.
}$\Diamond$
\end{remark}

\begin{remark}\label{Re:2Act}{\rm To our knowledge, the notion of bi-equivariant groupoid $C^\ast$-algebra has not been used before. We were naturally led to it once the analysis was lifted to the order cover and the Hamiltonians were expressed via bi-equivariant coefficients as in Eq.~\eqref{Eq:H1}. The concept is developed in section~\ref{Sec:EquiGg} and is based on the notion of 2-actions of a group on groupoids and $C^\ast$-algebras. By viewing groupoids as categories, a 2-action is the natural 2-category analogue of the action of a group on a space. Historically, such structures arose under the nomer of crossed module of a pair of groups, originally introduced by Whitehead \cite{Whitehead} in the context of homotopy 2-types. The actions of crossed modules on groupoids and $C^{*}$-algebras are described in detail in \cite[Sections 2 and 3]{BussMeyerZhu}. We believe that the bi-equivariant groupoid $C^\ast$-algebra for spin-less fermions introduced here can be generalized, for example, to generate dynamical models for particles with various statistics, including non-abelian anyons. 
}$\Diamond$
\end{remark}

\begin{remark}{\rm  The CAR algebra itself accepts a presentation as a groupoid $C^\ast$-algebra (see  Example~1.10 in \cite{RenaultBook},  Example~8.3.5 in \cite{SimsSzaboWilliamsBook2020}, or our Remark~\ref{Re:CARGroupoid}). This, in fact, represented the starting point for our investigation but we soon realized that this presentation is too rigid because it assumes a fixed pre-ordering of the lattice points. That, unfortunately, is incompatible with a context where the lattice $\Ll$ can be deformed and the points can be exchanged.
}$\Diamond$
\end{remark}

\begin{remark}\label{Re:Compl}{\rm The completion of the algebra of physical Hamiltonians puts a topology on the physical models. Indeed, as for any pro-$C^\ast$-algebra, the topology is given by the semi-norms supplied by the $C^\ast$-norms of the projective tower. Let us acknowledge that the completion includes physical Hamiltonians with infinite interaction range. As we shall see in Proposition~\ref{Pro:SeedG}, the coefficients $h_k$ in Eq~\eqref{Eq:H1} define compactly supported functions over $\Gg_k$. The passage from finite to infinite interaction range happens when the algebra of functions with compact support on $\Gg_k$ is completed to the groupoid $C^\ast$-algebra $C^\ast_r(\Gg_k)$.
}$\Diamond$
\end{remark}

The paper is organized as follows. Section~\ref{Sec:Gru} introduces the key mathematical concept developed by our work, namely, the bi-equivariant groupoid $C^\ast$-algebras. This section starts with a brisk introduction to \'etale groupoids and their associated $C^\ast$-algebras and continues with a discussion of the group of bisections of a groupoid. Then it supplies our definition of a 2-action of a group on groupoids and defines the bi-equivariant groupoid $C^\ast$-algebras relative to such 2-actions of groups. The use of groupoids and their associated $C^\ast$-algebras is showcased in section~\ref{Sec:SingleFermion}, where the Bellissard-Kellendonk formalism \cite{Bellissard1986,Bellissard1995,KellendonkRMP95} for the single-fermion setting is briefly reviewed. This gives us the opportunity to introduce fundamental concepts related to the space of patterns, such as the hulls and transversals. In addition, our presentation in this section adopts a fresh point of view in which the  Bellissard-Kellendonk groupoid $C^\ast$-algebra is seen as formalizing the inner-limit derivations over a local algebra of observables, which in the 1-fermion setting is just the algebra of compact operators. Then the goal of our work can be understood as generalizing the accomplishments of Bellissard and Kellendonk to the setting where the algebra of compact operators is replaced by the GICAR algebra. Section~\ref{Sec:FermionsCAR} focuses on the ${\rm CAR}$ and ${\rm GICAR}$ algebras over Delone sets, specifically, on a symmetric presentation that avoids the use of any pre-ordering of the points and on the ideal structure, as well as on the symmetric Fock representations. Section~\ref{Sec:FermionsDynamics} analyzes the dynamics of the local degrees of freedom. The first part focuses on the physical aspects of the problem, making sure that our subsequent mathematical analysis covers all physically sound Hamiltonians. After the investigation of specific model Hamiltonians generated with many-body potentials, the section concludes that the Hamiltonian coefficients are naturally defined over many-body covers of the space of Delone sets. The deck transformations for this covers supply the 2-actions that will eventually lead us to a bi-equivariant theory as presented in section~\ref{Sec:Gru}. Section~\ref{Sec:Fermiongroupoid} develops the elements needed in Theorem~\ref{Th:0} and supplies the proof. In the process, the \'etale groupoid $\Gg_N$ associated with the dynamics of $N$-fermions is introduced and its left regular representations are shown to reproduce an essential ideal of $\Jj_{N-1}(\Ll)/\Jj_{N}(\Ll)$. Section~\ref{Sec:SpecialCases} analyses special lattices for which the many-body order covers over the transversals are disconnected. Section~\ref{Sec:Discussion} contains concluding remarks and a list of tools that we see being developed in the future within our formalism.

\section{Bi-Equivariant Groupoid $C^\ast$-Algebras}
\label{Sec:Gru}

This section develops the notion of group 2-actions on groupoids and a bi-equivariant version of the $C^\ast$-algebra associated to a groupoid in the presence of such a group 2-actions. There will be several specialized groupoids appearing in this work and, as we already stated in the introduction, our main results point to a certain groupoid structure in the dynamics of interacting fermions. As such, the reader needs a minimal familiarity with these structures and the first two subsections are intended to supply just that. In parallel, these subsections introduce notation and examples that will be used throughout.

\subsection{Background: \'Etale groupoids}

We follow here and throughout the notations and the conventions from \cite[Ch.~8]{SimsSzaboWilliamsBook2020}. We start with the algebraic structure of a groupoid. Abstractly, a groupoid is a small category in which all morphisms are invertible. Concretely, we have the following:

\begin{definition}\label{Def:groupoid}
A groupoid consists of 
\begin{itemize} 
\item A set $\Gg$ equipped with an inverse map $\alpha \mapsto \alpha^{-1}$;
\item A subset $\Gg^{(2)} \subset \Gg \times \Gg$ equipped with a multiplication map $(\alpha,\beta)\mapsto \alpha \beta \in \Gg$,
\end{itemize} 
such that
\begin{enumerate}[{\rm \ 1)}]
\item $(\alpha^{-1})^{-1} = \alpha$ for all $\alpha \in \Gg$;
\item If $(\alpha,\beta)$, $(\beta,\gamma) \in \Gg^{(2)}$, then $(\alpha \beta,\gamma)$, $(\alpha,\beta \gamma) \in \Gg^{(2)}$;
\item $(\alpha,\alpha^{-1}) \in \Gg^{(2)}$ for all $\alpha \in \Gg$;
\item For all $(\alpha,\beta) \in \Gg^{(2)}$, $ (\alpha\beta)\beta^{-1}=\alpha$ and $\alpha^{-1} (\alpha \beta) = \beta$.
\end{enumerate}
\end{definition}

The set $\Gg^{(2)}$ is said to contain the pairs of composable elements and the set 
$$
\Gg^{(0)} = \{\alpha^{-1} \alpha, \ \alpha \in \Gg\} = \{\alpha \alpha^{-1}, \ \alpha \in \Gg\}
$$ 
is referred to as the space of units. The maps
\begin{equation}\label{Eq:RS}
\mathfrak r, \ \mathfrak s : \Gg \rightarrow \Gg^{(0)}, \quad \mathfrak r(\alpha) := \alpha \alpha^{-1}, \quad \mathfrak s(\alpha) := \alpha^{-1} \alpha
\end{equation}
are called the range and the source maps, respectively. Note that $\Gg^{(2)}$ can be equivalently defined as
$$
\Gg^{(2)} = \big \{(\alpha,\beta) \in \Gg \times \Gg, \ \mathfrak s(\alpha) = \mathfrak r(\beta) \big \}.
$$

\begin{remark}{\rm It follows that 
$$
\mathfrak r(\alpha^{-1}) = \mathfrak s(\alpha), \quad \mathfrak s(\alpha^{-1}) =  \mathfrak r(\alpha).
$$
Hence, $(\alpha,\beta^{-1}) \in \Gg^{(2)}$ iff $\mathfrak s(\alpha) = \mathfrak s(\beta)$. This is relevant for the product \eqref{Eq:Conv} and for the left regular representations \eqref{Eq:PiX}. Many other useful properties and identities can be found in \cite[Ch.~8]{SimsSzaboWilliamsBook2020}.
}$\Diamond$
\end{remark}

Below, we introduce several standard groupoids, which will appear in our discussion of aperiodic physical systems.

\begin{example}\label{Ex:Matgroupoid}
{\rm For an {\it arbitrary} set $X$, one defines the matrix groupoid over $X$ to be the set $\Mm_X = X \times X$. The set $\Mm_X^{(2)} \subset \Mm_X \times \Mm_X$ of composable elements consists of pairs $((x,y),(z,w))$ such that $y=z$ and the composition of two composable elements is $(x,y) \cdot (y,z) = (x,z)$. The inversion map acts as $(x,y)^{-1}=(y,x)$. The space of units $\Mm_X^{(0)}$ consists of $(x,x)$, $x \in X$, hence it can be canonically identified with $X$. As we shall see below, if $X = \{1,2,\ldots,N\}$, this groupoid is related to the standard algebra of $N\times N$ matrices and, if $X = \NM$ or any other discrete and infinite topological space, then the groupoid relates to the algebra of compact operators over the separable Hilbert space $\ell^{2}(X)$. 
}$\Diamond$
\end{example}

\begin{example}\label{Ex:Relgroupoid}
{\rm Let $R$ be an equivalence relation over an arbitrary set $X$. The groupoid associated to $R$ consists of its graph $\Rr =\{(x,y) \in  X \times X, \ x \sim_R y\}$ and the set of composable elements $\Rr^{(2)}$ consists of pairs $((x,y),(z,w))$ such that $y=z$. The composition of two composable elements is $(x,y) \cdot (y,z) = (x,z)$. The inversion map acts as $(x,y)^{-1}=(y,x)$ and the space of units $\Rr^{(0)}$ can be identified again with $X$. The CAR-algebra admits a presentation as the $C^\ast$-algebra associated to a groupoid of this type (see Remark~\ref{Re:CARGroupoid}).
}$\Diamond$
\end{example}

\begin{example}{\rm The Bellissard-Kellendonk groupoid associated to the single fermion dynamics is introduced in Definition~\ref{Def:Groupoid1}. The groupoids associated with the $N$-fermion dynamics are introduced in Definition~\ref{Def:GroupoidN}.
}$\Diamond$
\end{example}
The groupoids we will consider carry extra structure in the form of a topology. 
\begin{definition} 
A groupoid is called a topological groupoid if $\Gg$ is equipped with a locally compact Hausdorff topology such that
\begin{enumerate}[\ \rm 1)]
\item The inversion, source and range maps are all continuous;
\item  The composition is also continuous when $\Gg^{(2)} \subset \Gg \times \Gg$ is equipped with the relative topology. 
\end{enumerate}
\end{definition}
Although the Hausdorff condition in this definition can be weakened (see \cite[Def. 8.3.1]{SimsSzaboWilliamsBook2020}), the groupoids in this paper will always satisfy it.
\begin{example}{\rm \cite[p.~73]{SimsSzaboWilliamsBook2020} Every groupoid is a topological groupoid in the discrete topology. In particular, if the set $X$ is a discrete topological space, then the groupoids constructed in examples~\ref{Ex:Matgroupoid} and \ref{Ex:Relgroupoid} are topological.
}$\Diamond$
\end{example}

\begin{example}{\rm \cite[p.~73]{SimsSzaboWilliamsBook2020} If $X$ in example~\ref{Ex:Relgroupoid} is a second-countable Hausdorff space, then $\Rr$ is a topological groupoid in the relative topology inherited from $X \times X$. The topology of $\Rr$, however, can be finer than the relative one, but never coarser. This is a very important observation because, as we shall see, the compact subsets of the groupoid enter in an essential way in the definition of the groupoid algebra. As such, by adjusting the groupoid's topology, one can drastically alter the character of this algebra (see Remark~\ref{Re:CompactSupp}). This issue will be highlighted at several key points of our narrative (see for example Remark~\ref{Re:Xi1}).
}$\Diamond$
\end{example} 

In this work, we will only consider \'etale groupoids:
\begin{definition}[\cite{SimsSzaboWilliamsBook2020},~Def.~8.4.1] A topological groupoid $\Gg$ is called \'etale if the range map $\mathfrak r : \Gg \rightarrow \Gg^{(0)}$ is a local homeomorphism.
\end{definition}

\begin{remark}{\rm Among all topological groupoids, the \'etale groupoids are the equivalent of the discrete groups within the topological groups. Indeed, an important characteristic of \'etale groupoids is that $\mathfrak r^{-1}(\alpha)$ and $\mathfrak s^{-1}(\alpha)$ are discrete topological spaces for any $\alpha \in \Gg^{(0)}$. Furthermore, any \'etale groupoid admits a Haar system \cite[Def. 2.2]{RenaultBook}, which is supplied by the system of counting measures on each fiber \cite[Def. 2.6, Lemma 2.7]{RenaultBook}. As we will be working solely with \'etale groupoids, we will not need the general theory of Haar systems and we will therefore omit a discussion of them.
}$\Diamond$
\end{remark}

\subsection{Background: Reduced $C^\ast$-algebra of an \'etale groupoid}
We develop this section in a generality that is beyond what is needed for our present work. We do so to facilitate the new concept of bi-equivariant groupoid $C^\ast$-algebras, which will play a central role in our present and possibly future program (per Remark~\ref{Re:2Act}).

Given an \'etale groupoid $\Gg$ and a $C^{*}$-algebra $\Aa$, one defines the space $C_c(\Gg,\Aa)$ of compactly supported $\Aa$-valued continuous functions on $\Gg$ and endows it with the multiplication
\begin{equation}\label{Eq:Conv}
(f_1 \ast f_2)(\alpha) = \sum_{\beta \in \mathfrak s^{-1}( \mathfrak s(\alpha))} f_1(\alpha \beta^{-1}) f_2(\beta)=\sum_{\beta \in \mathfrak r^{-1}(\mathfrak{r}(\alpha))}f_{1}(\beta)f_{2}(\beta^{-1}\alpha), 
\end{equation}
and the $\ast$-operation
\begin{equation}\label{Eq:Inv}
f^\ast(\alpha) = f(\alpha^{-1})^{*}.
\end{equation}
Then  $C_c(\Gg,\Aa)$ becomes a $\ast$-algebra, which can be equipped with a norm such that its completion supplies a $C^\ast$-algebra. The physical implication of this process was already addressed in Remark~\ref{Re:Compl}.

\begin{remark}\label{Re:CompactSupp}{\rm Note that the topology of $\Gg$, specifically the compact subsets of $\Gg$ in this topology, determines which functions are contained in this $C^\ast$-algebra. This somewhat technical point has important physical implications and will become a central point of discussion in the next sections.
}$\Diamond$
\end{remark}

We now return to the task of defining a $C^\ast$-norm and for this we follow the procedure from~\cite{KhoshkamJRAM2002}, where $C_c(\Gg,\Aa)$ is embedded and then completed in the $C^\ast$-algebra of adjointable operators over a natural Hilbert $C^\ast$-module associated with the pair $(\Gg,\Aa)$. We first recall:

\begin{definition}[\cite{BlackadarBook1998,JensenThomsenBook1991,WeggeOlsenBook1993de}] 
A right Hilbert C$^\ast$-module over the $C^\ast$-algebra $\mathcal B$ is a right $\mathcal B$-module equipped with an inner product
$$
\langle \,\cdot \,,\, \cdot\, \rangle \;:\; E_\mathcal B \times E_\mathcal B \rightarrow \mathcal B
\;,
$$
which is linear in the second variable and satisfies the following relations for all $\psi, \psi' \in E_\Bb$ and $b \in \Bb$:
\begin{enumerate}[{\rm \ \ (i)}]

\item $\langle \psi , \psi' b \rangle = \langle \psi , \psi' \rangle b$,

\item $\langle \psi ,\psi' \rangle^\ast = \langle \psi' ,\psi \rangle$,

\item $\langle \psi , \psi \rangle \geq 0$,

\item $\psi \neq 0$ implies $\langle \psi , \psi \rangle \neq 0$.

\end{enumerate}
In addition, $E_\mathcal B$ must be complete in the norm induced by the inner product
$$
\|\psi \| 
\; :=\; 
\| \langle \psi,\psi \rangle \|^\frac{1}{2}_{\mathcal{B}}
\;, 
\qquad \psi \in E_\Bb\;,
$$
where $\| \cdot \|_\Bb$ is the $C^\ast$-norm of $\Bb$.
\end{definition}

\begin{example}\label{Ex-StandardHM1} 
{\rm The standard Hilbert $\mathcal B$-module is defined to be the set
$$
\ell^2(\NM,\mathcal B)
\;=\; 
\Big\{
\{b_n\}_{n\in\NM}\ : \  b_n \in \mathcal B, \ \sum_{n=1}^N b_n^\ast b_n \ \mbox{converges in} \ \mathcal B \ {\rm as} \ N\rightarrow \infty
\Big\}
\;,
$$
endowed with the obvious addition and the $\Bb$-valued inner product
$$
\big\langle \{b_n\}_{n\in\NM}  ,  \{b'_n\}_{n\in\NM}  \big \rangle 
\;=\; \sum_{n \in \mathbb N} b_n^\ast b'_n.
$$
The set $\NM$ can be replaced by any other countable set $X$, in which case we write $\ell^2(X,\mathcal B)$.
}$\Diamond$
\end{example}

\begin{definition}[\cite{BlackadarBook1998,JensenThomsenBook1991,WeggeOlsenBook1993de}]  The space of adjointable operators ${\rm End}^\ast(E_\mathcal B)$ over a Hilbert $\mathcal B$-module consists of the maps $T: E_\mathcal B \rightarrow E_\mathcal B$ for which there exists a linear map $T^\ast : E_\mathcal B \rightarrow E_\mathcal B$ such that

$$
\langle T \psi,\psi' \rangle 
\;=\; 
\langle \psi, T^\ast \psi' \rangle
\;, 
\qquad \psi,\psi' \in E_\mathcal B
\;.
$$

\end{definition}

\begin{remark} 
{\rm 
An adjointable operator is automatically a bounded $\mathcal B$-module map, but the reversed implication is not always true. Hence, the attribute ``adjointable'' is significant. If the operator $T^\ast$ exists, it is unique.$\Diamond$
}
\end{remark}

\begin{proposition}[\cite{JensenThomsenBook1991} p.~4, and \cite{BlackadarBook1998,WeggeOlsenBook1993de}] When endowed with the operator norm,
$$
\|T \| \;=\; \sup \big \{\|T \psi\| \ : \ \psi \in E_\mathcal B, \ \|\psi\| \leq 1 \big \}
\;,
$$
the space ${\rm End}^\ast(E_\Bb)$ becomes a unital C$^\ast$-algebra.
\end{proposition}

We now describe the Hilbert $C^\ast$-module associated to the pair $(\Gg,\Aa)$. First, note that the space $C_c(\Gg,\Aa)$ is already a right module over $C_0(\Gg^{(0)},\Aa)$, the $C^\ast$-algebra of $\Aa$-valued continuous and compactly supported functions over $\Gg^{(0)}$, endowed with the sup norm. Indeed, for $f \in C_c(\Gg,\Aa)$ and $\chi \in C_0(\Gg^{(0)},\Aa)$, one defines
\begin{equation*}
(f\cdot \chi)(\alpha) = f(\alpha) \chi(\mathfrak s(\alpha)), \quad \alpha \in \Gg.
\end{equation*}
Since $\Gg^{(0)} \subset \Gg$ is closed, one can also consider the restriction map \[\rho :C_c(\Gg,\Aa) \rightarrow C_0(\Gg^{(0)},\Aa),\] and define a $C_0(\Gg^{(0)},\Aa)$-valued inner product on the right $C_0(\Gg^{(0)},\Aa)$-module $C_c(\Gg,\Aa)$ via
\begin{equation}\label{Eq:InnerProd1}
\langle f_1|f_2 \rangle_{C_0(\Gg^{(0)},\Aa)}(x) :=\rho(f_1^\ast \ast f_2)(x) = \sum_{\alpha \in \mathfrak s^{-1}(x)} f_1(\alpha)^{*} f_2(\alpha).
\end{equation}
We denote by $E_{C_0(\Gg^{(0)},\Aa)}$ the Hilbert $C^\ast$-module completion of $C_c(\Gg,\Aa)$ in this inner product. The convolution from the left supplies a left action of the $\ast$-algebra $C_c(\Gg,\Aa)$ on $E_{C_0(\Gg^{(0)},\Aa)}$ by bounded adjointable endomorphisms, extending the action of $C_c(\Gg,\Aa)$ on itself.
\begin{definition}[\cite{KhoshkamJRAM2002}]
The reduced groupoid $C^\ast$-algebra of $\mathcal{G}$ over the $C^\ast$-algebra $\Aa$ is the completion of $C_c(\Gg,\Aa)$ in the norm inherited from the embedding 
$$C_c(\Gg,\Aa) \rightarrowtail {\rm End}^\ast(E_{C_0(\Gg^{(0)},\Aa)})$$ 
and is denoted $C^\ast_r(\Gg,\Aa)$.
\end{definition}
In case $\Aa=\mathbb{C}$ it is usually suppressed in the notation. We refer to $C^{*}_{r}(\mathcal{G}):=C^{*}_{r}(\mathcal{G},\mathbb{C})$ as the reduced $C^{*}$-algebra of $\mathcal{G}$.

\begin{remark}{\rm For a second countable, Hausdorff, \'etale groupoid $\Gg$, the reduced $C^\ast$-algebra $C^\ast_r(\Gg)$ is separable. In turn, this implies that its $K$-groups are countable \cite{BlackadarBook1998} and this is of fundamental importance for various classifications programs of topological phases of matter, as we already stressed in the introduction. 
}$\Diamond$
\end{remark}
 
\subsection{Bi-equivariant groupoid $C^{*}$-algebras}
\label{Sec:EquiGg}

Let $\mathcal{G}$ be a topological groupoid. The group of bisections $\mathcal{S}(\mathcal{G})$ of $\mathcal{G}$ is the space of continuous maps
\begin{equation*}
\mathcal{S}(\mathcal{G}):=\left\{b:\mathcal{G}^{(0)}\to \mathcal{G}: \mathfrak s\circ b=\textnormal{Id},\quad \mathfrak r\circ b \textnormal{ is a homeomorphism}\right\}.
\end{equation*}
The group structure on $\mathcal{S}(\mathcal{G})$ is given by 
\begin{equation}\label{Eq:Rule1}
b_{1}\cdot b_{2}(\alpha):=b_{1}\big(\mathfrak r\circ b_{2}(\alpha)\big) b_{2}(\alpha),\quad b^{-1}(\alpha):=b\big((\mathfrak r\circ b)^{-1}(\alpha)\big)^{-1},\quad\alpha\in\mathcal{G}^{(0)}.
\end{equation}
Here, $b^{-1}$ is the inverse of $b$ in $\Ss(\Gg)$, $(\mathfrak r\circ b)^{-1}$ denotes the inverse homeomorphism to $\mathfrak r\circ b$, whereas $b(\alpha)^{-1}$ denotes the inverse of $b(\alpha)$ in $\mathcal{G}$. The identity element of $\mathcal{S}(\mathcal{G})$ is the inclusion $i:\mathcal{G}^{(0)}\to \mathcal{G}$. In fact, $\mathcal{S}(\mathcal{G})$ is a locally compact group in the compact open topology \cite[Example 13]{BussMeyerZhu}. It admits continuous commuting left and right actions on $\mathcal{G}$ via
\[b_{1}\cdot\alpha\cdot b_{2}:=b_{1}\big(\mathfrak r(\alpha)\big) \, \alpha \, b_{2}^{-1}\big(\mathfrak s(\alpha)\big)^{-1} ,\quad b\in\mathcal{S}(\mathcal{G}),\quad \alpha\in\mathcal{G}.\]

\begin{lemma} The actions satisfy
\begin{align}\label{eq: actioninversion}(b_{1}\cdot\alpha\cdot b_{2})^{-1} &=b_{2}^{-1}\cdot \alpha^{-1}\cdot b_{1}^{-1},\\
\label{eq: actionrange}\mathfrak r(b_{1}\cdot\alpha\cdot b_{2}) &=b_{1}\cdot \mathfrak r(\alpha)\cdot b_{1}^{-1},\\
\label{eq:actionsource}\mathfrak s(b_{1}\cdot\alpha\cdot b_{2})&=b_{2}^{-1}\cdot \mathfrak s(\alpha)\cdot b_{2}.\end{align}
\end{lemma}

\proof Indeed:
\begin{align*}(b_{1}\cdot\alpha\cdot b_{2})^{-1}&=b_{2}^{-1}(\mathfrak s(\alpha))\, \alpha^{-1}\, b_{1}(\mathfrak r(\alpha))^{-1}=b_{2}^{-1}\cdot \alpha^{-1}\cdot b_{1}^{-1}\\
\mathfrak r(b_{1}\cdot\alpha\cdot b_{2})&=\mathfrak r\circ b_{1}(\mathfrak r(\alpha))=b_{1}(\mathfrak r(\alpha))\, \mathfrak r(\alpha)\, b_{1}(\mathfrak r(\alpha))^{-1}=b_{1}\cdot \mathfrak r(\alpha)\cdot b_{1}^{-1}\\
\mathfrak s(b_{1}\cdot\alpha\cdot b_{2})&=\mathfrak s\circ b_{2}^{-1}(\mathfrak s(\alpha))^{-1}=b_{2}^{-1}(\mathfrak s(\alpha))\, \mathfrak s(\alpha)\, b_{2}^{-1}(\mathfrak s(\alpha))^{-1}=b_{2}^{-1}\cdot \mathfrak s(\alpha)\cdot b_{2},
\end{align*}
which prove the statements.\qed

\begin{corollary} If $(\alpha,\beta)\in\mathcal{G}^{(2)}$, then $(b_{1}\cdot \alpha \cdot b_{2}^{-1}, b_{2}\cdot \beta \cdot b_{3})\in\mathcal{G}^{(2)}$ and
$$
(b_{1}\cdot \alpha \cdot b_{2}^{-1}) (b_{2}\cdot \beta \cdot b_{3})=b_{1}\cdot \alpha\beta\cdot b_{3}.
$$
\end{corollary}

\begin{definition}\label{Def:2ActionGroupoid} 
Let $\Gamma$ be a discrete group and $\mathcal{G}$ a locally compact Hausdorff groupoid. A 2-action of $\Gamma$ on $\mathcal{G}$ is a group homomorphism $\tau:\Gamma\to \mathcal{S}(\mathcal{G})$.
\end{definition}  

Given a 2-action of $\Gamma$ on $\mathcal{G}$ we obtain commuting left and right actions by setting
\begin{equation}\label{Eq:2Action1}
\gamma_{1}\cdot\alpha\cdot\gamma_{2}:=\tau_{\gamma_{1}}\cdot\alpha \cdot \tau_{\gamma_{2}}=\tau_{\gamma_{1}}\big(\mathfrak r(\alpha)\big)\cdot \alpha \cdot\tau_{\gamma_{2}^{-1}}\big(\mathfrak s(\alpha)\big)^{-1}.
\end{equation}
It should be noted that these actions {\it do not} determine the full 2-action of the group. 

\begin{remark}{\rm For the specialized groupoids appearing in our work, a 2-action is introduced in Proposition~\ref{Pro:2Action} and the corresponding left and right actions are computed in Proposition~\ref{Pro:LRActions}. 
}$\Diamond$
\end{remark}

For a $C^{*}$-algebra $\Aa$, we denote by $UM(\Aa)$ the unitary group of the multiplier algebra of $\Aa$. A similar notion of 2-action of a group on $C^{*}$-algebras is given in the following definition.
\begin{definition} \label{2actionalgebra}
Let $\Gamma$ be a discrete group and $\Aa$ a $C^{*}$-algebra. A 2-action of $\Gamma$ on $\Aa$ is a group homomorphism $u:\Gamma \to UM(\Aa)$ from $\Gamma$ to the unitary multipliers on $\Aa$. 
\end{definition}

Given a 2-action of $\Gamma$ on $\Aa$ we obtain left and right actions on $\Aa$ by setting
\[\gamma_{1}\cdot a \cdot \gamma_{2}:=u_{\gamma_{1}}au_{\gamma_{2}}^{*}.\]
This in turn induces an action \[\bm \alpha:\Gamma\to \textnormal{Aut}(\Aa),\quad \bm \alpha_{\gamma}(a):=u_{\gamma}a u_{\gamma}^{*},\]
of $\Gamma$ on $\Aa$ by $*$-automorphisms, making $A$ into a $\Gamma$-$C^{*}$-algebra. Note that in case $\Aa$ is commutative, the latter action is trivial but the 2-action is not.

\begin{remark}{\rm In our concrete application, $\Gamma$ is the permutation group $\mathcal{S}_{N}$ and the $C^\ast$-algebra $\Aa$ will be simply $\CM$. The relevant 2-action is given by the sign homomorphism $\tau:\mathcal{S}_{N}\to \{\pm 1\}\subset \mathbb{C}=UM(\mathbb{C})$. We often write this action as $\Ss_N \ni s \mapsto (-1)^s \in UM(\CM)$. 
}$\Diamond$
\end{remark}
 
\begin{remark}{\rm We chose the terminology of 2-actions because,
given a group $\Gamma$, the pair groupoid $\Gamma\times\Gamma$ is an example of a $2$-groupoid (for a comprehensive definition, see \cite[Section 2.2]{Amini}). A 2-groupoid is a 2-category in which all 1- and 2-morphisms are invertible. Somewhat informally, a 2-groupoid is a groupoid such that the unit space carries the structure of a groupoid itself. In case the unit space is a group, we speak of a 2-group. 
}$\Diamond$
\end{remark}

\begin{remark}{\rm For us, the relevant example arises from a discrete group $\Gamma$. The pair groupoid $\Gamma\times \Gamma$ associated to the full equivalence relation on $\Gamma$ as in Example \ref{Ex:Relgroupoid}, carries a natural 2-group structure, since $(\Gamma\times\Gamma)^{(0)}=\Gamma$ is a group.  The notion of 2-action in Definition \ref{Def:2ActionGroupoid} is equivalent to an action of the 2-groupoid $\Gamma\times \Gamma$ on the 1-groupoid $\mathcal{G}$, and the data in Definition \ref{2actionalgebra} is equivalent to an action of the 2-groupoid $\Gamma\times\Gamma$ on the $C^{*}$-algebra $\Aa$. 
}$\Diamond$
\end{remark}

\begin{remark}{\rm Alternatively, a 2-group can be associated to a so called crossed module of a pair of groups, originally introduced by Whitehead \cite{Whitehead} in the context of homotopy 2-types. In our case, the relevant crossed module is given by the group $\Gamma$ acting on itself by conjugation. The equivalence of 2-groups and crossed modules is described in \cite[Section 3]{Noohi}. The actions of crossed modules on groupoids and $C^{*}$-algebras are described in detail in \cite[Sections 2 and 3]{BussMeyerZhu}. Applying those definitions to the special case of the 2-group $\Gamma\times\Gamma$ gives our definitions.
}$\Diamond$
\end{remark}

Given 2-actions of $\Gamma$ on a $C^{*}$-algebra $\Aa$ and on a groupoid $\mathcal{G}$, 
 we write
 \[ C_{c,\Gamma}(\mathcal{G},\Aa):=\Big\{f\in C_{c}(\mathcal{G},\Aa): f(\gamma_{1}\cdot\alpha\cdot\gamma_{2})=\gamma_{1}\cdot f(\alpha)\cdot\gamma_{2}\Big\}\subset C_{c}(\mathcal{G},\Aa), \]
 for the space of $\Gamma$-bi-equivariant maps from $\mathcal{G}$ to $\Aa$. The space $C_{c,\Gamma}(\mathcal{G},\Aa)$ carries the structure of a $*$-algebra:
\begin{lemma}
The subspace $C_{c,\Gamma}(\mathcal{G},\Aa)$ is a $*$-subalgebra of $C_{c}(\mathcal{G},\Aa)$.
\end{lemma} 
\begin{proof}
The identities \eqref{eq: actioninversion}, \eqref{eq: actionrange} and \eqref{eq:actionsource} allow us to show that convolution and involution are compatible with the $\Gamma$-equivariance:
\begin{align*}
f*g(\gamma_{1}\cdot \beta\cdot \gamma_{2})&=\sum_{\alpha\in\mathfrak{r}^{-1}(\mathfrak{r}(\gamma_{1}\cdot\beta\cdot \gamma_{2}))}f(\alpha)g\big (\alpha^{-1}(\gamma_{1}\cdot\beta\cdot\gamma_{2})\big )\\
&=\sum_{\alpha\in\mathfrak{r}^{-1}(\gamma_{1}\mathfrak{r}(\beta)\gamma_{1}^{-1})} f(\alpha)g\Big(\gamma_{1}\cdot \big((\gamma_{1}^{-1}\cdot \alpha^{-1}\cdot\gamma_{1})\beta\big)\cdot\gamma_{2}\Big)\\
&=\sum_{\alpha\in\mathfrak{r}^{-1}(\mathfrak{r}(\beta))}  f(\gamma_{1}\cdot\alpha\cdot\gamma_{1}^{-1})g(\gamma_{1}\cdot (\alpha^{-1}\beta)\cdot \gamma_{2})\\
&=\sum_{\alpha\in\mathfrak{r}^{-1}(\mathfrak{r}(\beta))}  \gamma_{1}\cdot f(\alpha)\cdot\gamma_{1}^{-1}\cdot\gamma_{1}\cdot g(\alpha^{-1}\beta)\cdot \gamma_{2}\\
&=\sum_{\alpha\in\mathfrak{r}^{-1}(\mathfrak{r}(\beta))}  \gamma_{1}\cdot f(\alpha) g(\alpha^{-1}\beta)\cdot \gamma_{2}=\gamma_{1}\cdot (f * g)(\beta)\cdot \gamma_{2}
\end{align*}
\begin{align*}
f^{*}(\gamma_{1}\cdot \alpha\cdot\gamma_{2})=f((\gamma_{1}\cdot \alpha\cdot\gamma_{2})^{-1})^{*}=f(\gamma_{2}^{-1}\cdot \alpha^{-1}\cdot\gamma_{1}^{-1})^{*}=\gamma_{1}\cdot f^{*}(\alpha)\cdot\gamma_{2}.
\end{align*}
This proves the lemma.
\end{proof}

\begin{definition} Let $\Gamma$ be a discrete group. Suppose that $\mathcal{G}$ is a locally compact Hausdorff \'etale groupoid and $\Aa$ is a $C^{*}$-algebra, both of which carry a 2-action by $\Gamma$.
The $\Gamma$-bi-equivariant reduced $\Aa$-$C^{*}$-algebra $C^{*}_{r,\Gamma}(\mathcal{G},\Aa)$ of $\mathcal{G}$ is the closure of the image of $C_{c,\Gamma}(\mathcal{G},\Aa)$ inside $C^{*}_{r}(\mathcal{G},\Aa)$.
\end{definition}

The following statement supplies a computational tool for projecting onto a $\Gamma$-bi-equivariant sub-algebra when the group is finite. It can be used, for example, to project the range of the $\CM$-module morphisms mentioned in Remark~\ref{Re:CMor}. 

\begin{proposition} Let $\Gamma$ be a finite group and $C^{*}_{r,\Gamma}(\mathcal{G},\Aa)$ a $\Gamma$-bi-equivariant reduced $\Aa$-$C^{*}$-algebra of $\mathcal{G}$. Then the map $E: C^{*}_{r}(\mathcal{G},\Aa) \to C^{*}_{r}(\mathcal{G},\Aa)$,
\begin{equation}\label{Eq:Exp}
\big(E(f)\big)(\alpha) = \frac{1}{|\Gamma|^2} \sum_{\gamma_1,\gamma_2 \in \Gamma} \gamma_1^{-1} \cdot f(\gamma_1 \cdot \alpha \cdot \gamma_2) \cdot \gamma_2^{-1}
\end{equation}
is a conditonal expectation onto $C^{*}_{r,\Gamma}(\mathcal{G},\Aa)$.
\end{proposition}

\proof The map $E$ is a projection. Indeed, for $\theta_1,\theta_2 \in \Gamma$,
$$
\begin{aligned}
\big(E(f)\big)(\theta_1 \cdot \alpha \cdot \theta_2) = \frac{1}{|\Gamma|^2} \sum_{\gamma_1,\gamma_2 \in \Gamma} \theta_1 (\gamma_1 \theta_1)^{-1} \cdot f(\gamma_1 \theta_1 \cdot \alpha \cdot \theta_2 \gamma_2) \cdot (\theta_2 \gamma_2)^{-1} \theta_2,
\end{aligned}
$$
hence $E(f)$ is bi-equivariant. As such, the range of $E$ is contained in $C^{*}_{r,\Gamma}(\mathcal{G},\Aa)$. Furthermore, for $f \in C^{*}_{r,\Gamma}(\mathcal{G},\Aa)$, we can see directly from Eq.~\eqref{Eq:Exp} that $E(f) = f$. Lastly, $E$ is a contractive projection, hence an expectation \cite{TomiyamaPJA1958}.\qed

\subsection{Left regular representations}
\label{Sec:LReps}

After completing the algebra $C_c(\Gg,\Aa)$, the restriction map $\rho :C_c(\Gg,\Aa) \rightarrow C_0(\Gg^{(0)},\Aa)$ extends by continuity to a positive map of $C^\ast$-algebras. By composing $\rho$ with the evaluation maps on $C_0(\Gg^{(0)},\Aa)$, $j_x(\chi) = \chi(x)$ for some $x \in \Gg^{(0)}$, one obtains a family of $\Aa$-valued positive maps $\rho_x = j_x \circ \rho$ on $C^\ast_r(\Gg,\Aa) $, indexed by the space of units $\Gg^{(0)}$. Then, for each $x \in \Gg^{(0)}$, $C^\ast_r(\Gg,\Aa) $ can be endowed with an $\Aa$-valued inner product $\langle f_1,f_2 \rangle_x :=\rho_x(f_1^\ast f_2)$, which makes $C^\ast_r(\Gg,\Aa) $ into a right pre-Hilbert $\Aa$-module. The completion of the latter can be canonically identified with the standard Hilbert module $\ell^2\big(\mathfrak s^{-1}(x),\Aa\big)$. The action from the left of the algebra $C^\ast_r(\Gg,\Aa) $ on itself can be extended to an action by adjointable operators on $\ell^2(\mathfrak s^{-1}(x),\Aa)$. Explicitly,
\begin{equation}\label{Eq:PiX}
[\pi_x(f) \psi](\alpha) = \sum_{\beta \in \mathfrak s^{-1}(x)} f(\alpha\beta^{-1})\psi(\beta), \quad \alpha \in \mathfrak s^{-1}(x), \quad \psi \in \ell^2\big(\mathfrak s^{-1}(x),\Aa\big).
\end{equation}
The representations $\pi_x$, $x \in \Gg^{(0)}$, are called the left regular representations of the reduced groupoid $C^{*}$-algebra $C^\ast_r(\Gg,\Aa) $.

\begin{example}\label{Ex:MatgroupoidRep}{\rm For the matrix groupoid from example~\ref{Ex:Matgroupoid}, $$\mathfrak s^{-1}(x) = \{(y,x), \ y \in X\},$$ and these sets can all be canonically identified with $X$. If $X$ is a discrete topological space, then $\Mm_X$ is \'etale and the regular representations of $C^\ast(\Mm_x)$ are on $\ell^2(X)$ and given by
\begin{equation*}
[\pi_{x}(f) \psi](y) =  \sum_{z \in X} f(y,z) \psi(z).
\end{equation*}
These are all compact operators over $\ell^2(X)$.
}$\Diamond$
\end{example} 

\begin{example}\label{Ex:RelgroupoidRep}{\rm For the groupoid from example~\ref{Ex:Relgroupoid}, 
\begin{equation*}
\mathfrak s^{-1}(x) = \{(y,x) \in X \times X, \ y \sim_R x\}.
\end{equation*} 
If $\Rr$ is \'etale, then the regular representations of $C^\ast(\Rr)$ are on $\ell^2(\mathfrak s^{-1}(x))$ and given by
\begin{equation*}
[\pi_{x}(f) \psi](y) =  \sum_{z \sim_R x} f(y,z) \psi(z).
\end{equation*}
Depending on the topology of $\Rr$, these may or may not be compact operators over $\ell^2(\mathfrak s^{-1}(x))$. An example of the latter is the CAR groupoid, for which all the left regular representations are carried by non-compact operators.
}$\Diamond$
\end{example} 

\section{Single Fermion Dynamics by  groupoid Methods}
\label{Sec:SingleFermion}

Our main goals for this section are to highlight the important role played by groupoid $C^{*}$-algebras in the analysis of the dynamics of a single fermion hopping over an aperiodic lattice and to formulate the problem from a perspective that affords, at the conceptual level, a straightforward generalization to the case of interacting fermions. For this, we start from an algebra of local physical observables available to an experimenter. As we shall see, the time evolutions of these observables are implemented by groups of outer authomorphisms. Then the specialized groupoid algebras constructed by Bellissard \cite{Bellissard1986,Bellissard1995} and Kellendonk \cite{KellendonkRMP95} can be seen as formalizing the generators of these groups of outer automorphisms.

\subsection{The algebra of local observables}
\label{SubSec:LocalAlg0}

The setting is that of a network $\Ll$ of identical quantum resonators, such as atoms or quantum circuits, rendered in the physical space $\RM^d$ and whose identical internal structures have been fixed. When placed next to each other, the resonators couple via potentials that are entirely determined by their internal structure and space arrangement. As such, under the stated constraints, the dynamics of the collective degrees of freedom is {\it fully} determined by the pattern $\Ll \subset \RM^d$. In Fig.~\ref{Fig:LocalAlg1}, we show such a network of quantum resonators, together with an experimenter who excites the resonators, perhaps with a laser beam, and then observes the dynamics of the coupled degrees of freedom. The experimenter can only excite resonators and take measurements on resonators near its position. In other words, the experimenter can only probe the local degrees of freedom. Of course, we are referring to the idealized case of an infinite sample.

\begin{figure}[t]
\center
\includegraphics[width=\textwidth]{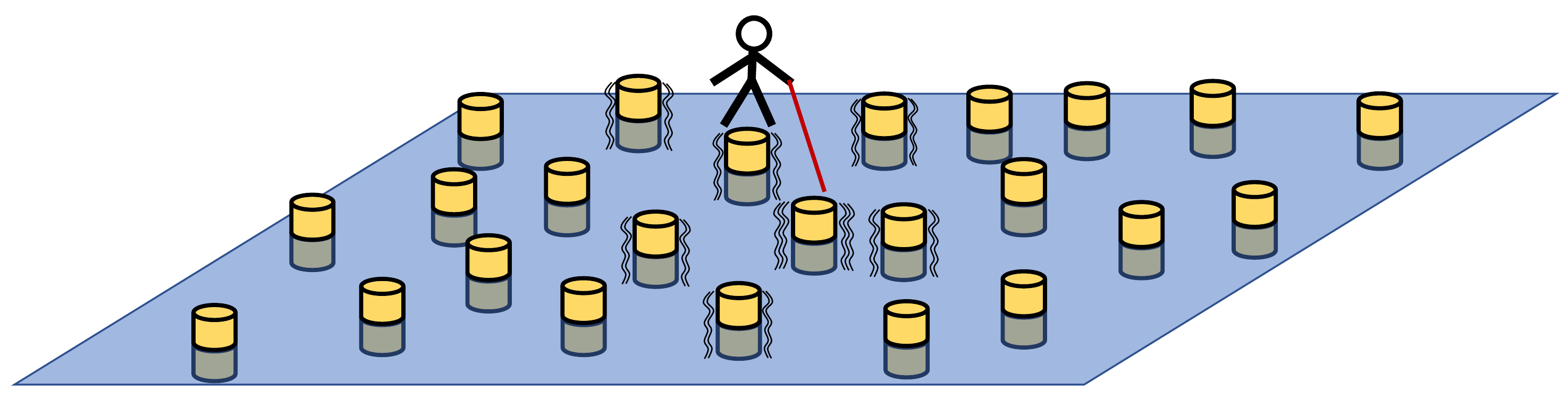}\\
  \caption{\small Schematic of a local observation.
}
 \label{Fig:LocalAlg1}
\end{figure}

In the 1-fermion setting, the algebra of observables available to the observer is that of finite-rank operators over the Hilbert space $\ell^2(\Ll)$. If the resonators have more than one internal degree of freedom, this Hilbert space is tensored by $\CM^N$ but this detail can be dealt with later. By a limiting procedure, one can close this algebra to the $C^\ast$-algebra $\KM\big (\ell^2(\Ll)\big )$  of compact operators over $\ell^2(\Ll)$, which represents the algebra of local observables in the 1-fermion setting. The main interest, however, is not in this algebra but in the time evolution of the observables, which is formalized by a one parameter group of $*$-automorphisms
\begin{equation}
\label{eq: R-action}
\bm \alpha : \RM \to {\rm Aut}\Big(\KM\big (\ell^2(\Ll) \big) \Big ).
\end{equation} 
We assume from the beginning that the time evolution is point-wise continuous, {\it i.e.} the map $t \mapsto \bm \alpha_t(K) \in \KM\big (\ell^2(\Ll) \big)$ is continuous for all $K \in \KM\big (\ell^2(\Ll) \big)$. The generator $\delta_{\bm \alpha}$ of the dynamics is defined on the dense $*$-sub-algebra of elements $K\in\KM\big (\ell^2(\Ll) \big)$ for which $\bm \alpha_t(K)$ is first order differentiable in $t$. The derivations over the algebra of compact operators have been fully characterized in \cite[Example~1.6.4]{BratteliBook1}. In particular, all {\it bounded} derivations come in the form
\begin{equation*}
{\rm ad}_{H} (K) = \imath [K,H], \quad (\imath = \sqrt{-1}),
\end{equation*}
for some $H$ from the algebra $\BM\big ( \ell^2(\Ll)\big )$ of bounded operators over $\ell^2(\Ll)$, which will be referred to as the Hamiltonian.

\subsection{Galilean invariant theories}
\label{Sec:GalInvHam1}

As a bounded operator over $\BM\big ( \ell^2(\Ll)\big )$, the Hamiltonian always affords a presentation of the form
\begin{equation}\label{Eq:DynMat}
H_\Ll = \sum_{x,x' \in \Ll} w_{x,x'}(\Ll) \, |x\rangle \langle  x' |, \quad w_{x,x'}(\Ll) \in \CM,
\end{equation}
with the sum converging in the strong operator topology. The notation in Eq.~\ref{Eq:DynMat} emphasizes that the Hamiltonian coefficients $w_{x,x'}(\Ll)$ are fully determined by the pattern itself. The experimenter can modify the pattern and map again the coefficients, hence, at least in principle, the entire map $\Ll \mapsto \{w_{x,x'}(\Ll)\}$ can be accessed experimentally. Physical considerations enable us to assume that the coefficients vary continuously with $\Ll$, in a sense that will be made precise later.

The Hamiltonians display additional structure steaming from the Galilean invariance of the non-relativistic physical laws. Indeed, in the absence of any background field or potential, suppose that two patterns $\Ll$ and $\Ll'$ enter the relation $\Ll'=\Ll-y$. Then necessarily
\begin{equation}\label{Eq:Equivariance0}
w_{x,x'}(\Ll) = w_{x-y,x'-y}(\Ll - y), \quad x,x' \in \Ll,
\end{equation}
and, certainly, this is what the laboratory measurements will return. Now, we fix a pattern $\Ll_0$ and apply these facts to the corresponding Hamiltonian. We find
\begin{equation*}
H_{\Ll_0} = \sum_{x,x'\in \Ll_0} w_{0,x'-x}(\Ll_0-x) \, |x \rangle \langle x' |,
\end{equation*}
and, if we introduce $y=x'-x \in \Ll_0 -x$ and drop the trivial index, then
\begin{equation*}
H_{\Ll_0} = \sum_{x\in \Ll_0} \sum_{y\in \Ll_0-x} w_{y}(\Ll_0-x) \, |x \rangle \langle x+y |,
\end{equation*}
or
\begin{equation}\label{Eq:CanonicalD}
(H_{\Ll_0} \psi)(x) = \sum_{y\in \Ll_0-x} w_{y}(\Ll_0-x) \, \psi(x+y), \quad \psi \in \ell^2(\Ll_0).
\end{equation}

\begin{remark}\label{Re:Galilean}{\rm We want to stress that the equivariance relation~\eqref{Eq:Equivariance0} is not a statement about the pattern but about the physical processes involved in the coupling of the quantum resonators. It is for this reason that we speak of Galilean invariance as opposed to translational invariance, as the latter is often an attribute of a pattern rather than of a physical law. 
}$\Diamond$
\end{remark}

As it was already pointed out in \cite{BourneJPA2018}, and will be exemplified again in subsection~\ref{SubSec:groupoidAlg1}, the expression in Eq.~\eqref{Eq:CanonicalD} is directly related to the left regular representations of the groupoid algebra canonically attached to the pattern  \cite{Bellissard1986,Bellissard1995,KellendonkRMP95}. Let us state explicitly that operators with the type of action as in Eq.~\ref{Eq:CanonicalD} are never compact operators, hence they do not belong to the algebra of local observations available to an experimenter.

\subsection{Spaces of patterns}
\label{SubSec:Patterns}

The main goal of this section is to supply the definition of the continuous hull and of the transversal of a point pattern. The latter will serve as the space of units for the Bellissard-Kellendonk groupoid.

Since we are dealing with  patterns in $\RM^d$, we recall the metric space $(\Kk(\RM^d),{\rm d}_{\rm H})$ of compact subsets equipped with the Hausdorff metric, as well as the larger space $\Cc(\RM^d)$ of closed subsets of $\RM^d$, topologized as below. Throughout, if $(X,{\rm d})$ is a metric space, $B(x,r)$ and $\mathring B(x,r)$ denote the closed, respectively, open balls centered at $x$ and of radius $r$. Also, $\Kk(X)$ stands for the set of compact subsets of $X$.

\begin{definition}[\cite{ForrestAMS2002,LenzTheta2003}] \label{Def-PatternMetric} Given a closed subset $\Lambda \subset \RM^d$, define
\begin{equation*}
\Lambda[r] = \big (\Lambda \cap B(0,r)\big ) \cup \partial B(0,r).
\end{equation*} 
Then
\begin{equation}\label{Eq:PatternMetric}
D(\Lambda,\Lambda')=\inf \big \{ 1/(1+r) \, | \, {\rm d}_{\rm H} (\Lambda[r],\Lambda'[r])<1/r \big \}
\end{equation}
defines a metric on $\Cc(\RM^d)$. 
\end{definition}

\begin{definition} We call the metric space $\big ( \Cc(\RM^d),D \big )$ the space of patterns in $\RM^d$.
\end{definition}

The space of patterns is bounded, compact and complete. Furthermore, there is a continuous action of $\RM^d$ by translations, that is, a homomorphism $\mathfrak{t}$ between topological groups,
\begin{equation}\label{Eq:RDGroupAction}
\mathfrak t : \RM^d \rightarrow {\rm Homeo}\big ( \Cc(\RM^d),D \big ), \quad \mathfrak t_{x}(\Lambda) = \Lambda - x.
\end{equation} 

\begin{remark}{\rm The topology of $\big ( \Cc(\RM^d),D\big )$ can be conveniently described in the following way. The space $\RM^d$ can be canonically embedded in its one-point compactification, the $d$-dimensional sphere $\SM^d$. As such, there exists a canonical embedding $\mathfrak j : \Cc(\RM^d) \to \Kk(\SM^d)$ and the latter can be topologized using the Hausdorff distance. The pullback topology through $\mathfrak j$ defines a topology on $\Cc(\RM^d)$ that is equivalent with the topology of $\big ( \Cc(\RM^d),D\big )$ \cite[p.~17]{ForrestAMS2002}.
}$\Diamond$
\end{remark} 

\begin{definition} Let $\Ll \in \RM^d$ be discrete and infinite and fix $0<r<R$.
\begin{enumerate}[\rm \ (1)] 
\item $\Ll$ is $r$-uniformly discrete if $|B(x,r) \cap \Ll| \leq 1$ for all $x \in \RM^d$.
\item $\Ll$ is $R$-relatively dense if $|B(x,R) \cap \Ll | \geq 1$ for all $x \in \RM^d$.
\end{enumerate}
An $r$-uniform discrete and $R$-relatively dense set $\Ll$ is called an $(r,R)$-Delone set.
\end{definition}

\begin{remark}{\rm Throughout, if $S$ is a set then $|S|$ denotes its cardinal.
}$\Diamond$
\end{remark}

\begin{proposition} The set of $(r,R)$-Delone sets in $\RM^d$, ${\rm Del}_{(r,R)}(\RM^d)$, is a compact subset of $(\Cc(\RM^d),D)$. A basis for its topology is supplied by the family of sub-sets
\begin{equation}\label{Eq:TopoBase2}
U_M^\epsilon(\Ll) = \big \{\mathcal{L}'\in {\rm Del}_{(r,R)}(\RM^d), \quad \mathrm{d}_{H}(\mathcal{L}[M],\mathcal{L}'[M])<\epsilon \big \}, \ M,\epsilon >0.
\end{equation}
\end{proposition}

Let $\Ll_0 \subset \RM^d$ be a fixed Delone point set. 

\begin{definition}\label{Def:Hull} The continuous hull of $\Ll_0$ is the topological dynamical system $(\Omega_{\Ll_0},\mathfrak t,\RM^d)$, where
\begin{equation*}
\Omega_{\Ll_0}= \overline{\{\mathfrak t_a(\Ll_0)=\Ll_0 - a, \ a\in \RM^d\}} \subset \Cc(\RM^d),
\end{equation*}
with the closure in the metric topology induced by \eqref{Eq:PatternMetric}.
\end{definition}

Since $\Omega_{\Ll_0}$ is a closed subset of the space of patterns, every point $\Ll \in \Omega_{\Ll_0}$ defines a closed subset of $\RM^d$. Theorem~2.8 of \cite{Bellissard2000} assures us that all these closed subsets are in fact Delone point sets. 

\begin{remark}{\rm The above dynamical system is always transitive but may fail to be minimal. Indeed, if $\Ll \in \Omega_{\Ll_0}$ belongs to the closure and not to the orbit of $\Ll_0$, then the orbit of $\Ll$ belongs to $\Omega_{\Ll_0}$ but might not be dense in $\Omega_{\Ll_0}$, as it is indeed the case for {\it e.g.} periodic patterns in $\RM$ with one defect \cite{SadunBook}[p.~9] or the disordered pattern (see Example~\ref{Ex:Random}). In fact, the character of the patterns contained inside the hull of a fixed pattern $\Ll_0$ can vary drastically and $(\Omega_{\Ll_0},\mathfrak t,\RM^d)$ can have many minimal components (see {\it e.g.} \cite{ProdanJPA2021}). It is important to acknowledge that, while $\Omega_\Ll$ may not coincide with $\Omega_{\Ll_0}$ for $\Ll \in \Omega_{\Ll_0}$, it is always true that $\Omega_\Ll \subseteq \Omega_{\Ll_0}$. 
}$\Diamond$
\end{remark}

\begin{definition}\label{Def:Transversal} The canonical transversal of a continuous hull $(\Omega_{\Ll_0},\mathfrak t,\RM^d)$ of a Delone set $\Ll_0$ is defined as
\begin{equation*}
\Xi_{\Ll_0} = \{ \Ll \in \Omega_{\Ll_0}, \ 0 \in \Ll\}.
\end{equation*}
The transversal is a compact subspace of $\Cc(\RM^d)$.
\end{definition}

\begin{example}\label{Ex:PointXi}
{\rm
 For a periodic pattern such as $\Ll_0 = \ZM^d \subset \RM^d$, the transversal consists of just one point, the lattice $\Ll_0$ itself.
}$\Diamond$
\end{example}

\begin{example}\label{Ex:Random}
{\rm
For the pattern $\{n + \lambda (\xi_n - \frac{1}{2})\}_{n \in \ZM}$, with the $\xi_n$ entries drawn randomly and distinctly from the interval $[0,1]$ and $\lambda < 1$, the transversal is the Hilbert cube $[0,1]^\ZM$.
}$\Diamond$
\end{example}

\begin{figure}[t]
\center
\includegraphics[width=1\textwidth]{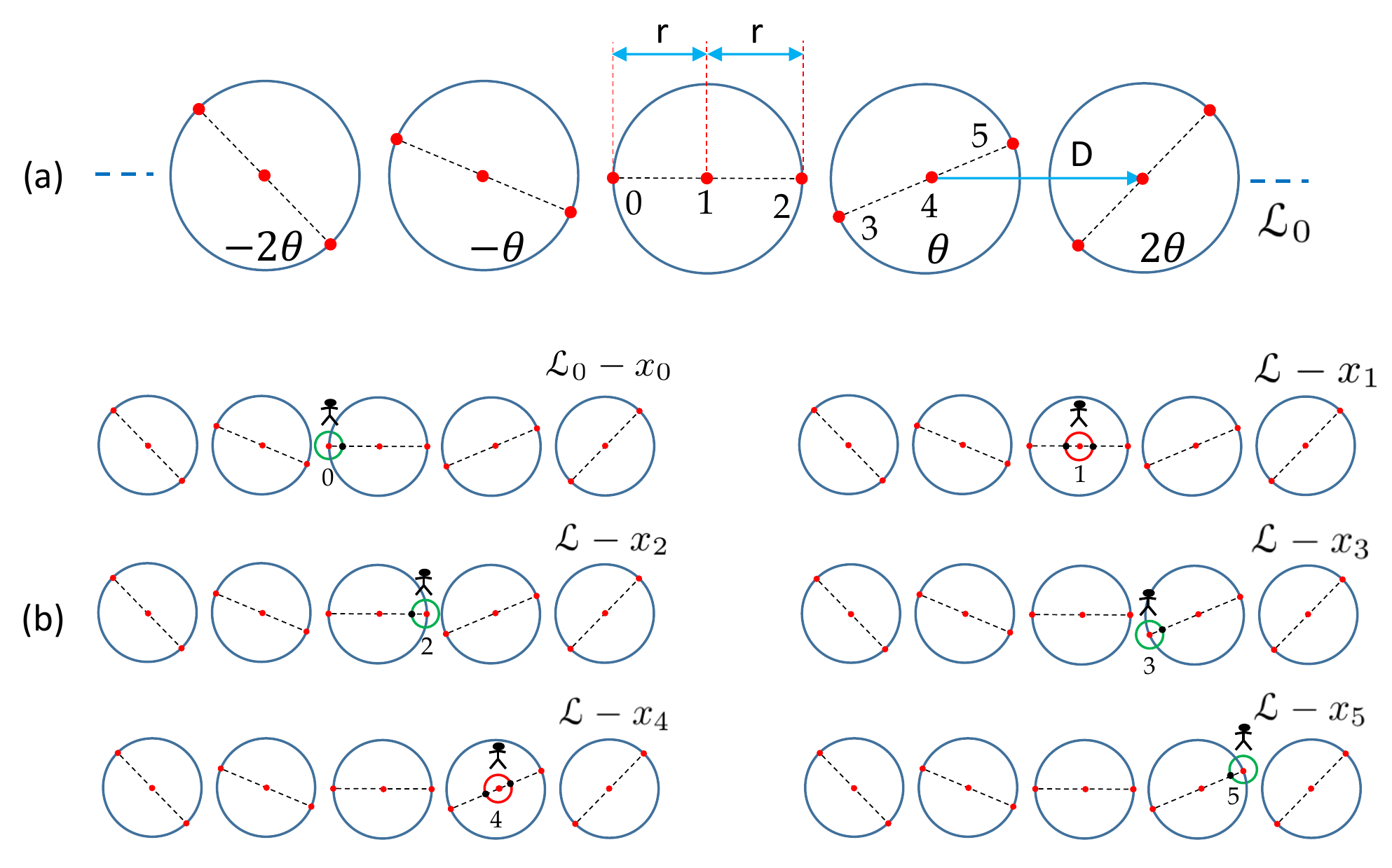}\\
  \caption{\small (a) The geometric algorithm for the pattern from Example~\ref{Ex:CircleXi} consists of a colinear triplet of points separated by a distance $r$, that are iteratively translated horizontally by $D>2r$ and rotated by an angle $\theta \in \RM \setminus \QM$. (b) The practical process of mapping the transversal of the pattern (see Example~\ref{Ex:CircleXi} for details).}
 \label{Fig:Pattern}
\end{figure}

\begin{example}\label{Ex:CircleXi}
{\rm Consider the pattern generated by the geometric algorithm described in Fig.~\ref{Fig:Pattern}(a). To map its transversal, an experimenter rigidly moves the pattern so that each point sits at the origin of the laboratory frame and makes the observations depicted in Fig.~\ref{Fig:Pattern}(b). Specifically, if one of the outer points of the triplets sits at the origin, such as 0, 2, etc., the experimenter centers the green circle at the origin and marks its intersection with the dotted line joining the triplets. If one of the inner points of the triplets sits at the origin, such as 1, 4, etc., the experimenter centers the red circle at the origin and marks its intersections with the dotted line joining the triplets. The experimenter then realizes that these markings are all the information needed to reproduce the entire sequence of translated patterns. Indeed, moving a distance $r$ in the direction indicated by a marked point on the green circle, the experimenter puts down one point and then moves by another distance $r$ and puts down a second point, hence completing the triplet (the origin is also counted in). Then the experimenter translates and rotates the triplet according to the algorithm. In this way, the experimenter reproduces the translated pattern which produced the marked point on the green circle in the first place. On the red circle, the experimenter chooses a pair of opposite markings and walks distances $r$ from the origin in the directions indicated by those markings and puts down two points, hence completing again the triplet. Then the experimenter translates and rotates the triplet according to the algorithm. In this way, the experimenter reproduces the translated pattern which produced the pair of marked points on the red circle in the first place. The conclusion is that the translated patterns $\Ll_0-x$, $x \in \Ll_0$, are in one to one correspondence with the markings on the green circle and with the pairs of markings on the red circle. If $\theta/2\pi$ is an irrational number, these markings densely fill the two circles. Lastly, if the seeding triplet $\{0,1,2\}$ is rotated by an angle $\varphi$ and the algorithm is applied again, the so obtained patterns vary continuously with $\varphi$ inside $\Cc(\RM^2)$. This assures us that the closure of the discrete set of patterns $\{\Ll_0-x, \ x \in \Ll_0\}$ fills entirely the two circles. We can conclude that $\Xi_{\Ll_0}$ consists of the green circle and of the red circle with the opposite points identified (topologically, this is again a circle). Note that moving continuously along the red connected component of the transversal results in the exchange of the outer points of each triplet.
}$\Diamond$
\end{example}

\begin{example}{\rm The reader can find in \cite{ChengArxiv2020} a general algorithm that generates patterns for which $\Xi_{\Ll_0}$ is a torus of arbitrary dimension.
}$\Diamond$
\end{example}

\begin{example}{\rm The transversal of the Fibonacci quasicrystal is a cantorized circle (see {\it e.g.} \cite{KellendonkAHP2019}).
}$\Diamond$
\end{example}

\begin{remark}{\rm In the physics literature, the transversal of a pattern appears under the name of phason space and a point of this space is called a phason. The patterns with transversals $\Xi_{\Ll_0}$ that have non-trivial topology are particularly interesting for practical applications and certainly all the works in materials science cited in our introduction involve such patterns. This is so because the phason can be physically driven inside $\Xi_{\Ll_0}$ to achieve various exquisite effects. This is one of the main reasons why, to be relevant for applications, the framework of interacting fermions must be developed for generic aperiodic patterns.
}$\Diamond$
\end{remark}

\subsection{The groupoid $C^\ast$-algebra of single fermion dynamics}
\label{SubSec:groupoidAlg1}

In this subsection, we review the Bellissard-Kellendonk groupoid associated to the dynamics of a single fermion over a Delone point set. We then show how the left regular representations of the groupoid $C^\ast$-algebra generates all Galilean invariant Hamiltonians over the Delone set. 

\begin{definition}\label{Def:Groupoid1} The Bellissard-Kellendonk topological groupoid associated to a Delone set $\Ll_0$ consists of:
\begin{enumerate}[\ \rm 1.]

\item The set
\begin{equation*}
\Gg_1 :=\{(\Ll,x) \in \Xi_{\Ll_0} \times \RM^d, \ x \in \Ll\}
\end{equation*}
equipped with the inversion map
\begin{equation*}
(\Ll,x)^{-1} = (\Ll-x,-x).
\end{equation*}
\item The subset of $\Gg_1 \times \Gg_1$
\begin{equation*}
 \Gg_1^{(2)} =\{((\Ll,x),(\Ll',y)) \in \Gg_1 \times \Gg_1, \  \Ll'=\Ll-x\}
 \end{equation*}
equipped with the composition 
\begin{equation*}
(\Ll,x) \cdot (\Ll-x,y) = (\Ll,x+y).
\end{equation*}
\end{enumerate} 
The topology on $\Gg_1$ is the relative topology inherited from $\Xi_{\Ll_0} \times \RM^d$.
\end{definition}

\begin{remark}\label{Re:Xi1}{\rm The algebraic structure of $\Gg_1$ can be also described as coming from the equivalence relation on $\Xi_{\Ll_0}$ (see Example~\ref{Ex:Relgroupoid}):
\begin{equation*}
R_{\Xi_{\Ll_0}} = \big \{ (\Ll,\Ll') \in \Xi_{\Ll_0} \times \Xi_{\Ll_0}, \ \Ll' = \Ll -a \ \mbox{for some} \ a \in \RM^d \big \}.
\end{equation*}
Note, however, that the topology of $\Gg_1$ is not the one inherited from $\Xi_{\Ll_0} \times \Xi_{\Ll_0}$.
}$\Diamond$
\end{remark}

\begin{proposition}[\cite{BourneAHP2020},~Prop.~2.16] The set $\Gg_1$ is a second-countable, Hausdorff, \'etale groupoid.
\end{proposition}

\begin{remark}{\rm The above statement remains true if $\Ll_0$ is just a uniformly discrete pattern. This becomes relevant when lattices with defects are investigated.
}$\Diamond$
\end{remark}

\begin{remark}{\rm It follows from their very definition~\eqref{Eq:RS} that the range and the source maps are
\begin{equation*}
\mathfrak r(\Ll,x) = (\Ll,0), \quad \mathfrak s(\Ll,x) = (\Ll-x,0),
\end{equation*}
hence the space of units $\Gg_1^{(0)}$ can be canonically identified with $\Xi_{\Ll_0}$ via $(\Ll,0) \mapsto \Ll$. Furthermore, we have
\begin{equation}\label{Eq:RSMinus}
\mathfrak r^{-1}\big (\mathfrak r (\Ll,x)\big )=\{(\Ll,y), \, y \in \Ll\}, \quad \mathfrak s^{-1}\big (\mathfrak s(\Ll,x) \big ) = \{(\Ll-x,y-x)^{-1},\ y\in \Ll\},
\end{equation}
hence both spaces $\mathfrak r^{-1}\big (\mathfrak r (\Ll,x)\big )$  and $\mathfrak s^{-1}\big (\mathfrak s(\Ll,x) \big )$ can be canonically identified with the lattice $\Ll$ itself. 
}$\Diamond$
\end{remark}

\begin{remark}{\rm  We recall that all $\Ll \in \Xi_{\Ll_0}$ are Delone point sets, hence infinite and discrete topological spaces. As such, the Bellissard-Kellendonk groupoid has a compact space of units but infinite discrete fibers.
}$\Diamond$
\end{remark}

We now turn our attention to the $C^\ast$-algebra $C_r^\ast(\Gg_1)$ canonically associated to the groupoid $\Gg_1$ and, hence, to the pattern $\Ll_0$. Eqs.~\eqref{Eq:Conv} and \eqref{Eq:RSMinus} give
\begin{equation*}
(f_1 \ast f_2)(\Ll,x) = \sum_{y \in \Ll} f_1(\Ll,y) f_2(\Ll-y,x-y), \quad (\Ll,x) \in \Gg_1.
\end{equation*}
Also, the involution takes the form $f^\ast(\Ll,x) = \overline{ f (\Ll -x,-x)}$. 

\begin{remark}{\rm We stress again that $C_r^\ast(\Gg_1)$ is separable for a Delone and, in fact, for any uniformly discrete pattern. This in turn assures us that its $K$-theory is countable, hence we have a sensible and meaningful classification of the Hamiltonians by stable homotopy. Without the remarkable insight from Bellissard and Kellendonk, the only available option would have been to place the Hamiltonians in the large, non-separable algebra $\BM\big ( \ell^2(\Ll) \big )$, which has trivial $K$-theory.
}$\Diamond$
\end{remark}

As discussed in section~\ref{Sec:LReps}, the left regular representations of $C^\ast_r(\Gg_1)$ are indexed by the points $\Ll \in \Xi_{\Ll_0}$ and are carried by the Hilbert spaces $\ell^2\big(\mathfrak s^{-1}(\Ll)\big )$. In particular, for $\Ll_0$ and $\varphi \in \ell^2\big(\mathfrak s^{-1}(\Ll_0)\big )$,
\begin{equation*}
\begin{aligned}
[\pi_{\Ll_0}(f)\varphi]\big ((\Ll_0,x)^{-1}\big ) =\sum_{y \in \Ll_0-x} f(\Ll_0-x,y) \varphi \big ((\Ll_0,x+y)^{-1}\big).
\end{aligned}
\end{equation*}
We can canonically map $\varphi$ to a function $\psi$ on $\Ll_0$, via $\psi (x) = \varphi\big ((\Ll_0,x)^{-1}\big )$, and define the left-regular representation over $\ell^2(\Ll_0)$,
\begin{equation}\label{Eq:Rep1}
[\pi_{\Ll_0}(f)\psi](x) = \sum_{y \in \Ll_0-x} f(\Ll_0-x,y) \psi(x+y).
\end{equation}
Furthermore, for any $a \in \Ll_0$, we can define the Hilbert space isomorphisms \[T_a:\ell^2(\Ll_0-a) \rightarrow \ell^2(\Ll_0), \quad T_a \psi= \psi \circ \mathfrak t_{a}.\] Then the left regular representations enjoy the following covariant property
\begin{equation}\label{Eq:CovP}
T_a^\ast \pi_{\Ll_0}(f) T_a = \pi_{\Ll_0-a}(f), \quad f\in C_r^\ast(\Gg_1).
\end{equation}

We arrive now to the main conclusion of the section. If we compare the expression \eqref{Eq:Rep1} with the action of Galilean invariant Hamiltonians~\eqref{Eq:CanonicalD}, we see that they are identical once we identify $f(\Ll_0-x,y)$ and $w_y(\Ll_0-x)$. The outstanding conclusion is that all Galilean invariant Hamiltonians over $\Ll_0$ can be generated from the left regular representations of $C_r^\ast(\Gg_1)$, which is the smallest $C^\ast$-algebra with this property.

\section{Interacting Fermions: The Algebra of Local Observables}
\label{Sec:FermionsCAR}

In this section, we first describe the $C^\ast$-algebra ${\rm CAR}(\Ll)$ of local observables for a system of many-fermions hopping on a discrete lattice $\Ll$. As in the previous section, this local algebra plays a similar role as a Hilbert space does when we define operators: It will help us define an algebra of derivations associated with the dynamics of the local observables. While our final goal is to characterize this algebra of derivations, the success of that program rests on a fine characterization of the structure of ${\rm CAR}(\Ll)$, which is supplied in this section. Of particular importance is our choice of presenting the elements in a symmetric manner that treats all possible orderings of the anti-commuting generators on equal footing. This leads to the notion of bi-equivariant coefficients w.r.t. permutation groups, which later will be connected with the material introduced in section~\ref{Sec:EquiGg}. The Fock representation of ${\rm CAR}(\Ll)$ is treated in the same spirit using frames that are stable against the actions of the permutation groups. One of our important observations is that the product of operators continues to manifest itself as a certain convolution of the ``symmetrized'' matrix elements. The reader will also find here the characterization \cite{BratteliTAMS1972} of the lattice of ideals of ${\rm GICAR}(\Ll)$, the sub-algebra of gauge invariant elements. Of essential importance for our program is the observation that ${\rm GICAR}(\Ll)$ admits a filtration by primitive ideals and that the quotient of two consecutive such ideals is isomorphic to the algebra of compact operators. In other words, ${\rm GICAR}(\Ll)$ is a solvable $C^\ast$-algebra in the sense of \cite{DyninPNAS1978}. 

\subsection{The {\rm CAR} and ${\rm GICAR}$ algebras over a lattice}
\label{SubSec:CARandGICAR}

The setting is that of a Delone pattern $\Ll$ whose points are populated by spin-less fermions. The algebra of local observables is supplied by the algebra of canonical anti-commutation relations over $\Ll$, denoted here by ${\rm CAR}(\Ll)$ \cite{BratteliBook2}. It is constructed from a net $\{\Ll_k\}$ of finite subsets of $\Ll$, with $\Ll_k \subset \Ll_{k+1}$ and $\bigcup_{k\in \ZM} \Ll_k = \Ll$. For each of these finite subsets, one defines ${\rm CAR}(\Ll_k)$ as the finite $\ast$-algebra generated by $a_x$, $x \in \Ll_k$, and the relations
\begin{equation}\label{Eq:CAR}
a_x a_{x'} + a_{x'} a_x =0, \quad a_x^\ast a_{x'} + a_{x'}a_x^\ast = \delta_{x,x'}.
\end{equation}

\begin{definition} The algebra ${\rm CAR}(\Ll)$ is the limit of the inductive tower of finite algebras ${\rm CAR}(\Ll_k)$, supplied by the canonical embeddings ${\rm CAR}(\Ll_k) \rightarrowtail {\rm CAR}(\Ll_{k+1})$.
\end{definition}

\begin{remark}{\rm We recall that the universal ${\rm CAR}$-algebra admits a unique faithful tracial state, which will be denoted by $\Tt$. It will play an essential role in section~\ref{Sec:AFRGI} (see Proposition~\ref{Pro:StarH}).
}$\Diamond$
\end{remark}

 The local algebra of observables constructed this way formalizes the local measurements available to an experimenter dealing with a system of fermions over $\Ll$. As in subsection~\ref{SubSec:LocalAlg0}, this experimenter analyses the dynamics of the available local observables and the main task is to figure out the group of automorphisms $\{\bm \alpha_t\}_{t \in \RM}$ that implements the time evolution, which is usually done by experimenting with many such local observables. As in subsection~\ref{SubSec:LocalAlg0}, the experimenter will find that these automorphisms are outer, hence the Hamiltonians generating the dynamics do not belong to the local algebra. This central aspect will be addressed in section~\ref{Sec:FermionsDynamics}.


\subsection{Symmetric presentation}\label{Sec:StandardP} Given the anti-commutation relations~\eqref{Eq:CAR}, a word constructed from the generators has many equivalent presentations. The formalism developed in sections~\ref{Sec:FermionsDynamics} and \ref{Sec:Fermiongroupoid} relies on our specific choice to treat all equivalent word presentations on equal footing. This leads to a specific presentation of ${\rm CAR}(\Ll)$, which we call the symmetric presentation. We describe it here. 

 For a finite subset $J \subset \RM^d$ of cardinality $|J|=N$, we consider the set $\mathcal{S}(\Ii_N,J)$ of bijections from $\Ii_N := \{1,\ldots,N\}$ to $J$. Any such bijection supplies a particular enumeration of the elements of $J$, hence, an order. The group $\mathcal{S}_N$ of ordinary permutations of $N$ objects can and will be identified with $\mathcal{S}(\Ii_N,\Ii_N)$. It has natural left and right actions on $\mathcal{S}(\Ii_N,J)$ via 
\begin{equation*}
s \cdot \chi_J := \chi_J \circ s^{-1}, \quad \chi_J \cdot s := \chi_J \circ s, \quad s\in \mathcal{S}_N,\quad \chi_{J}\in\mathcal{S}(\Ii_N,J),
\end{equation*} 
respectively. Consider now a Delone set $\Ll$. We introduce the following relations
\begin{equation}\label{Eq:Mono1}
 (J,\chi_J) \mapsto a_J(\chi_J):=a_{\chi_J(|J|)} \cdots a_{\chi_J(|1|)} \in {\rm CAR}(\Ll),
\end{equation}
for all $J \in \Kk(\Ll)$ and $\chi_J \in \mathcal{S}(\Ii_{|J|}, J)$. Obviously, 
\begin{equation*}
a_J\big(s \cdot \chi_J\big)=a_J\big(\chi_J\cdot s\big) = (-1)^s a_J\big(\chi_J), \quad \forall \ s \in \mathcal{S}_{|J|},
\end{equation*}
and 
\begin{equation}\label{Eq:Mono3}
a_J(\chi_J)^\ast= a^\ast_{\chi_J(1)} \cdots a^\ast_{\chi_J(|J|)}.
\end{equation}
It is also useful to introduce the special elements
\begin{equation*}
n_J = a^\ast_J(\chi_J) a_J(\chi_J) = \tfrac{1}{|J|!} \sum_{\chi_J} a^\ast_J(\chi_J) a_J(\chi_J).
\end{equation*}
Throughout, we will use the convention $a_\emptyset = a^\ast_\emptyset = 1$, the unit of ${\rm CAR}(\Ll)$.

\begin{remark}{\rm Since $\Ll$ is a discrete set, the compact subsets of $\Ll$ are exactly the finite subsets. Throughout, we will only involve compact subsets, hence the cardinality of the subsets is finite, everywhere in our discussion.
}$\Diamond$
\end{remark}

One could argue that a choice for the enumeration of the subsets of $\Ll$ can be made once and for all. This is certainly possible if the pattern $\Ll$ is fixed. However, we want and, in fact, are forced to allow for deformations of the pattern, {\it e.g.} at least those induced by the simple translations of $\Ll$, already encountered in section~\ref{Sec:SingleFermion}. Under such deformations, two points of the pattern can be exchanged, as it was the case in Example~\ref{Ex:CircleXi}, and the enumeration of the points of a subset containing those two points becomes ambiguous. This example shows that the only acceptable option is to use all the available enumerations on equal footing and this is indeed what we will do when presenting the elements of the ${\rm CAR}$ algebra. However, in doing so, we need to impose restrictions on the coefficients. It is at this point where the material from section~\ref{Sec:EquiGg} makes its entrance.

\begin{definition} For a pair $(J,J')$ of compact subsets of $\Ll$, a bi-equivariant coefficient is a map $c_{J,J'}: \mathcal{S}(\Ii_{|J|},J)\times \mathcal{S}(\Ii_{|J'|},J')\to \mathbb{C}$ such that
\begin{equation*}
c_{J,J'} ( s_1\cdot\chi_J,\chi_{J'} \cdot s_2 ) = (-1)^{s_1}  \ c_{J,J'}(\chi_{J},\chi_{J'}) \ (-1)^{s_2}.
\end{equation*}
\end{definition}

\begin{remark}{\rm Given the particular representations of the permutation groups entering in the above definition, one could argue that just an orientation of the subsets $J$ and $J'$ will suffice. This is indeed the case for the presentation of the CAR algebra. However, the full orders on the sub-sets supply special points, namely $\chi_J(1)$, which will be used in an essential way in our construction of the groupoid associated with the dynamics of the fermions (see section~\ref{SubSec:AlgGN}).
}$\Diamond$
\end{remark}

\begin{proposition} Every element of ${\rm CAR}(\Ll)$ can be uniquely presented as a norm convergent sum
\begin{equation}\label{Eq:StandardP1}
A = \sum_{J,J' \in \Kk(\Ll)} \tfrac{1}{\sqrt{|J|!|J'|!}}\sum_{\chi_J,\chi_J'}  c_{J,J'}\big(\chi_J,\chi_{J'}\big ) a_J(\chi_J)^\ast a_{J'}(\chi_{J'}),
\end{equation}
where it is understood that the coefficients are bi-equivariant and that the second sum runs over the whole set $\mathcal{S}(\Ii_{|J|},J)\times \mathcal{S}(\Ii_{|J'|},J')$.
\end{proposition}
We will refer to Eq.~\eqref{Eq:StandardP1} as the symmetric presentation of $A$. The consideration of the combinatorial factors in front of the second sum will be justified in Propositions~\ref{Pro:FockOp} and ~\ref{Pro:CARQuote}.

\begin{remark}{\rm The first sum in~\eqref{Eq:StandardP1} can include $J=\emptyset$, $J'=\emptyset$ or $J=J'=\emptyset$. In these cases, we use our convention $a_\emptyset =1$.
}$\Diamond$
\end{remark}

\begin{remark}{\rm It is important to acknowledge that the coefficients $c_{J,J'}$ in Eq.~\ref{Eq:StandardP1} must display a certain decay as the sets $J$ or $J'$ are pushed to infinity. For example, a formal series with coefficients $c_{J,J}$ that do not decay to zero as the diameter of the $\RM^d$ sub-set $\{0\}\cup J \cup J'$ goes to infinity is not part of ${\rm CAR}(\Ll)$.
}$\Diamond$
\end{remark}

The gauge invariant (GI) elements are, by definition, elements that can be written as in Eq.~\eqref{Eq:StandardP1}, but with the first sum restricted to pairs of subsets with $|J|=|J'|$. Such elements are invariant against the circle action $a_x \mapsto \lambda a_x$, $|\lambda|=1$, hence the name ``gauge invariant''. They form a sub-algebra denoted here by ${\rm GICAR}(\Ll)$. The symmetric presentation of such elements can be organized as
\begin{equation}\label{Eq:StandardGI}
A = \sum_{n \in \NM} \tfrac{1}{n!} \sum_{J,J' \in \Kk_n(\Ll)} \sum_{\chi_J,\chi_J'}  c_{J,J'}\big(\chi_J,\chi_{J'}\big ) a_J(\chi_J)^\ast a_{J'}(\chi_{J'}), 
\end{equation}
where we introduced the new notation $\Kk_n(\Ll)$ for the finite un-ordered sub-sets of $\Ll$ of cardinal $n$.

\begin{remark}\label{Re:CARGroupoid}{\rm A groupoid presentation of the CAR algebra was supplied in Example~1.10 of \cite{RenaultBook}. Although not used in our work, we provide it here for completeness. The construction requires an ordering of the lattice points and it starts from the configuration space $S=\prod_{x \in \Ll} S_x$, where $S_x =\{0,1\}$ and $S$ is equipped with the product topology. A configuration $s=(s_x) \in S$ singles out the sites $x$ with $s_x=1$, which can be thought of as the sites of $\Ll$ populated by fermions. Two configurations are declared equivalent iff they differ at at most a finite number of places. Algebraically, the CAR groupoid is the groupoid $\Cc$ corresponding to this equivalence relation (see Example~\ref{Ex:Relgroupoid}). The \'etale topology on $\Cc$ is strictly finer than the relative topology inherited from $S \times S$. To describe this topology, we follow Example~8.3.5 in \cite{SimsSzaboWilliamsBook2020}. First, one identifies $\Ll$ with $\NM$ by using the predefined order. Then, for $n \in \NM$ and any two finite words $v$ and $w$ from $\{0,1\}^n$, one defines the subset $$Z_n(v,w) = \{(vs,ws) | s \in S\} \subseteq \Cc,$$ where $vs$ means concatenation of the words $v$ and $s$. The collection $$\{Z_n(v,w), \, n \in \NM, \, v,w \in \{0,1\}^n\},$$ is the basis of a topology that makes $\Cc$ into an \'etale groupoid. The associated groupoid algebra $C^\ast_r(\Cc)$ is isomorphic to the CAR algebra. In \cite{RenaultBook}, the isomorphism established by showing that the two algebras display common Elliot invariants. A  direct proof of the isomorphism can be found in \cite[Example 9.2.7 ]{SimsSzaboWilliamsBook2020}. 
}$\Diamond$
\end{remark} 

\subsection{Algebraic relations and notation} The commutation relations~\eqref{Eq:CAR} lead to somewhat complex relations between the monomials. This section lists several key algebraic relations with proofs omitted. They can be stated more effectively once some straightforward but essential notation is introduced. 

\begin{definition}\label{Def:Wedge} For $J,J' \in \Kk(\Ll)$ disjoint subsets and $\chi_{J} \in \mathcal{S}(\Ii_{|J|},J) $, $\chi_{J'} \in \mathcal{S}(\Ii_{|J'|}, J') $, we define $\chi_{J} \vee \chi_{J'} \in \mathcal{S}(\Ii_{|J \cup J'|}, J \cup J')$ by
\begin{equation*}
\chi_{J} \vee \chi_{J'}(n) = \left \{
\begin{array}{ll}
 \chi_{J}(n) \ & {\rm if} \ n \leq |J|, \\
 \chi_{J'}(n-|J|) \ & {\rm if} \ n > |J|.
\end{array}
\right .
\end{equation*}
In words,  $\chi_{J} \vee \chi_{J'}$
 enumerates first the elements of $J$ and then continue with the enumeration of the elements of $J'$, in the order set by $\chi_{J}$ and $\chi_{J'}$.
 \end{definition}
 
With this notation, for example, we have:
\begin{equation}\label{Eq:Id1}
a_{J}(\chi_{J})^\ast \, a_{J'}(\chi_{J'})^\ast = a_{J \cup J'}(\chi_{J} \vee \chi_{J'})^\ast,
\end{equation}
whenever $J \cap J' =\emptyset$. The following relations are useful for reducing products of generic ${\rm CAR}$ elements to the symmetric form and for mapping to the Fock representations.

\begin{proposition} For $J,J' \in \Kk(\Ll)$, the following identity holds:
\begin{equation}\label{Eq:Id3}
\begin{aligned}
a_{J}(\chi_{J}) a^\ast_{J'}(\chi_{J'})  =  (-1)^\sigma \sum_{K \subseteq J \cap J'} (-1)^{|K|} a^{\ast}_{J'\setminus J}(\chi_{J' \setminus J}) \, n_K \, a_{J \setminus J'}(\chi_{J \setminus J'}),
\end{aligned}
\end{equation}
where the orderings  $\chi_{J}$, $\chi_{J'}$, $\chi_{J \setminus J'}$ and $\chi_{J' \setminus J}$ are independent and the sign factor in front is determined by the choice of these orderings (its exact form is not needed and is omitted).
\end{proposition}

We also mention the following identities, which will be instrumental for simplifying and manipulating the Fock representations of the ${\rm CAR}$ elements:

\begin{proposition}\label{Pro:U10} For $J,J',U \in \Kk(\Ll)$, the following identities hold:
\begin{equation*}
\begin{aligned}
n_U \, a_J(\chi_J)^\ast \, a_{J'}(\chi_{J'}) & = a_J(\chi_J)^\ast \, n_{U \setminus J} \, a_{J'}(\chi_{J'}), \\
a_J(\chi_J)^\ast \, a_{J'}(\chi_{J'}) \, n_U & = a_J(\chi_J)^\ast \, n_{U \setminus J'} \, a_{J'}(\chi_{J'}).
\end{aligned}
\end{equation*}
\end{proposition}

\subsection{Fock representation}
\label{Sec:FockRep}
The Fock representation is associated with the canonical vacuum state:

\begin{definition} The vacuum state $\eta:{\rm CAR}(\Ll)\to\mathbb{C}$ is defined by the rule
\begin{equation}\label{Eq:Eta}
\eta(A) = c_{\emptyset,\emptyset}, \quad A \in {\rm CAR}(\Ll),
\end{equation}
where $A$ is assumed to be in its symmetric presentation and $c_{\emptyset,\emptyset}$ is the coefficient corresponding to the unit.
\end{definition}

\begin{proposition} The following relations hold:
\begin{equation}\label{Eq:EtaRel1}
\eta\big (a_J(\chi_J)  a^\ast_{J'}(\chi_{J'})\big ) = (-1)^{\chi_J^{-1} \circ \chi_{J'}} \, \delta_{J,J'}, \quad J,J' \in \Kk(\Ll).
\end{equation}
\end{proposition}
\proof They are direct consequences of Eq.~\eqref{Eq:Id3}.\qed

\begin{proposition}\label{Pro:NullSpace} Let
\begin{equation*}
\Nn_\eta = \big \{ A \in {\rm CAR}(\Ll), \ \eta(A^\ast A) = 0 \big \},
\end{equation*}
be the closed left ideal of ${\rm CAR}(\Ll)$ associated to $\eta$. Then $\Nn_\eta$ is spanned by the monomials $a^\ast_J a_{J'}$ with $J' \neq \emptyset$.
\end{proposition}
\proof Clearly, any linear combination of the mentioned monomials belongs to $\Nn_\eta$. Then $\eta(Ba^\ast_J a_{J'})=0$ and $\eta(a^\ast_{J'} a_{J}B)=0$ for any $B \in {\rm CAR}(\Ll)$ and $J' \neq \emptyset$. Now, let $A \in {\rm CAR}(\Ll)$ and assume $A$ is in its symmetric presentation. Then, using the simple facts we just mentioned, $\eta(A^\ast A)$ simplifies to
 \begin{equation}\label{Eq:T24}
\begin{aligned}
\eta(A^\ast A) = \sum_{J,J' \in \Kk(\Ll)} \tfrac{1}{\sqrt{|J|!|J'|!}}\sum_{\chi_J,\chi_{J'}}  \overline{c_{J,\emptyset}(\chi_J)} c_{J',\emptyset}(\chi_{J'})  \eta\Big (a_J(\chi_J) a^\ast_{J'}(\chi_{J'}) \Big ),
\end{aligned}
\end{equation}
and Eq.~\eqref{Eq:EtaRel1} gives
 \begin{equation}\label{Eq:T25}
\eta(A^\ast A) = \sum_{J \in \Kk(\Ll)} \tfrac{1}{|J|!} \sum_{\chi_J} |c_{J,\emptyset}(\chi_{J})|^2.
\end{equation}
As such, if $A \in \Nn_\eta$, then necessarily all $c_{J,\emptyset}$ are identically zero.\qed

\begin{remark}{\rm It follows from the above that the linear space $\Nn_\eta+\Nn_\eta^\ast$ spans the entire ${\rm CAR}(\Ll)$, except the identity. From the very definition of $\eta$ in \eqref{Eq:Eta}, we can conclude that $\Nn_\eta+\Nn_\eta^\ast = {\rm ker} \, \eta $, which shows that $\eta$ is a pure state \cite[Th.~5.3.4]{MurphyBook}. Other arguments arriving at the same conclusion can be found in \cite{HugenholtzCMP1975}.
}$\Diamond$
\end{remark}

\begin{proposition}\label{Pro:Basis} Let $|U,\chi_U\rangle$ be the class of $a^\ast_U(\chi_U)$ in the GNS representation corresponding to $\eta$, 
\begin{equation}\label{Eq:Basis}
\big |U,\chi_U\big \rangle := a^\ast_U(\chi_U) +\Nn_\eta,
\end{equation} 
for $U \in \Kk(\Ll)$ and $\chi_U \in \mathcal{S}(\Ii_{|U|},U)$. Then the vectors~\eqref{Eq:Basis} span the Hilbert space of the GNS representation corresponding to $\eta$. Furthermore, the scalar product between two such vectors is
\begin{equation} \label{Eq:ScalarProd}
\big \langle U,\chi_U \big |V,\chi_V \big \rangle = (-1)^{\chi_U^{-1} \circ \chi_V} \, \delta_{U,V}. 
\end{equation}
\end{proposition}

\proof The statement is a direct consequence of Proposition~\ref{Pro:NullSpace} and Eq.~\eqref{Eq:EtaRel1}.\qed

\begin{remark}{\rm Proposition~\ref{Pro:Basis} gives the explicit connection between the GNS representation induced by $\eta$ and the well known representation on the anti-symmetric Fock space $\Ff^{(-)}(\Ll)$, spanned by \eqref{Eq:Basis}. Let us recall that, since $\eta$ is a pure state, the GNS space coincides with $ {\rm CAR}(\Ll)/\Nn_\eta$, {\it i.e.} no Cauchy completion is needed \cite[Th.~5.2.4]{MurphyBook}. Furthermore, the Fock representation is irreducible.
}$\Diamond$
\end{remark}

Since we want to avoid making any particular choice between the possible orderings of the sub-sets, we will work with the frame of $\Ff^{(-)}(\Ll)$ supplied by the vectors listed in Eq.~\eqref{Eq:Basis}. This frame is invariant against the action of the permutation group. As already mentioned in our introductory remarks, this is in fact a key point of our strategy. The following statement describes how operators can be uniquely presented using such a frame.

\begin{proposition}\label{Pro:FockOp} If $\BM\big (\Ff^{(-)}(\Ll)\big )$ is the algebra of bounded operators over the anti-symmetric Fock space, then any of its elements can be uniquely presented in the form
\begin{equation}\label{Eq:BOp1}
F = \sum_{U,U' \in \Kk(\Ll)} \tfrac{1}{\sqrt{|U|!|U'|!}}\sum_{\chi_U,\chi_{U'}} F_{U,U'}(\chi_U,\chi_{U'}) \, \big |U,\chi_U\big \rangle \big \langle U',\chi_{U'} \big |,
\end{equation}
where the coefficients are bi-equivariant and given by
\begin{equation}\label{Eq:C501}
F_{U,U'}(\chi_U,\chi_{U'}) =  \tfrac{1}{\sqrt{|U|! |U'|!}} \big \langle U,\chi_U\big | F  \big | U',\chi_{U'} \big \rangle.
\end{equation}
The sum in Eq.~\eqref{Eq:BOp1} converges in the strong operator topology of $\BM\big (\Ff^{(-)}(\Ll)\big )$. 
\end{proposition}

We will refer to Eq.~\eqref{Eq:BOp1} as the symmetric presentation of the operator. The following statement describes the rule for the composition of operators in this symmetric presentation.

\begin{proposition}\label{Pro:FProd} Let $F,F' \in \BM\big (\Ff^{(-)}(\Ll)\big )$ and consider their symmetric presentation as in Eq.~\eqref{Eq:BOp1}. Then the coefficients of the product $FF' \in \BM\big (\Ff^{(-)}(\Ll)\big )$ in the symmetric presentation are supplied by the convolution
\begin{equation*}
(FF')_{U,U'}(\chi_U,\chi_{U'}) = \sum_{V,\chi_V} F_{U,V}(\chi_U,\chi_{V}) F'_{V,U'}(\chi_V,\chi_{U'}).
\end{equation*}
Note that $(FF')_{U,U'}$ are indeed bi-equivariant coefficients.
\end{proposition}
\proof From Eq.~\eqref{Eq:ScalarProd} and the symmetric presentation of the operators,
\begin{equation}
\begin{aligned} \label{Eq:C101}
FF' = \sum_{U,U'} \tfrac{1}{\sqrt{|U|!|U'|!}} \sum_{V,V'} \tfrac{1}{\sqrt{|V|!|V'|!}} \sum_{\chi's}   & (-1)^{\chi_V^{-1} \circ \chi_{V'}} \delta_{V,V'} F_{U,V}(\chi_U,\chi_V)  \\
&  F'_{V',U'}\big(\chi_{V'},\chi_{U'}) \, \big |U,\chi_U\big \rangle \big \langle U', \chi_{U'}\big |,
\end{aligned}
\end{equation}
where the last sum is over all orderings $\chi$ appearing inside this sum. The sets $V$ and $V'$ need to coincide in the second sum, but their orderings do not. Nevertheless, using the equivariance of the coefficients, we have
\begin{equation*}
(-1)^{\chi_V^{-1} \circ \chi'_{V}}F'_{V,U'}\big(\chi'_{V},\chi_{U'}\big) = F'_{V,U'}\big(\chi_{V},\chi_{U'}\big).
\end{equation*}
Then the inner summand is independent of $\chi_{V'}$ and the sum over this ordering supplies the factor $|V|!$, bringing Eq.~\eqref{Eq:C101} to the form~\eqref{Eq:BOp1}, with the coefficients given in the statement.\qed

\begin{remark}{\rm It is at this point where the need for the cumbersome factorial factors is being explicitly displayed.
}$\Diamond$
\end{remark}

We now turn our attention to the Fock representation. The representation of the monomials can be derived directly from the identity~\eqref{Eq:Id1}:

\begin{proposition}\label{Pro:CARMono} The Fock representation of the monomials takes the form,
\begin{equation}\label{Eq:SimpleRep}
\begin{aligned}
\pi_\eta\big(a_J^\ast(\chi_J)\,  a_{J'}(\chi_{J'})\big )  = \sum_{\substack{\Gamma \in \Kk(\Ll) \\ \Gamma \cap(J \cup J') =\emptyset} } \ |J \cup \Gamma,\chi_J \vee \chi_\Gamma \rangle \langle J' \cup \Gamma,\chi_{J'} \vee \chi_\Gamma|,
\end{aligned}
\end{equation}
where the ordering $\chi_\Gamma$ can be any choice.
\end{proposition}

Let $\Ff^{(-)}_N$ be the closed linear sub-space of $\Ff^{(-)}$ spanned by the vectors $\big |U,\chi_U\big \rangle$ with $|U|=N$, usually called the $N$-fermion sector.

\begin{proposition}[\cite{BratteliTAMS1972}]The Fock representation $\pi_\eta$ restricted to the ${\rm GICAR}(\Ll)$ subalgebra decomposes into a direct sum
\begin{equation*}
\pi_\eta = \bigoplus_{N \in \NM} \pi_\eta^N, \quad \pi_\eta^N = \pi_\eta \downharpoonright _{\Ff^{(-)}_N},
\end{equation*}
of irreducible representations.
\end{proposition}
\proof  Proposition~\ref{Pro:CARMono} assures us that the representations of the elements from the ${\rm GICAR}$ sub-algebra leave the subspaces ${\Ff^{(-)}_N}$ invariant. These subspaces are mutually orthogonal and $\Ff^{(-)}=\bigoplus_{N \in \NM} {\Ff^{(-)}_N}$. Using again Proposition~\ref{Pro:CARMono}, one can easily see that the image of ${\rm GICAR}(\Ll)$ through $\pi_\eta^N$ contains the algebra $\KM\big (\Ff_N^{(-)}(\Ll)\big)$ of compact operators over $\Ff_N^{(-)}(\Ll)$. As such, the commutant of this image is $\CM \cdot I$, hence the representation is irreducible.\qed

\begin{proposition}\label{Pro:GICARRep} Let $J,J' \in \Kk_n(\Ll)$. Then:
\begin{enumerate}[{\rm \ (i)}]
\item If $n >N$, then $a_J(\chi_J)^\ast \, a_{J'}(\chi_{J'})$ is sent to zero by $\pi_\eta^N$. 

\item If $n=N$, then
\begin{equation}\label{W1}
\pi_\eta^N\big (a_J(\chi_J)^\ast \, a_{J'}(\chi_{J'}) \big ) = |J, \chi_J\rangle \langle J', \chi_{J'}|.
\end{equation}

\item If $n<N$, then
\begin{equation}\label{W2}
\pi_\eta^N\big (a_J(\chi_J)^\ast \, a_{J'}(\chi_{J'}) \big ) = \sum_{\Gamma \in \Kk_{N-n}(\Ll)} \pi_\eta^N\big (a_J(\chi_J)^\ast \, n_\Gamma \, a_{J'}(\chi_{J'}) \big ).
\end{equation}
The sum can be restricted to those $\Gamma$ with $\Gamma \cap (J \cup J') = \emptyset$ and the rule~\eqref{W1} can be applied to the expression~\eqref{W2}, if desired.
\end{enumerate}
\end{proposition}

\begin{remark}{\rm The above expression can be put in its symmetric presentation~\eqref{Eq:BOp1}, which we omitted to write out because it is not particularly illuminating. Note that the sums in \eqref{Eq:SimpleRep} and \eqref{W2} contain an infinite number of terms, with coefficients that do not decay to zero. As such, the Fock representations of the monomials are, in general, not by compact operators. This is a major difference between the CAR algebra and the algebra of local observables studied in section~\ref{Sec:SingleFermion}. 
}$\Diamond$
\end{remark}

 \begin{figure}[t]
\center
\includegraphics[width=0.8\textwidth]{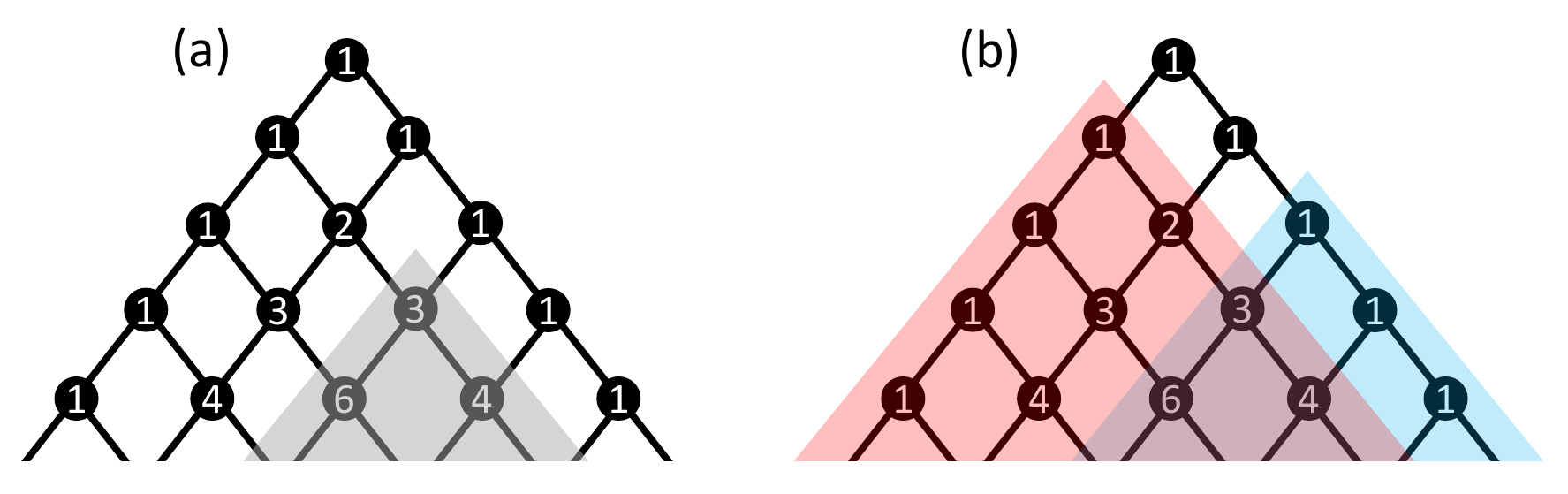}\\
  \caption{\small (a) The Bratteli diagram of the ${\rm GICAR}$ is the Pascal triangle. Its directed and hereditary subsets are triangles like the shaded one, corresponding to the ideal $_2{\rm GI}_1$. (b) The triangles shaded in red and blue corresponds to the primitive ideals $_0{\rm GI}_1$ and $_2{\rm GI}_0$, respectively. Their intersection give $_2{\rm GI}_1$, as it is apparent from the figure.}
 \label{Fig:Bratteli}
\end{figure}

\subsection{${\rm GICAR}$ is a solvable $C^\ast$-algebra}
\label{Sec:SolvableCAR} 

We start with a brief review of the work by Bratteli \cite{BratteliTAMS1972} on the global algebraic structure of the ${\rm GICAR}$ $C^{*}$-algebra. The latter was completely solved from the Bratteli diagram, which for ${\rm GICAR}$ is the Pascal triangle $\Pp$, reproduced in Fig.~\ref{Fig:Bratteli}. We recall that the double-sided ideals of ${\rm GICAR}$ are in one-to-one correspondence with the directed hereditary subsets of $\Pp$ \cite[Th.~III.4.2]{DavidsonBook} and, for the Pascal triangle, these sub-sets are the triangles described in  Fig.~\ref{Fig:Bratteli}. From this, the lattice of ideals can be easily derived: The ideal immediately below two ideals is given by the intersection of their corresponding triangles, while the ideal immediately above is given by the smallest triangle that covers both initial triangles. The particular triangles with the tip on the very left edge of $\Pp$, such as the one highlighted in red in Fig.~\ref{Fig:Bratteli}(b), form a tower of primitive ideals ${}_0{\rm GI}_N$, $N \in \NM$, associated with the Fock representations $\pi^{N}_{\eta}$. For orientation, $_0{\rm GI}_0$ coincides with ${\rm GICAR}$ itself. The triangles with the tip on the very right edge of $\Pp$, such as the one highlighted in blue in Fig.~\ref{Fig:Bratteli}(b), form another tower of primitive ideals $_N{\rm GI}_0$, $N \in \NM$, associated with the anti-Fock representations. According to \cite[Prop.~5.6]{BratteliTAMS1972}, ${}_N{\rm GI}_0$ and ${}_0{\rm GI}_N$ are the only primitive ideals of the ${\rm GICAR}$ algebra. Furthermore, any other ideal is the intersection of two such primitive ideals, ${}_M {\rm GI}_N = {}_M{\rm GI}_0 \cap  {} _0{\rm GI}_N$, as exemplified in Fig.~\ref{Fig:Bratteli}(b).

Our focus will be exclusively on the filtration supplied by the tower of ideals ${}_0{\rm GI}_N$, which will be denoted by ${\rm GI}_N$ from now on. Using Bratteli's diagram again, one can immediately see that ${\rm GI}_N/{\rm GI}_{N+1}$, whose diagram is given by the subtraction of the corresponding triangles \cite[Th.~III.4.4]{DavidsonBook}, is isomorphic to the algebra of compact operators \cite[Prop.~5.6]{BratteliTAMS1972}. This tells us that ${\rm GICAR}$ is a solvable $C^\ast$-algebra in the sense of \cite{DyninPNAS1978} and this particular algebraic structure will play an essential role in our program. In the rest of the section, we supply a more detailed account of the statements made above and supply alternative proofs that rely mostly on the commutation relations, hence accessible to a readership un-familiar with Bratteli diagrams.

\begin{proposition}[\cite{BratteliTAMS1972}]\label{Pro:GICARIdeals} The linear sub-spaces ${\rm GI}_N(\Ll) \subset {\rm GICAR}(\Ll)$ spanned by the elements
\begin{equation}\label{Eq:AJ1}
A = \sum_{n \geq N} \tfrac{1}{n!} \sum_{J,J' \in \Kk_n(\Ll)} \sum_{\chi_J,\chi_J'}  c_{J,J'}\big(\chi_J,\chi_{J'}\big ) a_J(\chi_J)^\ast a_{J'}(\chi_{J'}),
\end{equation}
are two-sided ideals of ${\rm GICAR}(\Ll)$.
\end{proposition}
\proof Focusing on monomials from ${\rm GICAR}$, we have from Eq.~\eqref{Eq:Id3} that
\begin{equation}\label{Eq:Prod}
\begin{aligned}
 & a^\ast_{J_1}(\chi_{J_1}) a_{J'_1}(\chi_{J'_1})   a^\ast_{J_2}(\chi_{J_2}) a_{J'_2}(\chi_{J'_2}) = (-1)^\sigma\sum_{K \subseteq J'_1 \cap J_2} (-1)^{|K|}  \\
& \qquad \qquad  a^\ast_{J_1}(\chi_{J_1}) \big (  a^\ast_{J_2 \setminus J'_1} \,n_K \, a_{J'_1 \setminus J_2} \big ) a_{J'_2}(\chi_{J'_2}).
\end{aligned}
\end{equation}
The above product is zero unless $J_1 \cap (J_2 \setminus J'_1) = \emptyset$, $J'_2 \cap (J'_1 \setminus J_2) = \emptyset$ and $|J_1 | + |J_2 \setminus J'_1|=|J'_2| + |J'_1 \setminus J_2|$.  It is then clear that the symmetric presentation of the product~\eqref{Eq:Prod} involves only terms with $J$ and $J'$ satisfying the constraints 
\begin{equation*}
\max \{ |J_1|,|J_2|\} \leq |J|=|J'| \leq |J_1|+|J_2|.
\end{equation*}
The statement follows.\qed

\begin{proposition}[\cite{BratteliTAMS1972}]We have ${\rm GI}_{N+1}(\Ll) \simeq {\rm Ker} \, \pi_\eta^N$, for all $N \in \NM$.
\end{proposition}
\proof  Assume that $A \in {\rm GICAR}(\Ll)$ but $A \notin {\rm GI}_{N+1}(\Ll)$ and that $A$ is a finite sum of monomials. In this situation, we can always find a set $\Gamma \in \Kk(\Ll)$, $|\Gamma| \leq N$, such that $\Gamma \cap (J \cup J')=\emptyset$, for all pairs $(J,J')$ appearing in the standard presentation of $A$, and $n_\Gamma A \in {\rm GI}_{N}(\Ll)$ but $n_\Gamma A \notin {\rm GI}_{N+1}(\Ll)$. Then the product $n_\Gamma A$ can be brought to its symmetric presentation using Proposition~\ref{Pro:U10} and, given our previous conclusion, necessarily
\begin{equation*}
0 \neq \pi_\eta^N(n_\Gamma A) =  \pi_\eta^N(n_\Gamma)  \pi_\eta^N(A),
\end{equation*}
hence such element $A$ cannot belong to ${\rm Ker}\, \pi_\eta^N$.\qed

\begin{remark}{\rm According to the above result, the ideals identified in Proposition~\ref{Pro:GICARIdeals} are the primitive ideals corresponding to the representations $\pi_\eta^N$.
}$\Diamond$
\end{remark}

\begin{proposition}[\cite{BratteliTAMS1972}]\label{Pro:CARQuote} We have
\begin{equation}\label{Eq:QIso}
{\rm GI}_{N}(\Ll) / {\rm GI}_{N+1}(\Ll) \simeq \KM\big (\Ff^{(-)}_N \big ),
\end{equation}
with the isomorphism supplied by the descent of $\pi_\eta^N$ onto the quotient algebra.
\end{proposition}
\proof Let $[\cdot]_N$ denote the classes in the quotient space. Obviously, the quotient space is spanned by $\big [a^\ast_J(\chi_J) a_{J'}(\chi_{J'})\big ]_N$, $J,J' \in \Kk_N(\Ll)$, and the class of an element $A \in {\rm GI}_{N}(\Ll)$, presented as in Eq~\eqref{Eq:AJ1}, can be always represented as
$$
[A]_N = \tfrac{1}{N!} \sum_{J,J' \in \Kk_N(\Ll)} \sum_{\chi_J,\chi_J'}  c_{J,J'}\big(\chi_J,\chi_{J'}\big ) \big [a_J(\chi_J)^\ast a_{J'}(\chi_{J'})\big ]_N.
$$
We claim that the map
\begin{equation}\label{Eq:Iso1}
\begin{aligned}
 \tfrac{1}{N!}\sum_{J,J'\in \Kk_N(\Ll)} \sum_{\chi_J,\chi_{J'}} & c_{J,J'}(\chi_J,\chi_{J'}) [a^\ast_J(\chi_J) a_{J'}(\chi_{J'})]_N \\
& \quad \mapsto \tfrac{1}{N!}\sum_{J,J'\in \Kk_N(\Ll)}\sum_{\chi_J,\chi_{J'}} c_{J,J'} (\chi_J,\chi_{J'}) \, |J,\chi_J\rangle \langle J',\chi_{J'}|
\end{aligned}
\end{equation}
supplies the isomorphism in Eq.~\eqref{Eq:QIso}. Indeed, for two monomials, 
$$
\begin{aligned}
 \big [a^\ast_J(\chi_J) a_{J'}(\chi_{J'})\big ]_N & \, [ a^\ast_{\tilde J}(\chi_{\tilde J}) a_{\tilde J'}(\chi_{\tilde J'})\big ]_N \\
& \qquad \qquad = (-1)^{\chi_{J'}^{-1} \circ \chi_{\tilde J}} \delta_{J',\tilde J} [a^\ast_J(\chi_J) a_{\tilde J'}(\chi_{\tilde J'})]_N,
\end{aligned}
$$
which follows from the same type of calculations as in the proof of Proposition~\ref{Pro:GICARIdeals}. Using the same arguments as in the proof of Proposition~\ref{Pro:FProd}, we conclude that
$$
[A]_N \, [\tilde A]_N = \sum_{J,J'\in \Kk_N(\Ll)} \tfrac{1}{N!} \sum_{\chi_J,\chi_{J'}} c'_{J,J'}(\chi_J,\chi_{J'}) [a^\ast_J(\chi_J) a_{J'}(\chi_{J'}]_N,
$$
where
$$
\begin{aligned}
c'_{J,J'}(\chi_J,\chi_{J'}) =  \sum_{K \in \Kk_N(\Ll)}\sum_{\chi_K} c_{J,K}(\chi_J,\chi_K) \tilde c_{K,J'}(\chi_K,\chi_{J'}).
\end{aligned}
$$
 For the second statement, we recall that $\pi_\eta^N$ sends the entire ${\rm GI}_{N+1}(\Ll)$ to zero. Hence, $\pi_\eta^N$ descends on the quotient space and Proposition~\ref{Pro:GICARRep}(ii) shows that this descent coincides with the map~\eqref{Eq:Iso1}.\qed
 
\section{Interacting Fermions: Dynamics}
\label{Sec:FermionsDynamics}

This section formalizes the dynamics of the gauge invariant local physical observables. Our main goal here is to identify a large enough set of derivations that can serve as the core of an algebra associated with the dynamics of the fermions. The starting point is the class of derivations with finite interaction range, which are defined over and return values in a fixed dense subalgebra $\Dd(\Ll) \subset {\rm CAR}(\Ll)$. These derivations can be composed with each other, hence they generate a subalgebra of ${\rm End}(\Dd(\Ll))$. Subsequently, we explore the dependence of the derivations on the pattern $\Ll$ and the constraints imposed by the (assumed) Galilean invariance. For guidance, we introduce a large class of Hamiltonians satisfying these constraints, inspired from the physics literature and generated from many-body potentials. In the process, we demonstrate that the lattice deformations and the fermion permutations cannot be separated. This leads us to the construction of the many-body covers of the space of Delone sets, which supply the natural and rightful domain for the Hamiltonian coefficients. It also enables us to give a precise formulation of a core $\ast$-algebra of derivations that can be actually mapped from real experiments. We show that this algebra admits representations by uniformly bounded operators on the Fock sectors with finite number of particles.

\subsection{Hamiltonians with finite interaction range}
\label{Sec:FIRHam}

We consider a strongly continuous dynamics of the local physical observables 
\begin{equation*}
\bm \alpha : \RM \to {\rm Aut}\big({\rm CAR}(\Ll)\big)
\end{equation*} 
and denote by $\delta_{\bm \alpha}$ the generator of this dynamics. In the majority of the studied physical systems, the domain of $\delta_{\bm \alpha}$ is
\begin{equation*}
{\rm Dom}(\delta_{\bm \alpha})=\Dd(\Ll) : = \mathop{\cup}_k {\rm CAR}(\Ll_k),
\end{equation*}
where $\{\Ll_k\}$ is the net of finite lattices used in section~\ref{SubSec:CARandGICAR} to define ${\rm CAR}(\Ll)$. In the above conditions, $\delta_{\bm \alpha}$ is an inner-limit derivation \cite[p.~26]{BratteliBook1} and, throughout, we will restrict to these cases.  We recall that we are seeking a core for the algebra of generators, which subsequently can be closed in many different ways (see section~\ref{Sec:Filtration}).  It is natural to start from the inner-limit derivations because the laboratory reality is that a finite team of experimenters can only map this type of generators in a finite amount of time.

\begin{remark}{\rm Note that $\Dd(\Ll)$ is not just a linear space but a dense subalgebra of ${\rm CAR}(\Ll)$, closed under the $\ast$-operation. Obviously, $\Dd(\Ll)$ is not the whole ${\rm CAR}(\Ll)$ and, in general, $\delta_{\bm \alpha}$'s are not uniformly bounded over this domain. This is another fundamental qualitative difference between interacting fermion systems and the systems studied in section~\ref{Sec:SingleFermion}, where the generators were bounded. 
}$\Diamond$
\end{remark}
 
Derivations are linear maps over $\Dd(\Ll)$ and a derivation that leaves $\Dd(\Ll)$ invariant is an element of the algebra ${\rm End}\big (\Dd(\Ll)\big )$ of linear maps over $\Dd(\Ll)$. A large class of such derivations is supplied by finite interaction range Hamiltonians:
 
\begin{definition}\label{Def:FRHam} A finite interaction range Hamiltonian is a formal sum
\begin{equation}\label{Eq:HOriginal0}
H_\Ll = \sum_{J,J' \in \Kk(\Ll)}\tfrac{1}{\sqrt{|J|!|J'|!}} \sum_{\chi_J,\chi_{J'}} h_{J,J'}^\Ll(\chi_J,\chi_{J''}) \, a^\ast_J(\chi_J) a_{J'}(\chi_{J'}),
\end{equation} 
where the coefficients $h^\Ll_{J,J'}$ are bi-equivariant, uniformly bounded, obey the constraints 
\begin{equation}\label{Eq:SA}
h^\Ll_{J,J'}(\chi_J,\chi_{J'}) = \overline{h^\Ll_{J',J}(\chi_{J'},\chi_J)},
\end{equation}
and they vanish whenever the diameter of $J \cup J'$ exceeds a fixed value ${\rm R}_{\rm i}$, called here the interaction range. Also, the sum in Eq.~\eqref{Eq:HOriginal} runs only over pairs of sub-sets with $|J \cup J'|$ even.
\end{definition}

\begin{remark}{\rm The expression~\eqref{Eq:HOriginal0} is quite involved because we need to keep track of the essential dependencies. Of course, if the lattice is fixed and a global orientation is chosen, then the notation can be simplified, but this is not at all the case here. Still, the notation will be simplified after we introduce the many-body covers of the space of Delone sets (see Remark~\ref{Re:Not1}).
}$\Diamond$
\end{remark}

\begin{remark}{\rm For reader's convenience, we recall that the diameter of a subset $\Delta$ of a metric space $(X,{\rm d})$ is the real number defined by
\begin{equation*}
d_\Delta := \sup \big\{{\rm d}(x,y), \, x,y \in \Delta \big \}.
\end{equation*} 
In Definition~\ref{Def:FRHam}, the metric space is $\RM^d$ with its Euclidean distance. 
}$\Diamond$
\end{remark}

\begin{remark}{\rm Note that in Definition~\ref{Def:FRHam}, while the diameter of $J \cup J'$ is forced to be finite, there are no constraints on how far the set $J \cup J'$ can be from the origin. As such, the the formal series~\eqref{Eq:HOriginal0} does not belong to ${\rm CAR}(\Ll)$, in general.
}$\Diamond$
\end{remark}

\begin{proposition}\label{Pro:Derivation} Let $\{\Ll_k\}$ be the tower of subsets used to define ${\rm CAR}(\Ll)$ in section~\ref{SubSec:CARandGICAR}, and let $H_{\Ll_k}$ be the truncation of Eq.~\eqref{Eq:HOriginal0} to $J,J'\in \Kk(\Ll_k)$. Then
\begin{equation}\label{Eq:AdDef}
{\rm ad}_{H_\Ll}(A) : =  \lim_{k \rightarrow \infty} \imath [A,H_{\Ll_k}] , \quad A \in \Dd(\Ll),
\end{equation}
is a derivation that leaves $\Dd(\Ll)$ invariant. Furthermore,
\begin{equation*}
{\rm ad}_{H_\Ll}(A)^\ast = {\rm ad}_{H_\Ll}(A^\ast), \quad \forall \ A \in \Dd(\Ll),
\end{equation*}
hence ${\rm ad}_{H_\Ll}$ is a $\ast$-derivation.
\end{proposition}
\proof If $A \in \Dd(\Ll)$, then necessarily $A \in {\rm CAR}(\Ll_p)$ for some $p \in \NM$. Let $\Ll_k$ be the smallest  subset from the tower which includes all the pairs $(J,J')$ with $d_{J \cup J'}\leq {\rm R}_{\rm i}$ and $(J \cup J') \cap \Ll_p \neq \emptyset$. Since $|J \cup J'|$ is assumed to be even, we have $[a^\ast_K a_{K'},a^\ast_Ja_{J'}]=0$ for any pair $(K,K')$ with $(K \cup K') \cap (J \cup J') = \emptyset$. As a result, 
$$
[A, H_{\Ll_{k'}}] = [A, H_{\Ll_{k}}] \in {\rm CAR}(\Ll),
$$
for all $k' >k$, hence the limit in Eq.~\eqref{Eq:AdDef} exists and, in fact,
\begin{equation}\label{Eq:C1}
{\rm ad}_{H_\Ll}(A) = \imath [A, H_{\Ll_{k}}] \in {\rm CAR}(\Ll_k).
\end{equation}
The above shows that ${\rm ad}_{H_\Ll}$ is well defined over $\Dd(\Ll)$ and takes values in $\Dd(\Ll)$. Regarding the last statement, we have
\begin{equation*}
{\rm ad}_{H_\Ll}(A)^\ast = - \imath ( H_{\Ll_k}^\ast A^\ast - A^\ast H_{\Ll_k}^\ast).
\end{equation*}
With the stated assumptions on the Hamiltonian coefficients, $H_{\Ll_k}^\ast = H_{\Ll_k}$, hence,
\begin{equation*}
{\rm ad}_{H_\Ll}(A)^\ast = {\rm ad}_{H_{\Ll_k}}(A)^\ast = {\rm ad}_{H_{\Ll_k}}(A^\ast) = {\rm ad}_{H_\Ll}(A^\ast)
\end{equation*}
and the statement follows.\qed

\vspace{0.2cm}

\begin{definition}\label{Def:GIHam} A gauge invariant Hamiltonian, abbreviated as GI-Hamiltonian, is a finite-range Hamiltonian as in Definition~\ref{Def:FRHam} with the first sum constrained on sub-sets of equal cardinality. Hence, the formal expression of such a Hamiltonian can be organized as
\begin{equation}\label{Eq:HOriginal}
H_\Ll = \sum_{n \in \NM^\times} \tfrac{1}{n!} \sum_{J,J' \in \Kk_n(\Ll)}  \sum_{\chi_J,\chi_{J'}} h_{J,J'}^\Ll(\chi_J,\chi_{J'}) \, a^\ast_J(\chi_J) a_{J'}(\chi_{J'}).
\end{equation} 
\end{definition}

\begin{remark}{\rm The derivations ${\rm ad}_{H_\Ll}$ corresponding to GI-Hamiltonians commute with the gauge transformations on $\Dd(\Ll) \subset {\rm CAR}(\Ll)$. Our analysis will be restricted from now on to GI-Hamiltonians. Note that, in this case, $|J \cup J'|$ is automatically an even number.
}$\Diamond$
\end{remark}

We recall the unique faithful tracial state $\Tt$ of the ${\rm CAR}$ algebra. Since ${\rm ad}_{H_\Ll}$ are almost-inner, it follows automatically that $\Tt \circ {\rm ad}_{H_\Ll} =0$ on $\Dd(\Ll)$. Then, according to Corollary~1.5.6 in \cite{BratteliBook1}, the derivation ${\rm ad}_{H_\Ll}$ is closable. Furthemore, $\Dd(\Ll)$ is actually a set of analytic elements for ${\rm ad}_{H_\Ll}$, hence ${\rm ad}_{H_\Ll}$ is a pre-generator of a 1-parameter group $\bm \alpha_t$ of $\ast$-automorphisms. Clearly, $\Tt \circ \bm \alpha_t = \Tt$. The GI-Hamiltonians also have an intrinsic relation with the vacuum state:

\begin{proposition}\label{Pro:EtaAd} We have $\eta \circ {\rm ad}_{H_\Ll} =0$ on $\Dd(\Ll)$, for any $H_\Ll$ as in Definition~\ref{Def:GIHam}.
\end{proposition}

\proof We will use the notation from the proof of Proposition~\ref{Pro:Derivation}. If $A \in {\rm CAR}(\Ll_p)$ for some $p \in \NM$, then ${\rm ad}_{H_\Ll}(A) = \imath [A,H_{\Ll_k}]$ for some $k \geq p$. When evaluating $\eta$ on this commutator, the terms of the symmetric presentation of $A$ that are not gauge invariant can be ignored. We can then assume that $A$ is gauge invariant. Now, $H_{\Ll_k}$ belongs to the ideal ${\rm GI_1}(\Ll)$, hence $\imath [A,H_{\Ll_k}]$ also belongs to this ideal. As such, its coefficient $c_{\emptyset,\emptyset}$ is null and the statement follows.\qed

\vspace{0.2cm}

All the above are well established facts \cite{BratteliBook2} and the inner-limit derivations of the ${\rm CAR}$-algebra have been completely characterized in \cite{Jorgensen1991} using its embedding into the Cuntz algebra. Our interest, however, goes well beyond the individual derivations, as already stressed on several occasions.

\subsection{Galilean invariant theories}
\label{SubSec:GInvariance} As in section~\ref{Sec:GalInvHam1}, let us imagine an experimenter sitting at the origin of the physical space $\RM^d$ and studying a system of fermions over the lattice $\Ll$. From the dynamics of local observables, the experimenter maps the Hamiltonian $H_\Ll$, which is equivalent to the mapping of the equivariant coefficients $h_{J,J'}^\Ll$. The same experimenter then deforms the lattice, without changing the nature of the fermions or resonators, and maps again the Hamiltonian. After repeating this program for many Delone sets, the experimenter establishes a map 
\begin{equation}\label{Eq:Map0}
{\rm Del}_{(r,R)}(\RM^d) \ni \Ll \mapsto H_\Ll,
\end{equation} 
which is entirely determined by the nature of the fermions. The notation $H_\Ll$ can be interpreted as the evaluation of a global Hamiltonian $H$ at $\Ll$ and this is the point of view we will adopt from now on. Throughout, our working assumption is that $H_\Ll$'s in  Eq.~\ref{Eq:Map0} have finite interaction ranges.

The formalism faces two immediate challenges. The first one is how to properly define a continuity property for the maps in Eq.~\ref{Eq:Map0}. The second one is understanding the constraints imposed by the assumed Galilean invariance of the theory. The latter is investigated below, while the first challenge is addressed in the following subsections. For now, we will adopt the view that the map~\eqref{Eq:Map0} has been generated by the experiment. Now, during the lengthy experimental process we just described, two lattices $\Ll$ and $\Ll'$ may happen to enter the relation $\Ll' = \Ll-x$ for some $x \in \RM^d$. While the experimenter is pinned at the origin of the physical space at all times, just for this situation, we can imagine the experimenter and the lattice $\Ll'$ being rigidly shifted until the experimenter sits at position $x$. Then the experimenter is dealing with the same lattice $\Ll$ but from a difference location and, in the absence of background fields, Galilean invariance of the physical processes involved in the resonator couplings assures us that, up to a proper relabeling, the new experiments return the same coefficients. Therefore, the Hamiltonian coefficients must be subject to the following relations
\begin{equation}\label{Eq:Galilean0}
h_{J-x,J'-x}^{\Ll-x}\big (\mathfrak t_x \circ \chi_J, \mathfrak t_x \circ \chi_{J'} \big )= h_{J,J'}^\Ll \big (\chi_J,\chi_{J'}\big ),
\end{equation}
for all $\Ll \in {\rm Del}_{(r,R)}(\RM^d)$ and $J,J' \in \Kk(\Ll)$. These relations can be expressed more concisely, as already explained in section~\ref{Sec:Introduction} (see Eq.~\eqref{Eq:GInv}). Let us specify that Remark~\ref{Re:Galilean} applies here as well.

\begin{remark}{\rm Let us acknowledge that, among other things, Eq.~\eqref{Eq:Galilean0} implies
\begin{equation}\label{Eq:ReducedCoef}
h_{J,J'}^\Ll \big (\chi_J,\chi_{J'}\big ) = h_{J-x_1,J'-x_1}^{\Ll-x_1}\big (\mathfrak t_{x_1} \circ \chi_J, \mathfrak t_{x_1} \circ \chi_{J'} \big ), \quad x_1 =\chi_J(1).
\end{equation}
As one can see, the experimenter can archive the entire map $\Ll \mapsto H_\Ll$ by measuring just the coefficients $h_{J,J'}^{\Ll}$ with $J$ and $J'$ inside a ball of radius ${\rm R}_{\rm i}$ and centered at the origin of the physical space. In other words, by local experiments! Of course, the experimenter will need to sample many patterns $\Ll$ and it is at this point where the continuity of the coefficients w.r.t. $\Ll$ is essential. This is because it enables the experimentalist to extrapolate (aka connect the dots) the results generated by a finite number of observations.
}$\Diamond$
\end{remark}


So far, the coefficients of the Hamiltonians exist only in the tables generated by the experimenter. In the following, we describe an analytic method to generate Galilean invariant Hamiltonians that is often found in the physics literature. This class of Hamiltonians will serve as a stepping stone for our quest of the most general expression of global Hamiltonians displaying finite interaction range, Galilean invariance and continuity w.r.t. the underlying lattice. 

It is instructive to start with a simple example, which actually represents the most common many-body Hamiltonian found in the physics literature:

\begin{example}\label{Ex:Model1}{\rm The Hamiltonian of a system of fermions over a discrete lattice $\Ll$ and interacting pair-wise via a given potential $v: \RM^d \rightarrow \RM$ takes the form
\begin{equation}\label{Eq:HL5}
H_\Ll = \sum_{x,x' \in \Ll} w_{x,x'}(\Ll)\,  a_x^\ast a_{x'} + \sum_{x,x' \in \Ll} v(x'-x) \,  a_{x'}^\ast a_x^\ast a_{x'}  a_x,
\end{equation}
where $w$'s are as in section ~\ref{Sec:SingleFermion}. The Hamiltonian has a finite interaction range if the potential $v$ has compact support.
}$\Diamond$\
\end{example}

\begin{remark}\label{Re:2Body}{\rm Example~\ref{Ex:Model1} is exceptional in several ways. Firstly, note that it involves terms where $J$ and $J'$ either contain just one point or they coincide. Because of this particularity, the ordering of the local observables is in fact irrelevant. Secondly, note that, when $x=x'$ in the second sum, the summands cancel. As such, the summation can be restricted to pairs with $x \neq x'$. Since these pairs belong to a Delone set, they never come closer than a distance $r$. As such, the potential $v$ can be multiplied by the function $(1-\phi)$, with $\phi$ a continuous function with support inside $B(0,r)$,  without producing any modifications to the Hamiltonian. This last remark will become relevant for the discussion in Remark~\ref{Re:Odd}.
}$\Diamond$
\end{remark}
To define more general Hamiltonians, one needs to supply a whole family of bi-equivariant coefficients that are indexed by subsets of $\Ll$ and, preferably, depend continuously on $\Ll$. This non-trivial task can be accomplished with the help of many-body potentials. For this, we consider the space $(\RM^d)^{n} \times (\RM^d)^{n}$, together with the natural action of $\RM^d$ induced by the diagonal translations,
 \begin{equation*}
 \mathfrak t_x (x_1, \ldots, x_n;x'_1, \ldots, x'_{n}) = (x_1-x, \ldots, x_n-x;x'_1-x, \ldots, x'_{n}-x),
\end{equation*}  
and of the group $\mathcal{S}_n \times \mathcal{S}_n$, which permutes the first and last $n$-variables, separately. Let $\big [(\RM^d)^{2n}\big ]$ be the space of orbits for the $\RM^d$ action, equipped with the quotient topology. Since the permutations and the diagonal translations commute, the action of $\mathcal{S}_n \times \mathcal{S}_n$ descends to an action on $\big [(\RM^d)^{2n}\big ]$.

\begin{definition}\label{Def:KPot} We call $\hat w_{n} : \big [(\RM^d)^{2n}\big ] \to \CM$ a seed for a bi-equivariant $n$-body potential if $\hat w_{n}$ is continuous, has compact support and is odd relative to the action of $\mathcal{S}_n \times \mathcal{S}_n$. The $n$-body potential associated to such a seed is the function
\begin{equation*}
w_{n} = \hat w_{n} \circ {\rm q}_n : (\RM^d)^{2n} \to \CM,
\end{equation*}
where ${\rm q}_n: (\RM^d)^{2n} \to \big [(\RM^d)^{2n}\big ]$ is the quotient map. In addition, we assume that
\begin{equation*}
w_n(x_1, \ldots, x_n;x'_1, \ldots, x'_{n}) = \overline{w_n(x'_1, \ldots, x'_n;x_1, \ldots, x_{n})}.
\end{equation*}
\end{definition}

\begin{definition}\label{Def:Ham} Assume that a system $\{w_{n}\}_{n \in \NM^\times}$ of many-body potentials  has been supplied. We then declare that the evaluation at $\Ll$ of the global Hamiltonian corresponding to these potentials is
\begin{equation}\label{Eq:PotentialH}
\begin{aligned}
H_\Ll  = \sum_{n \in \NM^\times} \tfrac{1}{n!} \sum_{J,J' \in \Kk_n(\Ll)} \ \sum_{\chi_J,\chi_{J'}} \, & w_{J,J'}^{\Ll}(\chi_J,\chi_{J'})  a^\ast_{J}(\chi_J) a_{J'}(\chi_{J'}),
\end{aligned}
\end{equation}
where the bi-equivariant coefficients are given by
\begin{equation}\label{Eq:PotCoeff}
w_{J,J'}^{\Ll}\big (\chi_J,\chi_{J'}\big ) := w_{n}\big (\chi_J(1),\ldots,\chi_J(n);\chi_{J'}(1),\ldots \chi_{J'}(n)\big),
\end{equation} 
for $J,J' \in \Kk_n(\Ll)$.
\end{definition}

\begin{proposition} The Hamiltonians~\eqref{Eq:PotentialH} are manifestly Galilean invariant.
\end{proposition}
\proof Indeed, the coefficients obey relation~\eqref{Eq:Galilean0} because the potentials $w_n$ are constant along the orbit induced by the diagonal translations. \qed 

\begin{example}\label{Ex:TwoBody}{\rm For a $2$-body potential, Eq.~\eqref{Eq:PotentialH} takes the form
\begin{equation}\label{Eq:2body}
\begin{aligned}
H_\Ll = 2\sum_{x_1,x_2;x'_1,x'_2 \in \Ll} \;  w_{2}(x_1,x_2;x'_1,x'_2) \; a_{x_1}^\ast a_{x_2}^\ast \, a_{x'_1} a_{x'_2} .
\end{aligned}
\end{equation}
Consider now the following specific many-body potential 
\begin{equation}\label{Eq:C61}
w_{2}(x_1,x_2;x'_1,x'_2) =  v_1(x_2-x_1)v_2(x'_2-x'_1) \varphi\big ({\rm d}_{H}(\{x_1,x_2\},\{x'_1,x'_2\})\big ),
\end{equation}
where the two functions $v_1$ and $v_2$ are odd, $v_i(-x) = -v_i(x)$, and $2 v_1 v_2=v$, the potential from Example~\ref{Ex:Model1}. Also, $\varphi:\RM \rightarrow \RM$ is a function with $\varphi(0) = 1$ and with support in the interval $[0,r]$. Note that, if $\Ll \in {\rm Del}_{(r,R)}(\RM^d)$, then the last factor in Eq.~\ref{Eq:C61} takes non-zero values if only if $\{x_1,x_2\}=\{x'_1,x'_2\}$, as un-ordered sets. Then Eq.~\eqref{Eq:2body} reduces to the model Hamiltonian~\eqref{Ex:Model1}. Clearly, the potential~\eqref{Eq:C61} is constant along the orbit induced by the diagonal translations of the points, hence it descends to a seed on $\big [(\RM^d)^{4}\big ]$, which can be checked to be continuous.
}$\Diamond$
\end{example}

\begin{remark}\label{Re:Odd}{\rm As is always with the case for odd functions, $v_i(0)=0$ and, as such, $v(0)=0$. However, given the discussion in Remark~\ref{Re:2Body}, this does not restrict the generality of the above construction.
}$\Diamond$
\end{remark}

\begin{example}{\rm Explicit models of many-body potentials corresponding to various fractional Hall sequences can be found in \cite{ProdanPRB2009}.
}$\Diamond$
\end{example}

\begin{remark}\label{Re:FiniteR}
{\rm For finite interaction range Hamiltonians evaluated on a fixed class ${\rm Del}_{(r,R)}(\RM^d)$ of Delone sets, there is an upper limit on the cardinals of the sub-sets $J$, $J'$ entering in Eq.~\eqref{Eq:PotentialH}. As such, a finite interaction range Hamiltonian as in Eq.~\eqref{Eq:PotentialH} is always generated from a finite number of many-body potential seeds. Also, if the supports of these many-body potentials are contained inside the ball of radius ${\rm R}_{\rm i}$, then the range of the Hamiltonian~\eqref{Eq:PotentialH} is ${\rm R}_{\rm i}$.
}$\Diamond$
\end{remark} 

We recall that our ultimate goal is to define the most general class of finite interaction range Hamiltonians that are Galilean invariant and depend continuously on the pattern $\Ll$. Since we do not yet have a topology on the algebra of derivations, at this point, the only available option is to formulate the continuity directly on the bi-equivariant coefficients. Clearly, the Hamiltonians constructed from many-body potentials provide guidance in this respect. However, even for these explicitly constructed Hamiltonians, there is still a difficulty in assessing the continuity of the coefficients~\eqref{Eq:PotCoeff}, because they are also functions of the orderings. As elements of $\mathcal{S}(\Ii_{|J|} , J)$, $J \subset \Ll$, these orderings change with the deformation of the pattern $\Ll$ and they might not return to their initial value when $\Ll$ is deformed back to its initial configuration. As such, there is a non-trivial topological space where the data $(\Ll, V, \chi_V)$ lives and this is investigated next.

\subsection{The many-body covers of the Delone  space}\label{Sec:Cover} We start by looking at the partial data $(\Ll,V)$, $|V|=n$, which generates the set
$$
{\rm Del}_{(r,R)}^{(n)}(\RM^d)=\Big \{(\Ll,V), \ \Ll \in {\rm Del}_{(r,R)}(\RM^d), \ V \in \Kk_n(\Ll) \Big \}.
$$
It is a subset of ${\rm Del}_{(r,R)}(\RM^d) \times \Kk(\RM^d)$, hence it can be naturally equipped with the relative topology inherited from ${\rm Del}_{(r,R)}(\RM^d) \times \Kk(\RM^d)$, with the latter endowed with the product topology. The following statements describe this topology more concretely.

\begin{proposition} The family of subsets
\begin{equation}\label{Eq:TopoBase3}
\begin{aligned}
U_M^\epsilon(\Ll,V) =  \Big ( U_M^\epsilon(\Ll)\times \mathring B(V,\epsilon) \Big ) \cap {\rm Del}_{(r,R)}^{(n)}(\RM^d)
\end{aligned}
\end{equation}
form a neighborhood base for ${\rm Del}_{(r,R)}^{(n)}(\RM^d)$. Here $\epsilon > 0$, $M $ is positive and large enough such that any $V' \in \mathring B(V,\epsilon)$ is contained in the ball $B(0,M)$ of $\RM^d$, and $ U_M^\epsilon(\Ll)$ is as defined in Eq.~\eqref{Eq:TopoBase2}.
 \end{proposition}
 
\proof According to the definition, the following family of sub-sets
\begin{equation}\label{Eq:TopoBase4}
\begin{aligned}
U_M^{\epsilon,\epsilon'}(\Ll,V) = \Big ( U_M^\epsilon(\Ll)\times \mathring B(V,\epsilon') \Big ) \cap {\rm Del}_{(r,R)}^{(n)}(\RM^d),
\end{aligned}
\end{equation}
with $\epsilon,\epsilon'>0$ and $M>0$, is a basis of neighborhoods for the relative topology. Clearly 
\begin{equation*}
U_M^{\epsilon,\epsilon'}(\Ll,V) \subseteq U_{M'}^{\epsilon,\epsilon'}(\Ll,V), \quad \forall \  
 M'  \leq M,
 \end{equation*}
 and, as such, we can trim the family to the subsets $U_M^{\epsilon,\epsilon'}(\Ll,V)$ with $M$ larger than any threshold, without affecting the topology. In particular, we can restrict to $M$'s satisfying the stated constraints.  Furthermore, 
\begin{equation*}
U_M^{\underline \epsilon,\underline \epsilon}(\Ll,V) \subseteq U_M^{\epsilon,\epsilon'}(\Ll,V), \quad \underline \epsilon = \min(\epsilon,\epsilon'),
\end{equation*}
hence the family can be further trimmed to members with $\epsilon = \epsilon'$.\qed 

\begin{remark}\label{Re:LM}{\rm Note that, for any $(\Ll',V') \in U_M^\epsilon(\Ll,V)$, we have $V' \in \Ll'[M]$.
}$\Diamond$
\end{remark}

\begin{proposition}\label{Pro:G} Let $(\Ll',V')$ be any point from the open neighborhood $U_M^\epsilon(\Ll,V)$ of $(\Ll,V)$. Then, if $\epsilon< r/2$, there exists a canonical bijective map $g:V \to V'$.
\end{proposition}
\proof Since ${\rm d}_{\rm H}(V',V) < \epsilon$, the open ball $\mathring B(v,\epsilon)\subset \RM^d$ necessarily contains at least one  point of $V'$, for each $v \in V$. Since both $\Ll$ and $\Ll'$ are from ${\rm Del}_{(r,R)}(\RM^d)$ and $\epsilon<r/2$, each $\mathring B(v,\epsilon)$ can contain at most one point of $\Ll'$. The conclusion is that each $\mathring B(v,\epsilon)$ contains exactly one point $v'$ of $V'$ and, in fact, of the whole $\Ll'$. The canonical map mentioned in the statement is defined to be the map corresponding to the graph 
\begin{equation*}
g= \{(v,v') \in V \times V', \ v' \in \mathring B(v,\epsilon)\}.
\end{equation*}
The map is clearly invertible.\qed

\begin{corollary}\label{Cor:Cover1} The map $${\rm Del}_{(r,R)}^{(n)}(\RM^d)\to{\rm Del}_{(r,R)}(\RM^d),\quad (\Ll,V) \mapsto \Ll,$$ is a local homeomorphism making ${\rm Del}_{(r,R)}^{(n)}(\RM^d)$ into an infinite cover of ${\rm Del}_{(r,R)}(\RM^d)$.
\end{corollary}

We now define the order cover of the space ${\rm Del}_{(r,R)}^{(n)}(\RM^d)$, which is the natural topological space for the data $(\Ll,V,\chi_V)$. The following statements give the definition and supply a characterization of this cover.

\begin{proposition}\label{Pro:Cover2} Consider the set
\begin{equation*}
\widehat{{\rm Del}}_{(r,R)}^{(n)}(\RM^d) = \Big \{\big(\Ll,V,\chi_V \big): \ (\Ll,V) \in {\rm Del}_{(r,R)}^{(n)}(\RM^d), \ \chi_V \in \mathcal{S}(\Ii_n , V)\Big \}.
\end{equation*} 
Then the subsets
\begin{equation}\label{Eq:BNeigh}
U_M^\epsilon(\Ll,V,\chi_V) = \Big \{(\Ll',V',g \circ \chi_V): \ (\Ll',V') \in U_M^\epsilon(\Ll,V)\Big \},
\end{equation}
with $M >>0$, $0 <\epsilon < r/2$ and $g$ the canonical bijection from Proposition~\ref{Pro:G}, form a neighborhood base for a topology on $\widehat{{\rm Del}}_{(r,R)}^{(n)}(\RM^d)$. The map
\[\widehat{{\rm Del}}_{(r,R)}^{(n)}(\RM^d)\to {\rm Del}_{(r,R)}^{(n)}(\RM^d),\quad (\mathcal{L},V,\chi_{V})\mapsto (\mathcal{L},V),\]
 is a local homeomorphism making $\widehat{{\rm Del}}_{(r,R)}^{(n)}(\RM^d)$ into an $n !$-cover of ${\rm Del}_{(r,R)}^{(n)}(\RM^d)$.
\end{proposition}

\proof Consider $(\Ll_1,V_1)$ and $(\Ll_2,V_2)$ from ${\rm Del}_{(r,R)}^{(n)}(\RM^d)$ and two overlapping neighborhoods $U_{M_i}^{\epsilon_i}(\Ll_i,V_i,\chi_{V_i})$, $i=1,2$. Let $(\Ll,V,\chi_V)$ be an element from their intersection and let $g_j:V_j \to V$, $j = 1,2$, be the corresponding canonical bijections defined as in Proposition~\ref{Pro:G}. Then necessarily
\begin{equation}\label{Eq:FsubV}
g_1 \circ \chi_{V_1} = g_2 \circ \chi_{V_2} = \chi_V,
\end{equation}  
which is a consequence of the very definition~\eqref{Eq:BNeigh} of the open neighborhoods.
Clearly, $(\Ll,V) \in U_{M_1}^{\epsilon_1}(\Ll_1,V_1) \cap U_{M_2}^{\epsilon_2}(\Ll_2,V_2)$ and, since the subsets defined in Eq.~\eqref{Eq:TopoBase3} form a neighborhood base, there exists a whole neighborhood $U_M^\epsilon(\Ll,V)$ contained in this intersection. 

We claim that
\begin{equation}\label{Eq:Inclusion}
U_M^\epsilon (\Ll,V,\chi_{V}) \subset U_{M_1}^{\epsilon_1}(\Ll_1,V_1,\chi_{V_1}) \cap U_{M_2}^{\epsilon_2} (\Ll_2,V_2,\chi_{V_2}),
\end{equation}
where the neighborhood on the left is as in Eq.~\eqref{Eq:BNeigh}, with $\chi_V$ as in Eq.~\ref{Eq:FsubV}. Indeed, take $(\tilde \Ll, \tilde V)$ from $U_M^\epsilon(\Ll,V)$ and let $\tilde g_j : V_j \to \tilde V$, $j=1,2$ and $\tilde g: V \to \tilde V$ be the corresponding canonical bijection defined as in Proposition~\ref{Pro:G}. Then $\tilde g_j = \tilde g \circ g_j$, $j=1,2$, and, as a consequence,
\begin{equation*}
\tilde g_j\circ\chi_{V_j} = \tilde g_j \circ \chi_{V_j} = \tilde g \circ g_j \circ \chi_{V_j} = \tilde g \circ \chi_V, \quad j = 1,2.
\end{equation*}
In other words,
\begin{equation*}
(\tilde \Ll,\tilde V,\tilde g \circ \chi_V) \in U_{M_1}^{\epsilon_1}(\Ll_1,V_1,\chi_{V_1}) \cap U_{M_2}^{\epsilon_2}(\Ll_2,V_2,\chi_{V_2}),
\end{equation*}
and the statement from Eq.~\eqref{Eq:Inclusion} follows. Lastly, if ${\rm y}_n $ is the projection $(\Ll,V,\chi_V) \mapsto (\Ll,V)$, then ${\rm y}_n^{-1}\big (U_M^\epsilon(\Ll,V) \big )$ consist of $n!$ disjoint open subsets of $\widehat{{\rm Del}}_{(r,R)}^{(n)}(\RM^d)$. \qed

\begin{proposition} The order cover is path connected.
\end{proposition} 
\proof Take a continuous path 
\begin{equation*}
[0,1] \ni t \xmapsto{\sigma} (\Ll_t,V_t) \in {\rm Del}_{(r,R)}^{(n)}(\RM^d),
\end{equation*} 
where just two points move and switch positions. Assume that these two points are inside $V_t$. Let $s : V \rightarrow V$ be the permutation that encodes the exchange of the two points. Then, $(\Ll,V,\chi_V)$ and $(\Ll,V,\chi_V \circ s)$ are  connected by ${\rm y}_n^{-1}(\sigma[0,1])$. By further exchanging points of $V$, we can connect $(\Ll,V,\chi_V)$ with any other $(\Ll,V,\chi'_V)$ from the order cover.\qed

\begin{remark}{\rm By combining the order cover and the map introduced in Corollary~\ref{Cor:Cover1}, we obtain the covering map 
$$
\mathfrak a_n : \widehat{{\rm Del}}_{(r,R)}^{(n)}(\RM^d) \to {\rm Del}_{(r,R)}(\RM^d), \quad \mathfrak a_n(\Ll,V,\chi_V) :=\Ll,
$$ 
which we call the $n$-body covering map of the space of Delone sets.
}$\Diamond$
\end{remark}

\begin{remark}\label{Re:Not1}{\rm We now have the means to simplify the notation. The elements of $\widehat{{\rm Del}}_{(r,R)}^{(n)}(\RM^d)$ will be denoted by one letter, such as $\xi$, $\zeta$, etc.. If $\xi = (\Ll,V,\chi_V)$ and if any of those three pieces of information needs to be specified, we will do so using $\Ll_\xi$, $V_\xi$ and $\chi_\xi$, as exemplified below.
}$\Diamond$
\end{remark}

The topological space $\widehat{{\rm Del}}_{(r,R)}^{(n)}(\RM^d)$ inherits two important group actions:

\begin{proposition} The group action of $\RM^d$ by translations, defined in Eq.~\eqref{Eq:RDGroupAction}, extends to a group action on the many-body covers by
\begin{equation*}
\hat {\mathfrak t}_x (\xi ) := \Big (\Ll_\xi-x, V_\xi-x, \mathfrak t_{x} \circ \chi_\xi \Big ).
\end{equation*}
\end{proposition}
\proof It is enough to establish the validity of the statement for $x$ in an open ball of the origin. For $|x|<r/2$, the canonical map $g:V_\xi \to \mathfrak t_x(V_\xi)$ coincides with the shift $V_\xi \mapsto V_\xi-x$ and, as such, $\hat{\mathfrak t}_x$ is a homeomorphism in such cases. \qed.

\begin{proposition}\label{Pro:Deck} If $s \in \mathcal{S}_n$ is a permutation, then
\begin{equation*}
s\mapsto\Lambda_{s},\quad \Lambda_s (\xi) = (\Ll_\xi,V_\xi,\chi_\xi \circ s^{-1})
\end{equation*}
defines an isomorphism of $\mathcal{S}_{n}$ onto  the group of deck transformations of the order cover
$\widehat{\rm Del}^{(n)}_{(r,R)}(\RM^d)\to {\rm Del}^{(n)}_{(r,R)}(\RM^d)$.
\end{proposition} 
\proof From its very definition, $\Lambda_s$ acts fiber-wise and $\Lambda$ is clearly a group action. Furthermore,
\begin{equation*}
\Lambda_s\big (U_M^\epsilon(\xi)\big) = U_M^\epsilon\big (\Lambda_s(\xi)\big),
\end{equation*}  
which means $\Lambda_s$ are continuous maps with continuous inverses, hence homeomorphisms. The map $s\mapsto \Lambda_{s}$ is injective and since the order of the group of deck transformations is at most $n!$, it is surjective as well.\qed

\subsection{The algebra of finite-range Galilean invariant derivations}
\label{Sec:AFRGI}

We are now ready to specify the most general Hamiltonians expected to come out of an actual laboratory investigation. The data entering the Hamiltonian coefficients generate the sets
\begin{equation}\label{Eq:KeySpace}
\widehat{{\rm Del}}_{(r,R)}^{(n,m)}(\RM^d)=\Big \{ \big (\xi,\zeta \big ) \in \widehat{{\rm Del}}_{(r,R)}^{(n)}(\RM^d) \times \widehat{{\rm Del}}_{(r,R)}^{(m)}(\RM^d), \ \Ll_\xi = \Ll_\zeta \Big \},
\end{equation}
where $n$ and $m$ sample $\NM^\times$. They can and will be naturally equipped with the topology inherited from $\widehat{{\rm Del}}_{(r,R)}^{(n)}(\RM^d) \times \widehat{{\rm Del}}_{(r,R)}^{(m)}(\RM^d)$. 

\vspace{0.2cm}

At this point, we introduce more efficient notation, which will be adopted throughout from now on:
\begin{enumerate}[\ $\circ$]
\item We define the correspondence
\begin{equation*}
\widehat{{\rm Del}}_{(r,R)}^{(n)}(\RM^d) \ni \xi \mapsto \bm a(\xi) \in {\rm CAR}(\Ll_\xi), \quad \bm a(\xi) = a_{V_\xi}(\chi_\xi),
\end{equation*}
for all $n \in \NM^\times$. It will be useful to append a point $\o$ and declare that $\bm a(\o)=0$.

\item For $(\xi,\zeta) \in \widehat{{\rm Del}}_{(r,R)}^{(n,m)}(\RM^d)$, we write 
\begin{equation*}
\xi \vee \zeta = (\Ll_\xi=\Ll_\zeta, \ V_\xi \cup V_\zeta, \ \chi_\xi \vee \chi_\zeta) \in \widehat{{\rm Del}}_{(r,R)}^{(n+m)}(\RM^d)
\end{equation*}
if $V_\xi \cap V_\zeta = \emptyset$, and $\xi \vee \zeta = \o$ otherwise.

\item For $(\xi,\zeta) \in \widehat{{\rm Del}}_{(r,R)}^{(n,m)}(\RM^d)$, we write $\zeta \leq \xi$ if $V_\zeta \subseteq V_\xi$ and $\chi_\xi=\chi_\zeta \vee \chi_\Gamma$, for some order $\chi_\Gamma$ of $\Gamma = V_\xi \setminus V_\zeta$.

\item For $(\xi,\zeta) \in \widehat{{\rm Del}}_{(r,R)}^{(n,m)}(\RM^d)$ such that $\zeta \leq \xi$, we denote by $\xi \setminus \zeta$ the unique element of $\widehat{{\rm Del}}_{(r,R)}^{(n-m)}(\RM^d)$ with the property that $\zeta \vee (\xi \setminus \zeta) = \xi$.

\item We introduce the covering map
$$
\mathfrak b_{n}:\widehat{{\rm Del}}_{(r,R)}^{(n,n)}(\RM^d) \to {\rm Del}_{(r,R)}(\RM^d), \quad \mathfrak b_n(\xi,\zeta) : = \Ll_\xi = \Ll_\zeta.
$$
\end{enumerate}

\vspace{0.2cm}

Since our focus is on GI-Hamiltonians, we will mainly be dealing with the spaces $\widehat{{\rm Del}}_{(r,R)}^{(n,n)}(\RM^d)$. Below, we list several obvious but important properties of these spaces.

\begin{proposition} The following statements hold:
\begin{enumerate}[{\rm \ 1)}]
\item The group $\RM^d$ acts on $\widehat{{\rm Del}}_{(r,R)}^{(n,n)}(\RM^d)$ by homeomorphisms via 
\begin{equation}\label{Eq:RdAction}
x \mapsto \hat{\mathfrak t}_x \times \hat{\mathfrak t}_x.
\end{equation}
This action is continuous, free and proper, and will be denoted by the same symbol $\hat{\mathfrak t}$ to ease the notation.
\item The map
\begin{equation}\label{Eq:Un}
\mathfrak u_{n} : \widehat{{\rm Del}}_{(r,R)}^{(n,n)}(\RM^d) \to \widehat{{\rm Del}}_{(r,R)}^{(n,n)}(\RM^d), \quad \mathfrak u_{n}(\xi,\zeta)=(\zeta,\xi),
\end{equation}
is a homeomorphism.
\item The map
\begin{equation*}
\mathfrak{v}_{n} : \widehat{{\rm Del}}_{(r,R)}^{(n,n)}(\RM^d) \to \RM^d, \quad \mathfrak{v}_{n}(\xi,\zeta)=\chi_\xi(1),
\end{equation*}
is continuous.

\end{enumerate}
\end{proposition}

As we have already seen in subsection~\ref{SubSec:GInvariance}, the space of orbits for the $\RM^d$ action supplies an effective instrument for defining coefficients with finite interaction range, which obey the constraints~\eqref{Eq:Galilean0} imposed by the Galilean invariance. Similarly:

\begin{definition}\label{Def:QuotientSp} We denote by $\big [ \widehat{{\rm Del}}_{(r,R)}^{(n,n)}(\RM^d) \big ]$ the space of orbits for the action~\eqref{Eq:RdAction} of $\RM^d$ and we equip this space with the quotient topology. We let  
\begin{equation*}
\hat {\rm q}_{n} : \widehat{{\rm Del}}_{(r,R)}^{(n,n)} (\RM^d) \to \big [\widehat{{\rm Del}}_{(r,R)}^{(n,n)} (\RM^d) \big ]
\end{equation*} 
be the quotient map, making $\widehat{{\rm Del}}_{(r,R)}^{(n,n)} (\RM^d)$ into a principal $\mathbb{R}^{d}$-bundle over $[\widehat{{\rm Del}}_{(r,R)}^{(n,n)} (\RM^d)]$.
\end{definition}

\begin{proposition} 
\label{Prop:leftright}
The rule
\begin{equation*}
s_1 \cdot [(\xi,\zeta)] \cdot s_2 = [(\Lambda_{s_1} \xi,\Lambda_{s_2}^{-1}\zeta)], \quad s_i \in \mathcal{S}_N,
\end{equation*}
defines commuting left and right actions of $\mathcal{S}_n$ on $\big [ \widehat{{\rm Del}}_{(r,R)}^{(n,n)}(\RM^d) \big ]$. These actions are by homeomorphisms.
\end{proposition}
\proof This is a consequence of the fact that the homeomorphisms $\Lambda_s$ and $\hat{\mathfrak t}_x$ all commute.\qed

\begin{definition}\label{Def:Z2Coeff} A seed for a continuous bi-equivariant $n$-coefficient is a compactly supported continuous function 
\begin{equation*}
\hat h_{n}: \big [\widehat{{\rm Del}}_{(r,R)}^{(n,n)} (\RM^d) \big ] \to \CM,
\end{equation*} 
such that
\begin{equation*}
\hat h_{n}\big (s_1 \cdot [(\xi,\zeta)]\cdot s_2 \big ) =  (-1)^{s_1}\,   \hat h_{n}\big ([(\xi,\zeta)] \big )\, (-1)^{s_2}
\end{equation*}
for all $s_1,s_2 \in \Ss_n$. The continuous bi-equivariant $n$-coefficient corresponding to such seed is the function
\begin{equation*}
h_{n} := \hat h_{n} \circ \hat{\rm q}_{n} : \widehat{{\rm Del}}_{(r,R)}^{(n,n)} (\RM^d)  \to \CM.
\end{equation*} 
\end{definition}

We are now in the position to formulate the most general class of physical Hamiltonians that can be mapped by an experimenter:

\begin{definition}\label{Def:H} The evaluation at $\Ll$ of the GI-Hamiltonian corresponding to a finite set $\{h_{n} = \overline{ h_{n} \circ \mathfrak u_{n}}\}$ of continuous bi-equivariant coefficients is
\begin{equation}\label{Eq:FinalH}
\begin{aligned}
H_\Ll = \sum_{n \in \NM^\times} \tfrac{1}{n!}  \sum_{(\xi,\zeta)\in \mathfrak b^{-1}_{n}(\Ll)}
  h_{n} (\xi,\zeta) \bm a^\ast(\xi) \bm a(\zeta).
\end{aligned}
\end{equation}
\end{definition}

\begin{remark}{\rm As we have seen in Proposition~\ref{Pro:Derivation}, the condition $h_{n} = \overline{ h_{n} \circ \mathfrak u_{n}}$, which is just Eq.~\eqref{Eq:SA} formulated in a proper setting, ensures that the derivations corresponding to $H_\Ll$ from Eq.~\eqref{Eq:FinalH} are in fact $\ast$-derivations.
}$\Diamond$
\end{remark}

We recall that derivations associated with these physical GI-Hamiltonians leave their common domain $\Dd(\Ll)$ invariant, hence they can be composed to generate the core algebra canonically associated with the dynamics of fermions on discrete lattices:

\begin{definition}\label{Def:GIHAlgebra} For $\Ll \in {\rm Del}_{(r,R)}(\RM^d)$, we define $\dot\Sigma(\Ll)$ to be the sub-algebra of ${\rm End}\big (\Dd(\Ll)\big )$ generated by derivations ${\rm ad}_{Q_\Ll^n}$ coming from continuous bi-equivariant coefficients,
\begin{equation}\label{Eq:FinalGIH}
\begin{aligned}
Q^n_\Ll = \tfrac{1}{n!}   \sum_{(\xi,\zeta)\in \mathfrak b^{-1}_{n}(\Ll)}
  q_{n} (\xi,\zeta) \, \bm a^\ast(\xi) \bm a(\zeta), \quad n \in \NM^\times.
\end{aligned}
\end{equation}
The elements of $\dot\Sigma(\Ll)$ can be presented as 
\begin{equation}\label{Eq:Qq}
\Qq = \sum_{\{Q\}} c_{\{Q\}} \ {\rm ad}_{Q_\Ll^{n_1}} \circ \cdots \circ {\rm ad}_{Q_\Ll^{n_k}}, \quad c_{\{Q\}} \in \CM,
\end{equation}
where the sum is over finite tuples of $Q$'s as in Eq~\eqref{Eq:FinalGIH} and only a finite number of $c$-coefficients are non-zero. The multiplication of two elements $\Qq_1$ and $\Qq_2$ is given by the composition $\Qq_1 \circ \Qq_2$ of linear maps over $\Dd(\Ll)$.
\end{definition}

\begin{remark}{\rm The presentation in Eq.~\ref{Eq:Qq} is not unique, hence there is little chance to derive any algebraic properties directly from this definition. As such, Definition~\ref{Def:GIHAlgebra} mostly communicates which elements of ${\rm End}\big (\Dd(\Ll)\big )$ are included in the sub-algebra $\dot\Sigma(\Ll)$. Note that the coefficients of a generic element $\Qq$ are no longer of finite range.
}$\Diamond$
\end{remark}

\begin{remark}{\rm The linear space spanned by ${\rm ad}_{Q_\Ll^{n}}$, with $Q_\Ll^n$ as in Eq.~\eqref{Eq:FinalGIH}, is easily seen to be invariant against the Lie bracket 
$$
\big ( {\rm ad}_{Q_\Ll^{n}}, {\rm ad}_{Q_\Ll^{n'}} \big ) \mapsto {\rm ad}_{Q_\Ll^{n}} \circ {\rm ad}_{Q_\Ll^{n'}} - {\rm ad}_{Q_\Ll^{n'}} \circ {\rm ad}_{Q_\Ll^{n}},
$$ 
hence it can be regarded as a Lie algebra. Indeed, the coefficients of such Lie brackets continue to have finite range. Then $\dot \Sigma(\Ll)$ can be viewed as the smallest associative algebra enveloping this Lie algebra.
}$\Diamond$
\end{remark} 

The following statement shows that $\dot\Sigma(\Ll)$ carries an intrinsic $\ast$-operation:

\begin{proposition}\label{Pro:StarH} Consider the following $\ast$-operation defined on the generators of $\dot\Sigma(\Ll)$ via
$$
({\rm ad}_{Q_\Ll^n})^\ast(A) = - {\rm ad}_{(Q_\Ll^n)^\ast}(A) = - \big ({\rm ad}_{Q_\Ll^n}(A^\ast)\big )^\ast,
$$
which, at the level of coefficients, acts as 
$$
q_{n} \mapsto q^\ast_{n} = - \overline {q_{n} \circ \mathfrak u_{n}}.
$$ 
Then this operation can be extended to a $\ast$-operation over the entire $\dot\Sigma(\Ll)$ via,
\begin{equation}\label{Eq:QqStar}
\Qq^\ast = \sum_{\{Q\}} (-1)^k \, \overline{c_{\{Q\}}} \ {\rm Ad}_{(Q_\Ll^{n_k})^\ast} \circ \ldots \circ {\rm Ad}_{(Q_\Ll^{n_1})^\ast},
\end{equation}
where $\Qq$ is as in Eq.~\eqref{Eq:Qq}.

\end{proposition}

\proof The challenge for the definition in Eq.~\eqref{Eq:QqStar} is the fact that the presentation in Eq.~\ref{Eq:Qq} of $\Qq$ is not unique. Thus, we must prove that the protocol defined in Eq.~\ref{Eq:QqStar} is independent of the presentation. Our main tool will be the unique tracial state $\Tt$ of ${\rm CAR}(\Ll)$, which satisfies $\Tt \circ {\rm ad}_{Q_\Ll^n} =0$ on $\Dd(\Ll)$, a relation that holds for any almost-inner derivation. Then
$$
\Tt\big (B^\ast \, {\rm ad}_{Q_\Ll^{n}} (A)\big )  = \overline{\Tt\big ({\rm ad}_{Q_\Ll^{n}} (A)^\ast B \big)}= \overline{\Tt\big ({\rm ad}_{(Q_\Ll^{n})^\ast} (A^\ast) B\big )},
$$
and using the Leibniz rule,
$$
\Tt\big (B^\ast \, {\rm ad}_{Q_\Ll^{n}} (A)\big )  =  - \overline{\Tt\big (A^\ast \, {\rm ad}_{(Q_\Ll^{n})^\ast}(B)\big  )} = - \Tt\big ({\rm ad}_{(Q_\Ll^{n})^\ast}(B)^\ast A\big  ).
$$
Iterating,
$$
\Tt\big (B^\ast {\rm ad}_{Q_\Ll^{n_1}} \circ \cdots \circ {\rm ad}_{Q_\Ll^{n_k}} (A)\big )  = (-1)^k \Tt\big ({\rm ad}_{(Q_\Ll^{n_k})^\ast} \circ \cdots \circ {\rm ad}_{(Q_\Ll^{n_1})^\ast}(B)^\ast A\big  ),
$$
hence
\begin{equation}\label{Eq:RR1}
\Tt\big (B^\ast \Qq(A) \big ) = \Tt\big (\Qq^\ast(B)^\ast A \big ) = \overline {\Tt \big (A^\ast \Qq^\ast(B) \big )}.
\end{equation}
Now suppose $\Qq$ has two different presentations for which the protocol~\eqref{Eq:QqStar} returns two different outcomes $\Qq^\ast$ and $\tilde{\Qq}^\ast$. Then Eq.~\eqref{Eq:RR1} assures us that
$$
\Tt \big (A^\ast (\Qq^\ast - \tilde{\Qq}^\ast)(B) \big ) = 0, \quad \forall \ A,B \in \Dd(\Ll).
$$
Since $\Tt$ is faithful, we must conclude that $\Qq^\ast = \tilde{\Qq}^\ast$.\qed

\vspace{0.2cm}

We end the section with examples of elements from $\dot\Sigma(\Ll)$.

\begin{example}{\rm Let $w_n$ be a bi-equivariant $n$-body potential as in Defintion~\ref{Def:KPot}. Then $h_{n} : \widehat{{\rm Del}}_{(r,R)}^{(n,n)} (\RM^d)  \to \RM$
$$
h_n(\xi,\zeta) = w_n \big (\chi_\xi(1), \ldots, \chi_\xi(n);\chi_\zeta(1), \ldots, \chi_\zeta(n)\big )
$$
is a continuous bi-equivariant $n$-coefficient as in Definition~\ref{Def:Z2Coeff} and the corresponding derivation ${\rm ad}_{H_\Ll^n}$ belongs to $\dot\Sigma(\Ll)$.
}$\Diamond$
\end{example}

\begin{example}\label{Ex:DiagH}{\rm Let $w: \RM \to \RM$ be continuous and with compact support and consider $w_{n} : \widehat{{\rm Del}}_{(r,R)}^{(n,n)} (\RM^d)  \to \RM$,
\begin{equation}\label{Eq:CC}
w_{n}(\xi,\zeta ) = (-1)^{\chi_{\xi}^{-1} \circ \chi_{\zeta}} \, \delta_{V_\xi,V_\zeta} \, w(d_\xi),
\end{equation}
where $d_\xi$ is the diameter of $V_\xi$. Then $w_n$ is a continuous equivariant coefficient and the evaluation at $\Ll$ of the corresponding Hamiltonian is
$$
W_\Ll^n = \tfrac{1}{n!}\sum_{\xi \in \mathfrak a_n^{-1}(\Ll)} w(d_\xi) \, n(\xi), \quad n(\xi) := n_{V_\xi},
$$
which supplies an important generalization of Example~\ref{Ex:TwoBody}. To prove the continuity of the coefficient, let
$$
\tilde w_{n} : \widehat{{\rm Del}}_{(r,R)}^{(n)} (\RM^d) \times \widehat{{\rm Del}}_{(r,R)}^{(n)} (\RM^d)  \to \RM,
$$
be defined as
\begin{equation*}
\tilde {w}_{n}(\xi,\zeta) = \left \{
\begin{aligned}
& (-1)^{\chi_\xi^{-1} \circ g \circ \chi_\zeta} \, v\big ({\rm d}_{\rm H}(V_\xi,V_\zeta)\big ) w\big(\sqrt{d_{\xi}d_{\zeta}}\big), \ {\rm if} \ {\rm d}_{\rm H}(V_\xi,V_\zeta) <\tfrac{r}{2}, \\
& 0 \ {\rm otherwise},
\end{aligned}
\right . 
\end{equation*}
where $v:[0,\infty) \to \RM $ is a continuous function with support in the interval $[0,r]$ and $v(0)=1$, while $g : V_\zeta \to V_\xi$ is the canonical map described in Proposition~\ref{Pro:G}. This is clearly a continuous function and its restriction on $\widehat{{\rm Del}}_{(r,R)}^{(n,n)} (\RM^d)$ is exactly the coefficient~\eqref{Eq:CC}.
}$\Diamond$
\end{example}

\subsection{Fock representation of derivations}
\label{SubSec:FockH}

Formally, the derivations introduced in Definition~
\ref{Def:GIHAlgebra} act on the Fock space $\Ff^{(-)}(\Ll)$ via
\begin{equation}\label{Eq:HamAction}
|A\rangle \mapsto |{\rm ad}_{Q_\Ll^n}(A)\rangle, \quad A \in \Dd(\Ll).
\end{equation}
This section is devoted to the representation of $\dot\Sigma(\Ll)$ induced by these actions. 

\begin{proposition}\label{Pro:DotPi} The action in Eq.~\eqref{Eq:HamAction} is well defined and takes the form
$$
|A\rangle \mapsto -\imath \, \dot \pi_\eta(Q^n_\Ll)|A\rangle, \quad A \in \Dd(\Ll),
$$
where, if $Q^n_\Ll$ is as in Eq.~\eqref{Eq:FinalGIH}, then
\begin{equation}\label{Eq:FormalRep}
\begin{aligned}
\dot \pi_\eta(Q^n_\Ll) = \tfrac{1}{n!}   \sum_{(\xi,\zeta)\in \mathfrak b^{-1}_{n}(\Ll)}
  q_{n} (\xi,\zeta) \pi_\eta \big (\bm a^\ast(\xi) \bm a(\zeta) \big ).
\end{aligned}
\end{equation}
Eq.~\eqref{Eq:FormalRep} defines a linear operator on the sub-space $\Dd(\Ll)/\Nn_\eta$.
\end{proposition}
\proof i) The formal action~\eqref{Eq:HamAction} is well defined if and only if 
\begin{equation*}
{\rm ad}_{Q^n_\Ll}\big (\Dd(\Ll) \cap \Nn_\eta \big ) \subseteq \Nn_\eta.
\end{equation*} 
This property is a direct consequence of the fact that $\eta \circ {\rm ad}_{Q^n_\Ll} =0$ on $\Dd(\Ll)$ (see Proposition~\ref{Pro:EtaAd}). Indeed, recall that $A \in \Nn_\eta$ if and only if $\eta(BA)=0$ for all $B \in {\rm CAR}(\Ll)$. If that is the case, then a simple application of the Leibniz rule gives
\begin{equation}\label{Eq:TTT1}
\eta\big (B \, {\rm ad}_{Q^n_\Ll}(A)\big ) = - \eta\big ({\rm ad}_{Q^n_\Ll}(B) A\big ) =0,
\end{equation}
for all $B \in \Dd(\Ll)$. Since $\Dd(\Ll)$ is dense in ${\rm CAR}(\Ll)$, we can safely conclude from Eq.~\eqref{Eq:TTT1} that ${\rm ad}_{Q^n_\Ll}(A) \in \Nn_\eta$. Now, if $A \in \Dd(\Ll)$, then necessarily $A \in {\rm CAR}(\Ll_k)$, for some $k \in \NM$. Let $\tilde Q^n_{\Ll}$ be any finite truncation of the sums in Eq.~\eqref{Eq:FinalGIH} that includes all $\xi$ and $\zeta$ with $(V_\xi \cup V_\zeta) \cap \Ll_k \neq \emptyset$. This is possible because the $Q$'s have finite interaction range. Then $\tilde Q^n_{\Ll_k}$ defines an element of ${\rm CAR}(\Ll)$ and, furthermore,
$$
 {\rm ad}_{Q^n_{\Ll}}(A) = {\rm ad}_{\tilde Q^n_{\Ll}}(A) = \imath(A \, \tilde Q^n_{\Ll} - \tilde Q^n_{\Ll} \, A),
$$
if $A \in \Dd(\Ll) \cap \Nn_\eta$. Since $\tilde Q^n_{\Ll}$ belongs to $\Nn_\eta$, the action of the derivation reduces to
$$
|A \rangle \mapsto |{\rm ad}_{Q^n_{\Ll}}(A)\rangle= -\imath|\tilde Q^n_{\Ll} A \rangle = -\imath  \dot \pi_\eta(\tilde Q^n_{\Ll}) |A \rangle.
$$
The statement then follows.\qed

\begin{proposition}\label{Pro:PiEta} Let $Q^n_\Ll$ be as in Eq.~\eqref{Eq:FinalGIH}. Then
$$
\eta\big (B^\ast \, {\rm ad}_{Q_\Ll^n}(A)\big ) = - \eta\big ({\rm ad}_{(Q_\Ll^n)^\ast}(B)^\ast A\big ), \quad \forall \ A,B \in \Dd(\Ll).
$$
\end{proposition}

\proof The statement follows from Proposition~\ref{Pro:EtaAd} and Leibniz rule.\qed

\begin{proposition} The action $\dot \pi_\eta$ from Eq.~\eqref{Eq:FormalRep} extends to a representation of the whole algebra $\dot\Sigma(\Ll)$ on the linear space $\Dd(\Ll)/\Nn_\eta$. Explicitly, with $\Qq$ as in Eq.~\eqref{Eq:Qq}, then $\dot \pi_\eta(\Qq)|A\rangle : = |\Qq(A) \rangle$, for all $A \in \Dd(\Ll)$, or
\begin{equation}\label{Eq:DotPiQq}
\dot \pi_\eta(\Qq) = \sum_{\{Q\}} c_{\{Q\}} \ \dot \pi_\eta(Q_\Ll^{n_1}) \cdots \dot \pi_\eta(Q_\Ll^{n_k}).
\end{equation}
Furthermore,
\begin{equation}\label{Eq:Hermitean}
\langle B |\dot \pi_\eta(\Qq) |A \rangle = \overline{\langle A |\dot \pi_\eta(\Qq^\ast) |B \rangle}, \quad \forall \ A,B \in \Dd(\Ll).
\end{equation}
\end{proposition}
\proof It follows by iterating Propositions~\ref{Pro:DotPi} and \ref{Pro:PiEta}.\qed

\begin{remark}\label{Re:PhysH}{\rm The Fock representation $\dot \pi_\eta$ gives sense to the formal series~\eqref{Eq:FinalGIH} as endomorphisms of the linear space $\Dd(\Ll)/\Nn_\eta$. Hence, the image of $\dot \Sigma(\Ll)$ through $\dot \pi_\eta$, or equivalently $\dot \Sigma(\Ll) /  \ker \dot \pi_\eta$, can be regarded as the core algebra of the physical Hamiltonians.
}$\Diamond$
\end{remark}

The gauge invariant derivations~\eqref{Eq:FinalGIH} leave the sub-spaces $\Ff^{(-)}_N(\Ll) \cap \Dd(\Ll)/\Nn_\eta$ invariant and this is also true for any element of the algebra $\dot\Sigma(\Ll)$. As a result, the representation $\dot \pi_\eta$ of this algebra decomposes into a direct sum $\dot \pi_\eta = \bigoplus_{N \in \NM} \dot \pi_\eta^N$. 

\begin{proposition}\label{Pro:BoundedH} Let $Q^n_\Ll$ be as in Eq.~\eqref{Eq:FinalGIH}. Then the linear operator $\dot \pi_\eta^N(Q^n_\Ll)$ is uniformly bounded on $\Ff^{(-)}_N(\Ll) \cap \Dd(\Ll)/\Nn_\eta$.
\end{proposition}

\proof Given the finite interaction range of the Hamiltonians, the symmetric presentation of ${\rm ad}_{Q^n_\Ll}(a^\ast_x)$ contains a uniformly bounded number of terms with uniformly bounded coefficients,  w.r.t. $x \in \Ll$. By Leibniz's rule, same statement applies to ${\rm ad}_{Q^n_\Ll}\big (\bm a(\xi)^\ast\big )$ w.r.t. $\xi$, as long as $|V_\xi|$ is fixed. Then
$$
\eta\Big ( {\rm ad}_{Q^n_\Ll}\big (\bm a(\xi)^\ast\big )^\ast {\rm ad}_{Q^n_\Ll}\big(\bm a(\xi)^\ast\big ) \Big ) < \infty,
$$
with an upper bound uniform in $\xi \in \mathfrak a^{-1}_N(\Ll)$. Then the values of $\dot \pi_\eta^N(Q^n_\Ll)$ on the frame $|\xi\rangle = \bm a^\ast(\xi) + \Nn_\eta$ of $\Ff^{(-)}_N(\Ll) \cap \Dd(\Ll)/\Nn_\eta$ are uniformly bounded. \qed

\vspace{0.2cm}

\begin{proposition} The algebra $\dot\Sigma(\Ll)$ accepts a $\ast$-representation $\pi_\eta^N$ inside the algebra $\BM\big(\Ff^{(-)}_N(\Ll)\big )$, for all $N  \in \NM^\times$.
\end{proposition}

\proof A direct consequence of Proposition~\ref{Pro:BoundedH} is that $\dot \pi_\eta^N(\Qq)$ is a uniformly bounded operator on $\Ff^{(-)}_N(\Ll) \cap \Dd(\Ll)/\Nn_\eta$, for {\it any} $\Qq \in \dot\Sigma(\Ll)$. We define $\pi_\Ll^N(\Qq)$ to be the unique element of $\BM\big(\Ff^{(-)}_N(\Ll)\big )$ defined by the closure of the graph of $\dot \pi_\eta^N(\Qq)$. Then Eq.~\eqref{Eq:Hermitean} assures us that
$$
\pi_\eta^N(\Qq^\ast) = \pi_\eta^N(\Qq)^\ast, \quad \forall \ \Qq \in \dot\Sigma(\Ll),
$$
and the statement follows.\qed


\begin{proposition}\label{Pro:ExpRep} The explicit Fock representations of a derivation as in Eq.~\eqref{Eq:FinalGIH} are as follows:
\begin{enumerate}[{\ \rm 1)}]
\item If $n >N$, then
\begin{equation*}
\pi_\eta^N(Q_\Ll^n) =0.
\end{equation*} 

\item If $n=N$, then
\begin{equation}\label{Eq:GoodRep}
\pi_\eta^N(Q_\Ll^N) = \tfrac{1}{N!}   \sum_{(\xi,\zeta)\in \mathfrak b^{-1}_{N}(\Ll)}
  q_{N} (\xi,\zeta) |\xi \rangle \langle \zeta |.
\end{equation}

\item If $n<N$, then
\begin{equation*}
\pi_\eta^N(Q_\Ll^n) = \tfrac{1}{n!(N-n)!}   \sum_{(\xi,\zeta)\in \mathfrak b^{-1}_{n}(\Ll)} \ \sum_{\gamma \in \mathfrak a^{-1}_{N-n}(\Ll)}
  q_{n} (\xi,\zeta) |\xi \vee \gamma \rangle \langle \zeta \vee \gamma |.
\end{equation*}
\end{enumerate}
\end{proposition}
\proof Follows from Eqs.~\eqref{Eq:FormalRep} and Proposition~\ref{Pro:CARMono}.\qed

\section{Resolving the Algebra of Physical Hamiltonians}
\label{Sec:Fermiongroupoid}

The ideals of ${\rm GICAR}$ algebra are invariant by inner limit derivations, hence the latter descend on the quotient algebras generated from the ideals of the filtration introduced in section~\ref{Sec:SolvableCAR}. These quotients coincide with the algebras of compact operators over the Fock sectors, hence there are sub-algebras ${\mathfrak H}_N \subset \BM\big(\Ff_N^{(-)}(\Ll)\big )$ as well as a tower of representations of $\dot \Sigma(\Ll)$ inside ${\mathfrak H}_N \otimes {\mathfrak H}_N^{\rm op}$. Furthermore, the algebra of physical Hamiltonians $\dot{\mathfrak H}(\Ll)=\dot \Sigma(\Ll) / {\rm ker} \, \dot \pi_\eta$ accepts a presentation as the inverse limit $\varprojlim \, \dot\Gamma_N(\Ll)$, where $\dot \Gamma_N(\Ll) = \bigoplus_{n=0}^N \mathfrak H_N$. The first part of the section formalizes these observations. The second part of the section shows that the algebras ${\mathfrak H}_N$ are contained in an essential extension of a specific $C^\ast$-subalgebra $\mathfrak G_N(\Ll)$ of $\BM\big (\Ff_N^{(-)}(\Ll) \big )$, which excludes the compact operators. This enables us to complete $\dot\Sigma(\Ll)$ and $\dot{\mathfrak H}(\Ll)$ to pro-$C^\ast$-algebras given by inverse limits of projective towers of $C^\ast$-algebras. In the third part, we introduce specialized \'etale groupoids and demonstrate that their left regular representations of its (bi-equivariant) $C^\ast$-algebra reproduce the algebras $\mathfrak G_N(\Ll)$. Among other things, these results give a complete characterization of the action of $\dot\Sigma(\Ll)$ on the $N$-fermion sectors.

\subsection{Filtration by ideals} 
\label{Sec:Filtration}

We recall our Remark~\ref{Re:AFAlg}, where we stated that large parts of our program can be repeated for generic AF-algebras. The presentation in this section is intended to support that claim. However, the particularities of the ${\rm CAR}$-algebra and of the filtration used here will eventually enter in an essential way into our analysis. 

The following statement applies to any almost-inner $\ast$-derivation with invariant domain and to any filtration by ideals, in particular, to the derivations and filtration introduced in the previous section and in Proposition~\ref{Pro:GICARIdeals}, respectively:

\begin{proposition}\label{Pro:Descent1} Let $Q_\Ll^n$ be a derivation as in Definition~\ref{Def:GIHAlgebra}. Then:
\begin{enumerate}[{\ \rm i)}]

\item For all $N \in \NM$,
\begin{equation}\label{Eq:Q0}
{\rm ad}_{Q_\Ll^n} \big ({\rm GI}_N(\Ll) \cap \Dd (\Ll) \big ) \subseteq {\rm GI}_N(\Ll) \cap \Dd (\Ll).
\end{equation}

\item  The quotient linear spaces
\begin{equation}\label{Eq:Q1}
\big ({\rm GI}_{N}(\Ll) \cap \Dd (\Ll) \big )/{\rm GI}_{N+1}(\Ll) .
\end{equation}
are $\ast$-algebras.

\item $Q_\Ll^n$ descends to $\ast$-derivations over the $\ast$-algebras~\eqref{Eq:Q1}.

\end{enumerate}
\end{proposition}

\proof i) Let us acknowledge first that ${\rm GI}_{N}(\Ll) \cap \Dd (\Ll)$ are two-sided $\ast$-ideals of the $\ast$-algebra $\Dd(\Ll)$ and, obviously,
\begin{equation*}
{\rm GI}_{N+1}(\Ll) \cap \Dd (\Ll) \subset {\rm GI}_{N}(\Ll) \cap \Dd (\Ll).
\end{equation*}
Now, we use the notation from the proof of Proposition~\ref{Pro:Derivation}. Henceforth, if $A \in {\rm GI}_N(\Ll) \cap \Dd(\Ll)$, then necessarily $A \in {\rm GI}_N(\Ll_p)$ and ${\rm ad}_{Q_\Ll^n}(A) = {\rm ad}_{Q_{\Ll_k}^n}(A)$ for some finite lattices $\Ll_p \subseteq \Ll_k \subset \Ll$. Since ${\rm GI}_N(\Ll_{p}) \subseteq {\rm GI}_N(\Ll_{k})$ and the latter is an ideal of ${\rm GICAR}(\Ll_{k})$, 
\begin{equation*}
{\rm ad}_{Q_\Ll^n}(A)= \imath [A,Q_{\Ll_k}^n] \in  {\rm GI}_N(\Ll_{k})\subset  {\rm GI}_N(\Ll) \cap \Dd(\Ll)
\end{equation*}
and the first statement follows. 

ii) Let $[\cdot ]_N$ denote the quotient classes of the space in Eq.~\eqref{Eq:Q1}. Then the class $[A]_N$ of $A \in {\rm GI}_{N}(\Ll) \cap \Dd (\Ll)$ contains all $A' \in {\rm GI}_{N}(\Ll) \cap \Dd (\Ll)$ such that $A-A' \in {\rm GI}_{N+1}(\Ll)$. Since ${\rm GI}_{N}(\Ll) \cap \Dd (\Ll)$ is stable under addition, $A-A'$ automatically belongs to $\Dd(\Ll)$, hence to ${\rm GI}_{N+1}(\Ll) \cap \Dd(\Ll)$. The latter is an ideal of the sub-algebra ${\rm GI}_{N}(\Ll) \cap \Dd (\Ll)$, hence the quotient space defined in Eq.~\ref{Eq:Q1} is well defined and the algebraic structure of ${\rm GI}_{N}(\Ll) \cap \Dd (\Ll)$ descends on the quotient space~\eqref{Eq:Q1}. Since ${\rm GI}_{N+1}(\Ll) \cap \Dd(\Ll)$ is stable against the $\ast$-operation, the latter also descends to a $\ast$-operation on the quotient space~\eqref{Eq:Q1}. 

iii) Given Eq.~\eqref{Eq:Q0}, we can be sure that ${\rm ad}_{Q_\Ll^n}$ descends to a linear map on this algebra, which we denote by $ \widehat{\rm ad}_{Q_\Ll^n}^N$. We need to verify the Leibniz rule for this map and, for $A,B \in {\rm GI}_{N}(\Ll) \cap D(\Ll)$, we have
\begin{equation*}
\begin{aligned}
\widehat{\rm ad}_{Q_\Ll^n}^N\big([A]_N [B]_N\big) & = \widehat{\rm ad}_{Q_\Ll^n}^N\big ([AB]_N\big ) = \big [{\rm ad}_{Q_\Ll^n}(AB)\big ]_N \\
& = \big [{\rm ad}_{Q_\Ll^n}(A) \,B +A \, {\rm ad}_{Q_\Ll^n}(B)\big ]_N.
\end{aligned}
\end{equation*}
Since both ${\rm ad}_{Q_\Ll^n}(A)$ and ${\rm ad}_{Q_\Ll^n}(B)$ belong to ${\rm GI}_{N}(\Ll) \cap \Dd (\Ll)$, we can conclude
\begin{equation*}
\begin{aligned}
\widehat{\rm ad}_{Q_\Ll^n}^N\big([A]_N [B]_N\big ) = \big [{\rm ad}_{Q_\Ll^n}(A)\big ]_N [B]_N +[A]_N \big [{\rm ad}_{Q_\Ll^n}(B)\big ]_N,
\end{aligned}
\end{equation*}
and the Leibniz rule follows.\qed

\vspace{0.2cm}

We equip the ideals ${\rm GI}_N(\Ll) \cap \Dd (\Ll)$ with the norm inherited from ${\rm CAR}(\Ll)$ and the quotient algebras~\eqref{Eq:Q1} with the quotient norm.

\begin{proposition} The derivations $\widehat {\rm ad}_{Q^n_\Ll}^N$ are uniformly bounded.
\end{proposition}

\proof The quotient algebra~\eqref{Eq:Q1} is spanned by the monomials $[\bm a(\xi)^\ast \bm a(\zeta)]_N$ with $(\xi,\zeta) \in \mathfrak b_{N}^{-1}(\Ll)$. Hence, it is enough to probe the action of the derivations on these monomials, which all have norm one. For $Q^n_\Ll$ as in \eqref{Eq:FinalGIH}, we have
\begin{equation}\label{Eq:Der1}
\begin{aligned}
& \widehat {\rm ad}_{Q^n_\Ll}^N [\bm a(\xi)^\ast \bm a(\zeta)]_N \\
& \qquad = \tfrac{1}{n!(N-n)!}\sum_{\substack{(\xi',\zeta') \in \mathfrak b_n^{-1}(\Ll) \\ V_{\zeta'} \subseteq V_\xi}}  \sum_{\substack{\bar \xi \geq \zeta' \\ V_{\bar \xi} = V_\xi}} q_n(\xi',\zeta') \\
& \qquad  \qquad \qquad \qquad \qquad \qquad \qquad \qquad  (-1)^{\chi^{-1}_\xi \circ \chi_{\bar \xi}} \big [\bm a\big( \xi' \vee (\bar \xi \setminus \zeta')  \big)^\ast \bm a(\zeta)\big]_N \\
&  \qquad - \tfrac{1}{n!(N-n)!}\sum_{\substack{(\xi',\zeta') \in \mathfrak b_n^{-1}(\Ll) \\ V_{\xi'} \subseteq V_\zeta}} \sum_{\substack{\bar \zeta \geq \xi' \\ V_{\bar \zeta} = V_\zeta}} q_n(\xi',\zeta')\\
& \qquad  \qquad \qquad \qquad \qquad \qquad \qquad \qquad (-1)^{\chi^{-1}_\zeta \circ \chi_{\bar \zeta}}\big[\bm a(\xi)^\ast \bm a\big((\bar \zeta  \setminus \xi') \vee \zeta'\big ) \big ]_N,
\end{aligned}
\end{equation}
where we use our conventions on the $\vee$ operation stated in section~\ref{Sec:AFRGI}. There are finite numbers of non-zero terms in the two sums and these finite numbers have uniform upper bounds w.r.t. $\xi$ and $\zeta$. Then the statement follows from the fact that the coefficients of $Q^n_\Ll$ are uniformly bounded in $\xi'$ and $\zeta'$.\qed

\begin{corollary} The derivations $\widehat {\rm ad}_{Q^n_\Ll}^N$ extend to bounded derivations over the quotient algebras
\begin{equation}\label{Eq:Q2}
{\rm GI}_{N}(\Ll) / {\rm GI}_{N+1}(\Ll) \simeq \KM \big (\Ff^{(-)}_N(\Ll)\big) ,\quad N \in \NM.
\end{equation}
\end{corollary}

This is an important conclusion because, at this point, we are again in a context where we have a complete characterization of the derivations, cf. \cite[Example~1.6.4]{BratteliBook1}, namely by commutators with bounded operators over $\Ff^{(-)}_N(\Ll)$. We denote by ${\mathfrak H}_N$ the sub-algebra of $\BM\big (\Ff^{(-)}_N(\Ll)\big )$ generated by these bounded operators, more specifically: 

\begin{definition} For $N \geq 1$, let 
$$
\bar{\mathfrak H}_N = \big \{ B \in \BM\big(\Ff_N^{(-)}(\Ll)\big ) \, | \ \exists \ Q_\Ll^n \in \dot\Sigma(\Ll) \ \mbox{s.t.} \ \widehat{\rm ad}_{Q_\Ll^n}^N = \imath [ \cdot,B] \big \}.
$$
Then ${\mathfrak H}_N$ is the sub-algebra of $\BM\big (\Ff_N^{(-)}\big )$ generated by $\bar {\mathfrak H}_N$. As such, the elements of ${\mathfrak H}_N$ are those $B \in \BM\big (\Ff_N^{(-)}(\Ll)\big )$ that accept a presentation as
$$
B = \sum_{\{b\}} b_{\{B\}} B_{i_1} \cdots B_{i_k}, \quad B_{i_j} \in \bar{\mathfrak H}_N,
$$
with the sum containing only a finite number of terms. For $N=0$, $\mathfrak H_0 = \CM$.
\end{definition}

\begin{remark}{\rm Note that $I_N$, the identity operator over $\Ff_N^{(-)}(\Ll)$, belongs to $\bar {\mathfrak H}_N$. Also, if $B \in \bar {\mathfrak H}_N$, then $B +\alpha I_N$, $\alpha \in \RM$, is also in $\bar {\mathfrak H}_N$
}$\Diamond$
\end{remark}

\begin{remark}{\rm Note that, although the Fock space appears above, the Fock representation has not been used so far. Indeed, the Fock space enter the picture simply because of its relation to the quotient algebra in Eq.~\eqref{Eq:Q2}.
}$\Diamond$
\end{remark}

\begin{proposition} The algebra $\dot\Sigma(\Ll)$ accepts canonical representations $\ddot \pi_N$ inside the algebras ${\mathfrak H}_N \otimes {\mathfrak H}_N^{\rm op}$, where ${\mathfrak H}_N^{\rm op}$ is the algebra opposite to ${\mathfrak H}_N$.
\end{proposition}

\proof Given that $\widehat {\rm ad}_{Q^n_\Ll}^N$ act as commutators with bounded operators, the action of a generic $\Qq \in \dot\Sigma(\Ll)$, as descended on $\KM \big (\Ff^{(-)}_N(\Ll)\big)$, can always be expressed as
$$
\widehat \Qq^N (K) = \sum_{i} B_i K C_i, \quad B_i,C_i\in {\mathfrak H}_N \subset \BM\big (\Ff^{(-)}_N(\Ll)\big ),
$$
for any $K \in \KM \big (\Ff^{(-)}_N(\Ll)\big)$. We claim that the rule
\begin{equation}\label{Eq:BBop}
\Qq \mapsto \ddot \pi_N(\Qq) := \sum_{i} B_i \otimes C_i \in {\mathfrak H}_N \otimes \Hh_N^{\rm op}
\end{equation}
supplies the representation mentioned in the statement. Indeed, the map is well defined because, if $\widehat \Qq^N$ accepts two equivalent expansions,
$$
\widehat \Qq^N (K) = \sum_{i} B_i K C_i  = \sum_{i} B'_i K C'_i, \quad \forall \ K \in \KM \big (\Ff^{(-)}_N(\Ll)\big),
$$
then one finds that both expansions are sent into the same element of ${\mathfrak H}_N \otimes {\mathfrak H}_N^{\rm op}$ by the map~\eqref{Eq:BBop}.  Furthermore, 
$$
\big (\widehat \Qq^{'N} \circ \widehat \Qq^N \big )(K) = \sum_{i,j} B'_j B_i K C_i C'_j,
$$
hence the map~\eqref{Eq:BBop} respects multiplication.\qed 

\vspace{0.2cm}

We recall that one of our goals is to complete $\dot \Sigma(\Ll)$ to a topological algebra and the representations $\ddot \pi_N$ will supply the vehicle to reach this goal. The immediate task is to build an enveloping algebra for $\dot \Sigma(\Ll)$ out of the more manageable and computable algebras ${\mathfrak H}_N \otimes {\mathfrak H}_N^{\rm op}$. One candidate could be the coproduct $\bigoplus_{N=0}^\infty {\mathfrak H}_N \otimes {\mathfrak H}_N^{\rm op}$ but, unfortunately, the topologized version of this coproduct will not fit the unbounded derivations we are dealing with here. The natural alternative is to consider the inverse limit of the tower of algebras
\begin{equation}\label{Eq:Tower1}
\bigoplus_{n=0}^{N+1} {\mathfrak H}_n \otimes {\mathfrak H}_n^{\rm op} \twoheadrightarrow \bigoplus_{n=0}^{N} {\mathfrak H}_n \otimes {\mathfrak H}_n^{\rm op}, \quad N \in \NM.
\end{equation}
By the universal property of inverse limits, we can be assured that there exists an algebra morphism $\theta : \dot \Sigma(\Ll) \to \varprojlim \bigoplus_{n=0}^{N} {\mathfrak H}_n \otimes {\mathfrak H}_n^{\rm op}$. The double sided ideal of $\dot \Sigma(\Ll)$ supplied by the kernel of this morphism contains the linear maps on $\Dd(\Ll)$ that cannot be detected by examining the dynamics of finite but otherwise arbitrary number of fermions. Since this is the setting of any real laboratory experimentation, we can rightfully deem those linear maps as un-physical. The conclusion is that, modulo these un-physical elements, the algebra $\dot \Sigma(\Ll)$ accepts an embedding into $\varprojlim \bigoplus_{n=0}^{N} {\mathfrak H}_n \otimes {\mathfrak H}_n^{\rm op}$. This embedding is the first step towards closing  $\dot \Sigma(\Ll)$ to a pro-$C^\ast$-algebra.

Our focus now shifts towards the algebras ${\mathfrak H}_N$, which now become the central objects of our program. Indeed, the computation of ${\mathfrak H}_N$'s and the identification of suitable $C^\ast$-closures will provide a complete characterization of the algebra of derivations. In the same time, the $C^\ast$-closure of ${\mathfrak H}_N$ can be identified with the algebra of physical Hamiltonians generating the dynamics of $N$ fermions. In fact, we are going to consider the following:

\begin{definition}\label{Def:DotH} We declare the projective limit of algebras
\begin{equation}\label{Eq:Tower2}
\dot{\mathfrak H}(\Ll): = \varprojlim \, \bigoplus_{n=0}^N {\mathfrak H}_n
\end{equation}
to be the core algebra of physical GI-Hamiltonians for a system of self-interacting fermions populating the lattice $\Ll$.
\end{definition}

\begin{remark}{\rm Let us specify that, up to this point, our proposed program can be repeated for any AF-algebra that accepts a filtration by ideals. The computation and characterization of $\dot{\mathfrak H}(\Ll)$, however, will depend crucially on the particularities of the algebra and its filtration. Nevertheless, the case of the ${\rm GICAR}$ algebra analyzed here can server as a model.
}$\Diamond$
\end{remark}

The following statements establish the connection with the Fock representation.

\begin{proposition}\label{Pro:TwoReps} We have $\widehat {\rm ad}_{Q^n_\Ll}^N = \imath [\cdot,\pi_\eta^N (Q_\Ll^n)]$, for any $Q_\Ll^n$ as in Eq.~\eqref{Eq:FinalGIH}.
\end{proposition}
\proof Under isomorphism~\eqref{Eq:Iso1}, we have
\begin{equation*}
[\bm a(\xi)^\ast \bm a(\zeta)]_N \mapsto |\xi \rangle \langle \zeta|,
\end{equation*}
for all $(\xi,\zeta) \in \mathfrak b_N^{-1}(\Ll)$. Then, using Proposition~\ref{Pro:ExpRep}, for any $Q_\Ll^n$ as in Eq.~\eqref{Eq:FinalGIH}, one can check explicitly that the commutator of $\pi_\eta^N(Q_\Ll^n)$ with $|\xi \rangle \langle \zeta|$ implements Eq.~\eqref{Eq:Der1}.\qed

\begin{corollary} The algebras $\mathfrak H_N$ can be identified with the images of $\dot \Sigma(\Ll)$ through the Fock representations,
$$
{\mathfrak H}_N \simeq \pi_\eta^N\big ( \dot \Sigma(\Ll) \big ) \subset \BM\big (\Ff_N^{(-)}(\Ll) \big ).
$$
\end{corollary} 

The above property, which is specific to the particular filtration of ${\rm GICAR}(\Ll)$ considered here, enables us to give an alternative characterization of the algebra of physical Hamiltonians:

\begin{proposition} For $N \in \NM$, we denote by $\dot \Jj_N$ the double sided $\ast$-ideals of $\dot\Sigma(\Ll)$,
\begin{equation}\label{Eq:JjN}
\dot \Jj_N = {\rm ker} \ \pi_\eta^N \cap \ldots \cap {\rm ker} \  \pi_\eta^0 = {\rm ker} \ \pi_\eta^N \oplus \cdots \oplus \pi_\eta^0,
\end{equation}
which supply a filtration of $\dot \Sigma(\Ll)$
\begin{equation*}
\cdots \rightarrowtail \dot\Jj_{N+1} \rightarrowtail \dot \Jj_{N}  \rightarrowtail \cdots \rightarrowtail \dot \Jj_0 = \dot\Sigma(\Ll)
\end{equation*}
and a projective tower of $\ast$-algebras
\begin{equation}\label{Eq:ProjTower1}
\cdots \twoheadrightarrow \dot\Sigma(\Ll)/\dot \Jj_{N} \twoheadrightarrow \dot\Sigma(\Ll)/\dot \Jj_{N-1} \twoheadrightarrow \cdots \twoheadrightarrow 0.
\end{equation}
This projective tower is isomorphic with the projective tower~\eqref{Eq:Tower2} defining $\dot{\mathfrak H}(\Ll)$.
\end{proposition}

\begin{remark}{\rm The universal property of inverse limits guarantees the existence of an algebra morphism $\dot \Sigma(\Ll) \to \dot{\mathfrak H}(\Ll)$, which is clearly surjective in this case. If $\Jj_\infty$ is the kernel of this morphism, then $\Jj_\infty = \ker \dot \pi_\eta$ and, as such, $\dot \Sigma(\Ll)/\Jj_\infty \simeq \dot{\mathfrak H}(\Ll)$ coincides with the image of $\dot \Sigma(\Ll)$ through the Fock representation. In Remark~\ref{Re:PhysH}, we identified this image with algebra of physical Hamiltonians and this fully justifies Definition~\ref{Def:DotH}.
}$\Diamond$
\end{remark}

\subsection{Completing the algebra of the physical Hamiltonians}

Our next task is to complete $\dot{\mathfrak H}(\Ll)$ to a (pro-) $C^\ast$-algebra and to characterize this completion. For this, we supply first a finer characterization of ${\mathfrak H}_N \simeq \pi_\eta^N\big (\dot \Sigma(\Ll) \big )$. 

\begin{proposition} Let $Q_\Ll^N$ and $Q_\Ll^n$ be two derivations  as in Eq.~\eqref{Eq:FinalGIH}. Then
\begin{equation}\label{Eq:WW}
\pi_\eta^N\big ( Q_\Ll^n \circ Q_\Ll^N \big ) = \pi_\eta^N\big ( \tilde Q_\Ll^N \big )
\end{equation}
for some $\tilde Q_\Ll^N$ as in Eq.~\eqref{Eq:FinalGIH}.
\end{proposition}

\proof We will make use of Proposition~\ref{Pro:ExpRep}. If $n >N$, then $\tilde Q_\Ll^N$ can be taken as zero. If $n \leq N$, then, up to a normalization constant, 
\begin{equation*}
\pi_\eta^N(Q_\Ll^n \circ Q_\Ll^N) = \sum_{(\xi,\zeta) \in \mathfrak b_N^{-1}(\Ll)} \sum_{\zeta' \in \mathfrak a_n^{-1}(\Ll)} \sum_{\substack{\gamma \in \mathfrak a_{N-n}^{-1}(\Ll)\\ \gamma \leq \xi} } \, q_n(\xi \setminus \gamma,\zeta')  q_N(\zeta' \vee \gamma,\zeta) \, |\xi\rangle \langle \zeta |
\end{equation*}
and we can take $\tilde Q_\Ll^N$ to be the derivation corresponding to the coefficient
\begin{equation*}
\tilde q_N(\xi,\zeta) = \sum_{\zeta' \in \mathfrak a_n^{-1}(\Ll)} \sum_{\substack{\gamma \in \mathfrak a_{N-n}^{-1}(\Ll) \\ \gamma \leq \xi} } \, q_n(\xi \setminus \gamma,\zeta')  q_N(\zeta' \vee \gamma,\zeta). 
\end{equation*}
Indeed,  $\tilde q_N$ is continuous, bi-equivariant and vanishes if the diameter of $V_\xi \cup V_\zeta$ is sufficiently large. \qed

\begin{remark}{\rm Similar statements apply to the product $Q_\Ll^N \circ Q_\Ll^n$, which follows from applying the $\ast$-operation on $Q_\Ll^n \circ Q_\Ll^N$.
}$\Diamond$
\end{remark}

\begin{corollary} The linear space $\dot{\mathfrak G}_N(\Ll)$ of ${\mathfrak H}_N$ spanned by $\pi_\eta^N\big (Q_\Ll^N \big )$ with $Q_\Ll^N$ as in Eq.~\eqref{Eq:FinalGIH} is a two-sided $\ast$-ideal of ${\mathfrak H}_N$.
\end{corollary}

\begin{definition} For $N >1$, we define $\mathfrak G_N(\Ll)$ to be the non-unital $C^\ast$-algebra supplied by the norm closure of $\dot{\mathfrak G}_N(\Ll)$ inside $\BM\big(\Ff^{(-)}_N(\Ll)\big)$.
\end{definition}

\begin{remark}{\rm The case $N=1$ is special. Indeed, $\dot{\mathfrak G}_1(\Ll)$ is unital and its closure $\mathfrak G_1(\Ll)$ coincides with a left regular representation of the algebra $C^\ast_r(\Gg_1)$ studied in section~\ref{SubSec:groupoidAlg1}.
}$\Diamond$
\end{remark}

\begin{remark}{\rm As already pointed out in Remark~\ref{Re:Compl}, it is at this point, through the completion $\mathfrak G_N(\Ll)$, that Hamiltonians with infinite interaction range are included in our theory.
}$\Diamond$
\end{remark}

We will show next that ${\mathfrak H}_N$ is contained in an essential extension of $\mathfrak G_N(\Ll)$.

\begin{proposition}\label{Pro:UnitApprox} Let $w: \RM_+ \to [0,1]$ be a smooth, non-increasing function, which is equal to 1 over the interval $[0,1]$ and to 0 over the interval $[2,\infty)$. For $\epsilon>0$, we write $w_\epsilon(t) = w(\epsilon t)$ and define the net of uniformly bounded operators
\begin{equation}\label{Eq:AIdentity}
1_N^\epsilon : = \tfrac{1}{N!}\sum_{\gamma \in \mathfrak a_N^{-1}(\Ll)} w_\epsilon(d_\gamma) \, |\gamma \rangle \langle \gamma |, \quad \|1_N^\epsilon\| = 1,
\end{equation}
where $d_{\gamma}$ denotes the diameter of $V_\gamma$.
Then $1_N^\epsilon$ is a quasi-central approximate unit for $\mathfrak G_N(\Ll)$.
\end{proposition}

\proof Example~\ref{Ex:DiagH} assures us that the operator~\eqref{Eq:AIdentity} belongs to $\mathfrak G_N(\Ll)$. By definition, any element from $\mathfrak G_N(\Ll)$ can be approximated in norm by a $Q_\Ll^N$ as in Eq.~\eqref{Eq:FinalGIH}, for which we have
$$
1_N \, \pi_\eta^N (Q_\Ll^N) - \ddot \pi_\eta^N(Q_\Ll^N) = \sum_{(\xi,\zeta )\in \mathfrak b_N^{-1}(\Ll)} \big (w_\epsilon(d_{\xi}) - 1\big) q_N(\xi,\zeta)  
 |\xi  \rangle \langle \zeta |.
$$
The right hand side is a bounded operator that vanishes identically once $1/\epsilon$ exceeds the interaction range of $Q_\Ll^N$. Similarly, 
$$
\big [1_N^\epsilon,\pi_\eta^N(Q_\Ll^N) \big ]=\sum_{(\xi,\zeta) \in \mathfrak b_N^{-1}(\Ll)} \big (w_\epsilon(d_{\xi}) - w_\epsilon(d_{\zeta})\big) \, q_N(\xi,\zeta) \, 
 |\xi  \rangle \langle \zeta |
$$
vanishes identically once $1/\epsilon$ exceeds the interaction range of $Q_\Ll^N$.\qed 

\begin{corollary} ${\mathfrak H}_N$ is contained in an essential extension of $\mathfrak G_N(\Ll)$.
\end{corollary}

\proof The net $1_N^\epsilon$ converges to the identity of $\BM\big ( \Ff^{(-)}_N(\Ll) \big )$ in its strong topology. Therefore, for any $\Qq \in \dot\Sigma(\Ll)$, the net $ \pi_\eta^N(\Qq) 1_N^\epsilon$ belongs to $\mathfrak G_N(\Ll)$ and converges to $\pi_\eta^N(\Qq)$ in the strong topology. As a conclusion, ${\mathfrak H}_N$ is contained in the strong closure of $\mathfrak G_N(\Ll)$. Now recall that  $\mathfrak G_N(\Ll)$ is the norm closure of $\dot{\mathfrak G}_N(\Ll)$ and the latter is an ideal of ${\mathfrak H}_N$. Then the smallest $C^\ast$-algebra containing ${\mathfrak H}_N$ must contain ${\mathfrak G}_N(\Ll)$ as an ideal and be contained in the strong closure of $\mathfrak G_N(\Ll)$. Hence, ${\mathfrak H}_N$ is contained in an essential extension of $\mathfrak G_N(\Ll)$.  \qed

\vspace{0.2cm}

We now have the means to identify a natural completion of ${\mathfrak H}_N$ to a $C^\ast$-algebra:

\begin{definition}\label{Def:Multiplier} We define the completion of ${\mathfrak H}_N$ to be the maximal essential extension of $\mathfrak G_N(\Ll)$, {\it i.e.} its multiplier algebra $\Mm\big (\mathfrak G_N(\Ll)\big )$. 
\end{definition}

\begin{remark}\label{Re:Multiplier}{\rm A non-unital $C^\ast$-algebra $\Aa$ accepts, in general, many essential extensions. The minimal one corresponds to the unitization $\tilde A$ of the algebra and the maximal one to the multiplier algebra $\Mm(\Aa)$, as already stated above. The choice made in Definition~\ref{Def:Multiplier} is based on physical considerations. Indeed, $\Mm(\Aa)$ can be characterized as follows \cite{AkemannJFA1973}. If $\Aa^m$ denotes the set of self-adjoint elements of $\Aa''$, the double commutant of $\Aa$,  that can be reached from below by increasing nets from $\tilde A$ and $\Aa_m := - \Aa^m$, then the self-adjoint part of $\Mm(\Aa)$ coincides with $\Aa^m \cap \Aa_m$. As such, Definition~\ref{Def:Multiplier} declares that, from our point of view, if $\pm H \in \BM\big (\Ff_N^{(-)}(\Ll)\big )$ can be both reached from below via sequences of finite-interaction range Hamiltonians, then $H$ can play the role of a physical Hamiltonian. 
}$\Diamond$
\end{remark}

We now reach a major conclusion of our work:

\begin{theorem}\label{Th:M1} The algebra of physical Galilean and gauge invariant Hamiltonians with finite interaction range accepts a natural completion to a pro-$C^\ast$-algebra $\mathfrak H(\Ll)$ in the sense of \cite{PhillipsJOT1988}.
\end{theorem}

\proof We generate the projective tower of $C^\ast$-algebras
\begin{equation}\label{Eq:ProjTower2}
\cdots \twoheadrightarrow \bigoplus_{n=0}^{N+1} \Mm\big (\mathfrak G_n(\Ll)\big ) \twoheadrightarrow \bigoplus_{n=0}^N \Mm\big (\mathfrak G_n(\Ll)\big ) \twoheadrightarrow \cdots \twoheadrightarrow 0
\end{equation}
and take $\mathfrak H(\Ll)$ as its inverse limit in the category of topological $\ast$-algebras. Then $\mathfrak H(\Ll)$ is a pro-$C^\ast$-algebra \cite{PhillipsJOT1988}. Furthermore, the projective tower~\eqref{Eq:Tower2} can be embedded in the tower~\eqref{Eq:ProjTower2}. As a result, there is a faithful algebra morphism from $\dot {\mathfrak H}(\Ll)$ to ${\mathfrak H}(\Ll)$. \qed

\begin{remark}{\rm Let $\mathfrak p_N : {\mathfrak H}(\Ll) \twoheadrightarrow \bigoplus_{n=0}^N \Mm\big (\mathfrak G_n(\Ll)\big )$ be the empimorphisms associated with the inverse limit of the tower~\eqref{Eq:ProjTower2}. If $\Qq \in \dot\Sigma(\Ll)$ is a generic element as in Eq.~\eqref{Eq:Qq}, then its image in $\mathfrak H(\Ll)$ is equal to the coherent sequence 
$$
\{\mathfrak p_N(\Qq)\}_N, \quad \mathfrak p_N(\Qq) = \bigoplus_{n=0}^N \pi_\eta^N(\Qq),
$$
with $\pi_\eta^N (\Qq)$ supplied in and below Eq.~\eqref{Eq:DotPiQq}.  
While the projections $\mathfrak p_N(\Qq)$  are all bounded, the sequence $\{\mathfrak p_N(\Qq)\}_N$ is not uniformly bounded. As such, the inverse limit of the tower~\eqref{Eq:ProjTower2} cannot be taken inside the category of $C^\ast$-algebras. On the other hand, if the inverse limit is taken inside the category of topological $\ast$-algebras, then any $\Qq$ from $\dot \Sigma(\Ll)$ has an image in $\mathfrak H(\Ll)$.
}$\Diamond$
\end{remark}

The inverse limit $\mathfrak H(\Ll)$ admits the following intrinsic characterization:

\begin{corollary} The algebra $\mathfrak H(\Ll)$ admits a filtration by $C^\ast$-ideals 
\begin{equation*}
\cdots \rightarrowtail \Jj_{N+1} \rightarrowtail \Jj_{N} \rightarrowtail \cdots  \rightarrowtail \Jj_0 = \Sigma(\Ll),
\end{equation*}
with the property that
\begin{equation}\label{Eq:IdJ}
\Jj_{N-1}/\Jj_N = \Mm\big (\mathfrak G_N(\Ll)\big ).
\end{equation}
\end{corollary}

\proof  Take $\Jj_N : = {\rm ker}\, \mathfrak p_N$. Obviously, $\Jj_N$'s satisfy the identity~\eqref{Eq:IdJ}.\qed

\vspace{0.2cm}

Lastly, we can supply the desired completion of the algebra of derivations:

\begin{theorem}\label{Th:SigmaFinal} The algebra of physical derivations $\dot \Sigma(\Ll)$ accepts a completion to a pro-$C^\ast$-algebra. Specifically, modulo non-physical elements, there is the faithful morphism
$$
\dot \Sigma(\Ll) \rightarrowtail \varprojlim \, \bigoplus_{n=0}^N \Mm\big (\mathfrak G_n(\Ll)\big ) \otimes  \Mm\big (\mathfrak G_n(\Ll)\big )^{\rm op}.
$$
\end{theorem}

\subsection{Topological groupoid $\Gg_N$ associated to the dynamics of $N$-fermions}
\label{SubSec:AlgGN} 

From here on, our focus is entirely on the characterization of the $C^\ast$-algebra $\mathfrak G_N(\Ll)$. As we shall see in subsection~\ref{Sec:LeftRegRep}, $\mathfrak G_N(\Ll)$ coincides with a left regular representation of a bi-equivariant $C^\ast$-algebra associated to the \'etale groupoid $\mathcal{G}_{N}$ introduced below and studied in this section. 

Throughout this and following subsections, we fix a Delone set $\Ll_0$ and we restrict all other patterns $\Ll$ to the hull $\Omega_{\Ll_0}$. Recalling the definition of $\widehat{{\rm Del}}_{(r,R)}^{(N,N)}(\RM^d)$ from Eq.~\eqref{Eq:KeySpace}, we consider the subset where the patterns are restricted to $\Omega_{\Ll_0}$:
$$
\widehat{\Omega}_{\Ll_0}^{(N)} : =\Big \{ \big (\xi,\zeta \big ) \in \widehat{{\rm Del}}_{(r,R)}^{(N)}(\RM^d) \times \widehat{{\rm Del}}_{(r,R)}^{(N)}(\RM^d), \ \Ll_\xi = \Ll_\zeta \in \Omega_{\Ll_0} \Big \}.
$$
Since $\Omega_{\Ll_0}$ is closed and translation invariant, the action~\eqref{Eq:RdAction} of $\RM^d$ on $\widehat{{\rm Del}}_{(r,R)}^{(N,N)}(\RM^d)$ restricts to $\widehat{\Omega}_{\Ll_0}^{(N)}$ and we can consider the quotient space $\big [\widehat \Omega_{\Ll_0}^{(N)}\big ] \subset \big[\widehat{{\rm Del}}_{(r,R)}^{(N,N)}(\RM^d)\big]$. The latter was introduced in Definition~\ref{Def:QuotientSp} and served as the domain for the seeds of the Hamiltonian coefficients. We so obtain a commutative diagram
$$
\centerline{\xymatrix{\widehat{\Omega}_{\Ll_0}^{(N)} \ar[r] \ar[d]^{\hat{\rm q}_{N}} &\widehat{{\rm Del}}_{(r,R)}^{(N,N)}(\RM^d) \ar[d]^{\hat{\rm q}_{N}} \\  \big[\widehat{\Omega}_{\Ll_0}^{(N)} \big]\ar[r] & \big[\widehat{{\rm Del}}_{(r,R)}^{(N,N)}(\RM^d)\big] }}
$$
of principal $\mathbb{R}^{d}$-bundles.  We will show that the base space $\big [\widehat \Omega_{\Ll_0}^{(N)}\big ]$ carries the structure of an \'etale groupoid with a 2-action of the symmetric group $\mathcal{S}_{N}$. Its bi-equivariant $C^{*}$-algebra turns out to have the sought for properties. To this end, we will need a concrete model for this quotient space as a subspace of $\widehat \Omega_{\Ll_0}^{(N)}$.

\begin{lemma}The map
\begin{equation}\label{Eq:Tn}
\hat{\rm p}_N:\widehat \Omega_{\Ll_0}^{(N)}\to \widehat \Omega_{\Ll_0}^{(N)},\quad \hat{\rm p}_{N}(\xi,\zeta):=\hat{\mathfrak{t}}_{\chi_\xi(1)}(\xi,\zeta)
\end{equation}
is continuous, $\mathbb{R}^{d}$-invariant and satisfies $(\hat{\rm p}_{N})^{2}:=\hat{\rm p}_{N}\circ\hat{\rm p}_{N}=\hat{\rm p}_{N}$.
\end{lemma}
\begin{proof}
Continuity of $\hat{\rm p}_{N}$ is immediate. Furthermore, the identity
$$
\hat{\mathfrak t}_{\chi_{\hat{\mathfrak t}_{x}(\zeta)}(1)}= \hat{\mathfrak t}_{\chi_{\zeta}(1)}\circ \hat{\mathfrak t}_{x}^{-1},
$$
proves that $\hat{\rm p}_{N}$ is $\mathbb{R}{^d}$-invariant and that $\hat{\rm p}_{N}\circ\hat{\rm p}_{N}=\hat{\rm p}_{N}$.
\end{proof}

\begin{definition} We define 
\begin{align*}\Gg_N:&=\textnormal{Ran}(\hat{\rm p}_{N})=\hat{\rm p}_{N}\left(\widehat \Omega_{\Ll_0}^{(N)}\right)\\
&=\Big \{(\xi,\zeta)\in \widehat{{\rm Del}}_{(r,R)}^{(N,N)}(\RM^d), \, \Ll_\xi=\Ll_\zeta \in \Omega_{\mathcal{L}_{0}},\, \chi_\xi(1) = 0 \Big \},
\end{align*}
with the relative topology inherited from  the inclusion $\hat{\rm j}_{N}:\mathcal{G}_{N} \rightarrowtail \widehat{{\rm Del}}_{(r,R)}^{(N,N)}(\RM^d)$.
\end{definition}
We consider the space $\mathcal{G}_{N}\times\mathbb{R}^{d}$ with the $\mathbb{R}^{d}$ action given by translation action of $\mathbb{R}^{d}$ on itself: $t_{x}((\xi,\zeta),y):= ( (\xi,\zeta), y-x )$. The map $\widehat{\rm pr}_{N}:\mathcal{G}_{N}\times\mathbb{R}^{d}\to \mathcal{G}_{N}$ given by projection onto $\mathcal{G}_{N}$, is a trivial principal $\mathbb{R}^{d}$-bundle over $\mathcal{G}_{N}$. We will now show that $\mathcal{G}_{N}\times\mathbb{R}^{d}$ and $\widehat \Omega_{\Ll_0}^{(N)}$ are homeomorphic as principal $\mathbb{R}^{d}$-bundles.
\begin{proposition} The map
\[\sigma_{N} :\widehat \Omega_{\Ll_0}^{(N)} \xrightarrow{\sim} \mathcal{G}_{N}\times\mathbb{R}^{d},\quad (\xi,\zeta)\mapsto (\hat{\rm p}_{N}(\xi,\zeta), \chi_{\xi}(1)).\]
is an $\mathbb{R}^{d}$-equivariant homeomorphism. There is a commutative diagram
$$
\centerline{\xymatrix{\widehat{\Omega}_{\Ll_0}^{(N)} \ar[d]^{\hat{\rm q}_{N}} &\widehat{\Omega}_{\Ll_0}^{(N)} \ar[l]_{=}\ar[r]^{\sigma_{N}} \ar[d]^{\hat{\rm p}_{N}} &\mathcal{G}_{N}\times\mathbb{R}^{d}\ar[d]^{\widehat{\rm pr}_{N}}\\ \big[\widehat{\Omega}_{\Ll_0}^{(N)}\big]& \mathcal{G}_{N} \ar[l]_{\hat{\rm q}_{N}}\ar[r]^{=} & \mathcal{G}_{N}}}
$$
 of principal $\mathbb{R}{^d}$-bundles.
\end{proposition}
\begin{proof} 
The map $\sigma_{N}$ is continuous and $\mathbb{R}^{d}$-equivariant since $\hat{\rm p}_{N}$ is $\mathbb{R}^{d}$-invariant and $\chi_{\hat{\mathfrak{t}}_{x}\xi}(1)=\chi_{\xi}(1)-x$. Its inverse is given by
\[\tau_{N}:\mathcal{G}_{N}\times\mathbb{R}^{d}\to \widehat \Omega_{\Ll_0}^{(N)},\quad ((\xi,\zeta),x)\mapsto \hat{\mathfrak{t}}_{-x}(\xi,\zeta).\]
It follows that $\sigma_{N}$ and $\hat{\rm p}_{N}$ are $\mathbb{R}{^d}$ equivariant homeomorphims.
\end{proof}

\begin{corollary}\label{Co:Equiv} The relative topology on $\mathcal{G}_{N}$ induced by inclusion $\hat{\rm j}_{N}:\mathcal{G}_{N}\to \widehat \Omega_{\Ll_0}^{(N)}$ coincides with the quotient topology induced by $\hat{\rm p}_{N}$. The restriction of the quotient map $\hat{\rm q}_{N}:\widehat \Omega_{\Ll_0}^{(N)}\to\big[\widehat \Omega_{\Ll_0}^{(N)}\big]$ induces a homeomorphism $\mathcal{G}_{N}\xrightarrow{\sim} \big[\widehat \Omega_{\Ll_0}^{(N)}\big]$ with inverse induced by the map $\hat{\rm p}_{N}$.
\end{corollary}
We now continue to equip the space $\mathcal{G}_{N}$ with the structure of an \'etale groupoid. In the following, pairs like $(\xi,\zeta)$ will be automatically assumed from $\widehat \Omega_{\Ll_0}^{(N)}$ if not specified otherwise.

\begin{definition}\label{Def:GroupoidN} The groupoid associated to the dynamics of $N$ fermions on a Delone set $\Ll_0$ consists of:
\begin{enumerate}[\rm \ 1.] 
\item The topological space $\Gg_N$, whose definition is reproduced below 
$$
\label{eq: setGn}
\Gg_N = \big \{(\xi,\zeta)\in \widehat \Omega_{\Ll_0}^{(N)},\, \chi_\xi(1) = 0 \big \},
$$
equipped with the inversion map
$$
(\xi,\zeta)^{-1} = \hat{\mathfrak{t}}_{\chi_\zeta(1)} (\zeta,\xi).
$$

\item The set of composable elements 
\[\Gg^{(2)}_{N}:= \Big \{\big ((\xi,\zeta),(\xi',\zeta')\big ) \in \Gg_N \times \Gg_N:  \xi' = \hat{\mathfrak t}_{\chi_\zeta(1)} \zeta \Big \}\subset \Gg_N \times \Gg_N,\] 
equipped with the composition map
$$
(\xi,\zeta) \cdot \hat {\mathfrak t}_{\chi_\zeta(1)}( \zeta,\zeta') := (\xi, \zeta').
$$
\end{enumerate}
\end{definition}

\begin{remark}\label{Re:Pairs}{\rm Note that $\Ll_\xi$ ($=\Ll_\zeta $) actually belongs to the transversal $\Xi_{\Ll_0}$ for any $(\xi,\zeta) \in \Gg_N$. Also, pairs like $\hat{\mathfrak t}_{\chi_\xi(1)}(\xi,\zeta)$ and $\hat{\mathfrak t}_{\chi_{\zeta}(1)}(\zeta,\eta)$ are always composable and any composable pair can be written in this form. Furthermore,
\begin{equation}\label{Eq:V1}
\hat{\mathfrak t}_{\chi_\xi(1)}(\xi,\zeta) \cdot  \hat{\mathfrak t}_{\chi_{\zeta}(1)}(\zeta,\eta) = \hat{\mathfrak t}_{\chi_\xi(1)}(\xi,\eta).
\end{equation}
These statements follows directly from the identity
\begin{equation}\label{Eq:ID101}
\hat{\mathfrak t}_{\chi_{\hat{\mathfrak t}_{\chi_{\xi(1)}}(\zeta)}(1)}= \hat{\mathfrak t}_{\chi_{\zeta}(1)} \circ \hat{\mathfrak t}_{\chi_\xi(1)}^{-1},
\end{equation}
which will also be used below several times.
}$\Diamond$
\end{remark}

According to the above definition, the range and the source maps are given by
\begin{equation}\label{Eq:RSFinal}
\mathfrak{r}(\xi,\zeta) = (\xi,\xi),\quad \mathfrak{s}(\xi,\zeta) =\hat{\mathfrak{t}}_{\chi_\zeta(1)}(\zeta,\zeta).
\end{equation}
As such, the space of units $\Gg_N^{(0)}$ can be canonically identified with
\begin{equation}\label{Eq:XiN}
\Xi_{\Ll_0}^{(N)} := \big \{\xi \in \widehat{{\rm Del}}_{(r,R)}^{(N)}(\RM^d): \ \Ll_\xi \in \Xi_{\Ll_0}, \ \chi_\xi(1)=0 \big \}.
\end{equation}
With this notation, $\Xi_{\Ll_0}^{(1)} = \Xi_{\Ll_0}$, the transversal of $\Ll_0$ introduced in section~\ref{SubSec:Patterns}. Note that for all $N >1$, $\Xi_{\Ll_0}^{(N)}$ is not a compact space.

\begin{proposition} With its algebraic structure the set $\Gg_N$ is a groupoid over $\Xi_{\mathcal{L}_0}^{(N)}$.
\end{proposition}

\noindent{\it Proof.} We will go over the defining properties in the order stated in Definition~\ref{Def:groupoid}.

(1) First, in the light of Remark~\ref{Re:Pairs}, it is easily seen that the inversion indeed returns an element of $\Gg_N$. Furthermore,
$$
\begin{aligned}
\big ( (\xi,\zeta)^{-1} \big )^{-1} = \big ( (\hat{\mathfrak t}_{\chi_\zeta(1)} \zeta, \hat{\mathfrak t}_{\chi_\zeta(1)}\xi) \big)^{-1} = ( \hat{\mathfrak t}_{\chi_{\hat{\mathfrak t}_{\chi_\zeta(1)}(\xi)}(1) } \circ \hat{\mathfrak t}_{\chi_\zeta(1)}) ( \xi, \zeta)
\end{aligned}
$$
and, after using \eqref{Eq:ID101}, property (1) follows because $\mathfrak t_{\chi_\xi(1)} =\mathfrak t_0$, hence the identity.

(2) Consider now the pairs 
$$
\big ((\xi,\zeta),\hat{\mathfrak t}_{\chi_\zeta(1)}(\zeta,\zeta')\big ), \quad \big (\hat{\mathfrak t}_{\chi_\zeta(1)}(\zeta,\zeta'), \hat{\mathfrak t}_{\chi_{\zeta'}(1)}(\zeta',\zeta'')\big ),
$$ 
of composable elements (see Remark~\ref{Re:Pairs}). Then
\begin{equation*}
\mathfrak s \big ( (\xi,\zeta) \cdot \hat{\mathfrak t}_{\chi_\zeta(1)}(\zeta,\zeta') \big ) = \hat{\mathfrak t}_{\chi_{\zeta'}(1)}(\zeta',\zeta')=\mathfrak r \big ( \hat{\mathfrak t}_{\chi_{\zeta'}(1)}(\zeta',\zeta'') \big ),
\end{equation*}
hence property (2) holds. Furthermore, the composition is associative: computing a triple product one way gives
$$
\big ((\xi,\zeta) \cdot \hat{\mathfrak t}_{\chi_\zeta(1)}(\zeta,\zeta')\big ) \cdot \hat{\mathfrak t}_{\chi_{\zeta'}(1)} (\zeta',\zeta'') = (\xi,\zeta')\cdot \hat{\mathfrak t}_{\chi_{\zeta'}(1)} (\zeta',\zeta'')= (\xi, \zeta'')
$$
whereas the other way gives
$$
(\xi,\zeta) \cdot  \big (\hat{\mathfrak t}_{\chi_\zeta(1)}(\zeta,\zeta') \cdot \hat{\mathfrak t}_{\chi_{\zeta'}(1)} (\zeta',\zeta'') \big )  = (\xi,\zeta) \cdot  \hat{\mathfrak t}_{\chi_\zeta(1)}(\zeta, \zeta'') =(\xi, \zeta'').
$$

(3) Clearly, $(\xi,\zeta)$ and $(\xi,\zeta)^{-1} = \hat{\mathfrak t}_{\chi_\zeta(1)}(\zeta,\xi)$ satisfy the condition for composability, hence property (3) is satisfied.

(4) Using the associativity of the multiplication,
$$
(\xi,\zeta)^{-1} \cdot \big ( (\xi,\zeta)\cdot \hat{\mathfrak t}_{\chi_\zeta(1)}(\zeta,\zeta') \big ) = \mathfrak s (\xi,\zeta) \cdot \hat{\mathfrak t}_{\chi_\zeta(1)}(\zeta,\zeta') = \hat{\mathfrak t}_{\chi_\zeta(1)}(\zeta,\zeta'),
$$ 
and
$$
\big ( (\xi,\zeta)\cdot \hat{\mathfrak t}_{\chi_\zeta(1)}(\zeta,\zeta') \big ) \cdot \big ( \hat{\mathfrak t}_{\chi_\zeta(1)}(\zeta,\zeta') \big )^{-1}= (\xi,\zeta) \cdot \mathfrak r \big (\hat{\mathfrak t}_{\chi_\zeta(1)}(\zeta,\zeta') \big ) =(\xi,\zeta).
$$
Thus, property (4) follows. \qed

\vspace{0.2cm}

Recall from \eqref{Eq:BNeigh} the neighborhood base for the topology of $\widehat{{\rm Del}}_{(r,R)}^{(N)}(\RM^d)$:
\begin{equation}\label{Eq:USets}
U_M^\epsilon(\xi)= U_M^\epsilon(\Ll_\xi,V_\xi,\chi_{\xi}) = \Big \{(\Ll,U,g \circ \chi_U): \ (\Ll,U) \in U_M^\epsilon(\Ll_\xi,V_\xi)\Big \}.
\end{equation}
Since $\widehat{{\rm Del}}_{(r,R)}^{(N,N)}(\RM^d)$ carries the relative topology induced by the product topology, a base for the relative topology on  $\mathcal{G}_{N}$ is supplied by the sets
\begin{align*}
\mathcal{G}_{N}\cap \big (U_M^\epsilon(\xi)\times U_{M'}^{\epsilon'}(\zeta)\big )=\left\{ (\xi',\zeta')\in\mathcal{G}_{N}:\xi'\in U_M^\epsilon(\xi),\, \zeta'\in U_{M'}^{\epsilon'}(\zeta) \right\},
\end{align*}
where $(\xi,\zeta)\in\widehat{{\rm Del}}_{(r,R)}^{(N,N)}(\RM^d)$. We now provide a description of a neighbourhood basis for the points $(\xi,\zeta)$ of $\mathcal{G}_{N}$.
\begin{lemma} For $M>>0$ large enough and $\epsilon<r/2$, the collection
\[\big\{\mathcal{G}_{N}\cap(U^{\epsilon}_{M}(\xi)\times U^{\epsilon}_{M}(\zeta)):(\xi,\zeta)\in\mathcal{G}_{N} \big\}\]
form a base for the topology of $\mathcal{G}_{N}$.
\end{lemma}
\begin{proof}
For $(\xi',\zeta')\in \mathcal{G}_{N}\cap(U_{M'}^{\epsilon'}(\xi)\times U_{M''}^{\epsilon''}(\zeta))$, we have $\mathcal{L}_{\xi'}=\mathcal{L}_{\zeta'}$ and, since the sets~\eqref{Eq:USets} form a neighborhood base for the topology on $\Xi^{(N)}_{0}$, there exist $M,\epsilon$ such that 
\[U^{\epsilon}_{M}(\xi')\subset U^{\epsilon'}_{M'}(\xi),\quad U^{\epsilon}_{M}(\zeta')\subset U^{\epsilon''}_{M''}(\zeta).\]
Therefore every open set $\mathcal{G}_{N}\cap(U_{M'}^{\epsilon'}(\xi)\times U_{M''}^{\epsilon''}(\zeta))$ is a union of open sets of the form
\[\mathcal{G}_{N}\cap(U^{\epsilon}_{M}(\xi')\times U^{\epsilon}_{M}(\zeta')),\]
with $(\xi',\zeta')\in\mathcal{G}_{N}$. This proves the lemma.
\end{proof}
\begin{lemma} 
\label{lem: topology}
Let $(\xi,\zeta)\in\mathcal{G}_{N}$. For $M>>0$ sufficiently large and $\epsilon<r/2$ the sets $\mathcal{G}_{N}\cap(U_M^\epsilon(\xi)\times U_M^{\epsilon}(\zeta))$ and 


\[\left\{(\xi',\zeta'): \xi'\in U_{M}^{\epsilon}(\xi),\, \mathcal{L}_{\xi'}=\mathcal{L}_{\zeta'},\,V_{\zeta'}\subset\mathring{B}(V_\zeta,\epsilon),\, \chi_{\zeta'}=g\circ\chi_{\zeta}\right\},\]
coincide. In particular $\mathfrak{r}:\mathcal{G}_{N}\cap(U_M^\epsilon(\xi)\times U_M^{\epsilon}(\zeta))\to U^{\epsilon}_{M}(\xi)$ is a bijection.
\end{lemma}
\begin{proof}

Let $(\xi',\zeta')\in \mathcal{G}_{N}\cap\left(U_M^\epsilon(\xi)\times U_M^{\epsilon}(\zeta)\right)$, so that $\mathcal{L}_{\xi'}=\mathcal{L}_{\zeta'}$. Then, since $V_{\zeta'}\subset \mathcal{L}_{\zeta'}[M]$ (see Remark~\ref{Re:LM}), it follows that 
$$
{\rm d}_{\rm H}(V_\zeta,V_{\zeta'}) \leq {\rm d}_{\rm H}(\Ll_\zeta[M],\Ll_{\zeta'}[M]) \leq \epsilon
$$
hence $V_{\zeta'} \subset\mathring{B}(V_\zeta,\epsilon)$. Since $\epsilon \leq r/2$, there is a bijection $g:V_\zeta \to V_{\zeta'}$ and this property determines $V_{\zeta'}$ and $\chi_{\zeta'}$. Thus, we find that $\mathcal{G}_{N}\cap(U_M^\epsilon(\xi)\times U_M^{\epsilon}(\zeta))$ coincides with
\begin{align*}
\left\{(\xi',\zeta'): \xi'\in U_{M}^{\epsilon}(\xi),\,\Ll_{\xi'}=\Ll_{\zeta'},\, V_{\zeta'}\subset\mathring{B}(V_\zeta,\epsilon),\,\chi_{\zeta'}=g\circ \chi_\zeta \right\},
\end{align*}
which is the desired equality of sets. Now $\mathfrak{r}(\xi',\zeta')=\xi'$ and for every $\xi'\in U^{\epsilon}_{M}(\xi)$ there is a unique 
\[\zeta'=(\Ll_{\zeta'},V_{\zeta'},\chi_{\zeta'})=(\Ll_{\xi'},g(V_\zeta), g\circ \chi_\zeta) \in U^{\epsilon}_{M}(\zeta),\]
such that
$(\xi',\zeta')\in\mathcal{G}_{N}\cap (U^{\epsilon}_{M}(\xi)\times U^{\epsilon}_{M}(\zeta))$.
\end{proof}
\begin{proposition} The topological space $\mathcal{G}_{N}$ is an \'etale groupoid.
\end{proposition}

\begin{proof}
By Lemma \ref{lem: topology}, the restriction of $\mathfrak{r}:\Gg_N\to \Xi^{(N)}_{0}$ to the sets
\[\mathcal{W}=\mathcal{W}_{M}^{\epsilon}(\xi,\zeta):=\mathcal{G}_{N}\cap\left(U^{\epsilon}_{M}(\xi)\times U_{M}^{\epsilon}(\zeta)\right),\]  as above, induces a bijection
\[\mathfrak{r}:\mathcal{W}_{M}^{\epsilon}(\xi,\zeta)\to U^{\epsilon}_{M}(\xi,\zeta).\] 
Since the sets $U^{\epsilon}_{M}(\xi,\zeta)$ form a basis for the topology on $\Xi^{(N)}_{0}$, this proves that $\mathfrak{r}$ is continuous. Similarly, since $\mathcal{W}^{\epsilon}_{M}(\xi,\zeta)$ form a basis for the topology of $\mathcal{G}_{N}$, we conclude that $\mathfrak{r}$ is open, so that it is a local homeomorphism. 

Recall the homeomorphism $\mathfrak{u}_{N}$ and the bundle projection $\hat{\rm p}_N$ from Eqs.~\eqref{Eq:Un} and \eqref{Eq:Tn}, respectively. 
The inversion map
\[(\xi,\zeta)^{-1} = \hat{\mathfrak{t}}_{\chi_\zeta(1)}(\zeta,\xi)=(\hat{\rm p}_N \circ\mathfrak{u}_{N})(\xi,\zeta)\] is a homeomorphism:
Using identity~\eqref{Eq:ID101}, it follows that
\[(\hat{\rm p}_N \circ\mathfrak{u}_{N})^{2}(\xi,\zeta)=\hat{\rm p}_N\circ\mathfrak{u}_{N}(\hat{\mathfrak{t}}_{\chi_\zeta(1)}\zeta,\hat{\mathfrak{t}}_{\chi_\zeta(1)}\xi)=\hat{\mathfrak{t}}_{\chi_\xi(1)}(\xi,\zeta),\]
and thus \[(\hat{\rm p}_N\circ\mathfrak{u}_{N})(\mathcal{G}_{N})\subset\mathcal{G}_{N},\quad (\hat{\rm p}_N\circ\mathfrak{u}_{N})^{2}(\mathcal{G}_{N})=\mathcal{G}_{N},\]
from which we conclude $(\hat{\rm p}_N\circ\mathfrak{u}_{N})(\mathcal{G}_{N})=\mathcal{G}_{N}$. Hence the inversion map is the restriction to $\mathcal{G}_{N}$ of a continuous map of $\widehat{{\rm Del}}_{(r,R)}^{(N,N)}(\RM^d)$ to itself, and so it is a homeomorphism of $\mathcal{G}_{N}$. 

It now follows that the source map $\mathfrak{s}=\mathfrak{r}\circ \hat{\rm p}_N\circ\mathfrak{u}_{N}$ is a homeomorphism as well.
Lastly, we consider the groupoid multiplication
\begin{align*}
\mathfrak m:\mathcal{G}_{N}^{(2)}\to \mathcal{G}_{N},\quad
 ((\xi,\zeta),(\mathfrak{ t}_{\chi_\zeta(1)} \zeta ,\zeta'))\mapsto (\xi,\mathfrak{t}_{\chi_\zeta(1)} ^{-1}\zeta ').
\end{align*}
Using the coordinate projection \[\pi_{1}:\mathcal{G}_{N}^{(2)}\to \mathcal{G}_{N},\quad ((\xi,\zeta),(\xi',\zeta'))\mapsto (\xi,\zeta),\]
the contraction map \[\mathfrak c:\widehat{{\rm Del}}_{(r,R)}^{(N,N)}(\RM^d)\to \widehat{{\rm Del}}_{(r,R)}^{N}(\RM^d),\quad ((\xi,\zeta),(\xi',\zeta'))\mapsto (\xi,\zeta'),\] 
and the maps $\mathfrak{v}_{N},\mathfrak{u}_{N}$ from Proposition \ref{Eq:RdAction}, the map $\mathfrak m$
can be written as a composition
\begin{align*}
\mathfrak m=\mathfrak c\circ(\textnormal{Id}\times\mathfrak{t}_{\mathfrak{v}_{N}\circ \mathfrak{u}_{N}\circ\pi_{1}}^{-1}),
\end{align*}
of continuous maps and is therefore continuous.
\end{proof}



We end this section with a relation between the groupoids. 

\begin{proposition} Any $\xi \in \widehat{{\rm Del}}_{(r,R)}^{(N+M)}(\RM^d)$ can be uniquely decomposed as $\xi_1 \vee \xi_2$, where $\xi_1 \in \widehat{{\rm Del}}_{(r,R)}^{(N)}(\RM^d)$ and $\xi_2 \in \widehat{{\rm Del}}_{(r,R)}^{(M)}(\RM^d)$.
\end{proposition}

\proof Indeed, we can enumerate the first $N$ points of $V_\xi$ using the order function $\chi_\xi$ and generate $\xi_1$. The remaining points of $V_\xi$ together with their order supply $\xi_2$. This decomposition is clearly unique.\qed

\begin{proposition}\label{Pro:Mor10} The map $\mathfrak e: \Gg_{N+M} \to \Gg_N \times \Gg_M$,
\begin{equation}\label{Eq:Mor10}
\mathfrak e(\xi,\zeta) = \mathfrak e(\xi_1 \vee \xi_2,\zeta_1 \vee \zeta_2) = (\xi_1,\zeta_1) \times \hat{\mathfrak t}_{\chi_{\xi_2}(1)}(\xi_2,\zeta_2),
\end{equation}
is a continuous morphism of groupoids, which maps the fibers under the $\mathfrak r$ map injectively.
\end{proposition}

\proof Let $(\xi,\zeta)$ and $\hat{\mathfrak t}_{\chi_\zeta(1)}(\zeta,\eta)$ be a pair of composable elements from $\Gg_{N+M}$. Then 
$$
\mathfrak e \big ( \hat{\mathfrak t}_{\chi_\zeta(1)}(\zeta,\eta) \big )  =  \hat{\mathfrak t}_{\chi_\zeta(1)}(\zeta_1,\eta_1) \times \hat{\mathfrak t}_{\chi_{\zeta_2}(1)}(\zeta_2,\eta_2),
$$
hence $\mathfrak e(\xi,\zeta)$ and $\mathfrak e \big ( \hat{\mathfrak t}_{\chi_\zeta(1)}(\zeta,\eta) \big )$ are composable in $\Gg_N \times \Gg_M$ (see Remark~\ref{Re:Pairs}). Furthermore
$$
\mathfrak e \big ( (\xi,\zeta) \cdot \hat{\mathfrak t}_{\chi_\zeta(1)}(\zeta,\eta) \big ) = \mathfrak e (\xi,\zeta) \cdot \mathfrak e \big ( \hat{\mathfrak t}_{\chi_\zeta(1)}(\zeta,\eta) \big ),
$$
which can be also quickly derived from Remark~\ref{Re:Pairs}. The fiber at $(\xi,\xi)$ under the $\mathfrak r$ map consists of the pairs $(\xi,\zeta) \in \Gg_{N+M}$. It follows directly from Eq.~\eqref{Eq:Mor10} that $\mathfrak e$ restricted to such fiber is injective. \qed

\begin{remark}{\rm A pattern is called aperiodic if $\Ll_0-x \neq \Ll_0$ for all $x \in \RM^d$, $x \neq 0$. It follows directly from Eq.~\eqref{Eq:Mor10} that $\mathfrak e$ is injective if $\Ll_0$ is aperiodic. The morphism is, however, never surjective in this case, because $\mathfrak e$ cannot produce pairs of elements from $\Gg_N \times \Gg_M$ with identical lattices (note that necessarily $\chi_{\xi_2}(1) \neq 0$ in Eq.~\eqref{Eq:Mor10}). Nevertheless, by composing $\mathfrak e$ with the projection on the first entry of $\Gg_N \times \Gg_M$, we obtain a surjective groupoid morphism $\Gg_{N+M} \to \Gg_N$.
}$\Diamond$
\end{remark}

\begin{remark}\label{Re:Pullback}{\rm The pullback maps of proper morphisms preserving the Haar systems of groupoids induce morphisms between the corresponding reduced groupoid $C^\ast$-algebras \cite{AustinJFA2019}. Unfortunately, the morphisms introduced in Proposition~\ref{Pro:Mor10} do not preserve the Haar system of our groupoids. Indeed, for \'etale groupoids, the latter is true if and only if the morphism maps bijectively the fibers under $\mathfrak r$ \cite{AustinNYJM2021}, which we already know not be the case here, in general. The morphism $\mathfrak e$ is also not proper. In fact, the pre-image of any compact subset of $\Gg_N \times \Gg_M$ is never compact. Nevertheless, the pullback maps do supply $\CM$-module morphisms from $C^\ast_r(\Gg_N \times \Gg_M)$ ($\simeq C^\ast_r(\Gg_N) \otimes C^\ast_r(\Gg_M)$) to $\Mm\big (C^\ast_r(\Gg_{N+M})\big )$, whose ranges have trivial intersections with $C^\ast_r(\Gg_{N+M})$.
}$\Diamond$
\end{remark}

\subsection{The 2-action of the symmetric group and the bi-equivariant groupoid algebra}
\label{Sec:2Action}

We saw in Proposition~\ref{Pro:Deck} that $\mathcal{S}_{N}$ acts as the group of deck transformations on the order cover $\widehat{\mathrm{Del}}^{(N)}_{(r,R)}(\RM^d)$. We will prove here that this action can be combined with translations to generate a 2-action of $\mathcal{S}_{N}$ on $\mathcal{G}_{N}$. We will follow the notation introduced in section~\ref{Sec:EquiGg}, as adapted to our specific groupoid $\Gg_N$.

\begin{proposition}\label{Pro:2Action} Let $s \in\mathcal{S}_{N}$ be a permutation. The formula
\begin{equation}
\label{eq: 2actiongroupoid}
\tau_s(\xi):=\hat{\mathfrak{t}}_{\chi_{\xi}\circ s^{-1}(1)}\big(\Lambda_s(\xi),\xi\big)
\end{equation}
defines a homomorphism $\tau:\mathcal S_{N}\to \mathcal{S}(\mathcal{G}_{N})$ and thus a 2-action of $\mathcal{S}_{N}$ on $\mathcal{G}_{N}$.
\end{proposition}
\begin{proof}
It is clear that $\tau_s$ defines a section of the source map $\mathfrak{s}(\xi,\zeta) =\hat{\mathfrak{t}}_{\chi_\zeta(1)}(\zeta,\zeta)$. Moreover
\[\mathfrak{r}\circ\tau_{s}(\xi)= \hat{\mathfrak{t}}_{\chi_\xi\circ s^{-1}(1)}\circ \Lambda_s (\xi)\]
is a homeomorphism of $\Xi_{\Ll_0}^{(N)}$. We compute the product $\tau_{s_1}\cdot \tau_{s_2}$ inside the group $\mathcal{S}(\mathcal{G}_{N})$, using Eqs.~\eqref{Eq:Rule1} and \eqref{Eq:ID101}  :
\begin{align*}
\tau_{s_1}\cdot\tau_{s_2}(\xi)&=\tau_{s_1} (\mathfrak{r}\circ\tau_{s_2}(\xi))\tau_{s_2}(\xi)\\
&=\tau_{s_1}\left(\hat{\mathfrak{t}}_{\chi_\xi\circ s_2^{-1}(1)} \circ \Lambda_{s_2}(\xi)\right)\tau_{s_2}(\xi)\\
&=\hat{\mathfrak{t}}_{\chi_\xi\circ(s_1 s_2)^{-1}(1)}\big(\Lambda_{s_1s_2}(\xi),\xi\big),
\end{align*}
which proves that we have a group homomorphism $\mathcal{S}_{N}\to \mathcal{S}(\mathcal{G}_{N})$.
\end{proof}

The induced left and right actions (see Eq.~\ref{Eq:2Action1}) of $\mathcal{S}_{N}$ on $\mathcal{G}_{N}$ will be denoted by $s_1 \cdot (\xi,\zeta) \cdot s_2$, for $s_i \in \mathcal{S}_N$. It should not be confused with the action from Proposition~\ref{Prop:leftright}, which acts on a different space. We will need the explicit expression of the actions. 

\begin{proposition}\label{Pro:LRActions} The commuting left and right actions of $\mathcal{S}_{N}$ on $\mathcal{G}_{N}$ induced by the 2-action \eqref{eq: 2actiongroupoid} are given by
\begin{equation}\label{Eq:2Action2}
s_1 \cdot (\xi,\zeta) \cdot s_2  = \hat{\mathfrak{t}}_{\chi_{\xi}\circ s_1^{-1}(1)}\big (\Lambda_{s_1}(\xi),\Lambda_{s_2^{-1}}(\zeta)\big ),
\end{equation}
for $s_i \in \mathcal{S}_N$.
\end{proposition}

\proof From the definition~\ref{Eq:2Action1}), we have
\begin{equation*}
\begin{aligned}
s_1 \cdot (\xi,\zeta) \cdot s_2 & = \tau_{s_1}\big(\xi\big)\cdot (\xi,\zeta) \cdot\tau_{s_2^{-1}}\big(\hat{\mathfrak t}_{\chi_\zeta(1)}(\zeta,\zeta)\big)^{-1} \\
& = \hat{\mathfrak{t}}_{\chi_{\xi}\circ s_1^{-1}(1)}\big(\Lambda_{s_1}(\xi),\xi\big)\cdot (\xi,\zeta) \cdot \big(\hat{\mathfrak t}_{\chi_\zeta\circ s_2^{-1}(1)}\Lambda_{s_2^{-1}}\zeta,\hat{\mathfrak t}_{\chi_\zeta\circ s_2^{-1}(1)}\zeta\big)^{-1}.
\end{aligned}
\end{equation*}
After we apply Eq.~\eqref{Eq:V1} and compute the inverse,
\begin{equation*}
\begin{aligned}
s_1 \cdot (\xi,\zeta) \cdot s_2  = \hat{\mathfrak{t}}_{\chi_{\xi}\circ s_1^{-1}(1)}(\Lambda_{s_1}(\xi),\zeta) \cdot \hat{\mathfrak t}_{\chi_\zeta(1)} \big (\zeta,\Lambda_{s_2^{-1}}(\zeta) \big ).
\end{aligned}
\end{equation*}
Then one last application of Eq.~\eqref{Eq:V1} supplies the desired answer.\qed

\vspace{0.2cm}

We are now in the position where we can define the bi-equivariant groupoid algebra associated with the dynamics of the fermions:

\begin{definition} In the notation from section~\ref{Sec:EquiGg}, the bi-equivariant groupoid $C^\ast$-algebra associated with the dynamics of $N$-fermions is $C^{*}_{r,\mathcal{S}_N} (\Gg_N, \CM)$, where the 2-action of $\mathcal{S}_N$ on $\CM$ is simply $\Ss_N \ni s \mapsto (-1)^s \in UM(\CM)$.
\end{definition}

\subsection{Left regular representations and the dynamics of fermions} 
\label{Sec:LeftRegRep}

We now have all the pieces in place and can complete the characterization of the algebra of Hamiltonians stated in Theorem~\ref{Th:0}:

\begin{proposition} The algebras $\mathfrak G_N(\Ll)$, $\Ll \in \Xi_{\Ll_0}$, coincide with the left regular representations of the bi-equivariant groupoid algebra $C^{\ast}_{r,\mathcal{S}_N}(\Gg_N, \CM)$.
\end{proposition}

\proof The left regular representations of $C^\ast_r(\Gg_N)$ are carried by the Hilbert spaces $\ell^2\big(\mathfrak s^{-1}(\xi_0)\big)$ with $\xi_0 \in \Xi_{\Ll_0}^{(N)}$. From Eq.~\eqref{Eq:RSFinal}, we find
$$
s^{-1}(\xi_0) = \big \{ (\xi_0,\xi)^{-1}, \ (\xi_0,\xi) \in \Gg_N \}.
$$
Then, the generic expression~\eqref{Eq:PiX} of a left regular representation translates into
\begin{equation*}
\begin{aligned}
[\pi_{\xi_0}(f)\psi]\big ((\xi_0,\xi)^{-1}\big) & = \sum_{(\xi_0,\zeta) \in \Gg_N} f\big ((\xi_0,\xi)^{-1} \cdot (\xi_0,\zeta) \big ) \psi\big ( (\xi_0,\zeta)^{-1}\big ) \\
& = \sum_{(\xi_0,\zeta) \in \Gg_N} f\big (\hat{\mathfrak t}_{\chi_\xi(1)}(\xi,\zeta) \big ) \psi\big ( (\xi_0,\zeta)^{-1}\big )
\end{aligned}
\end{equation*}
We now observe that, if $\xi_0=(\Ll,V_0,\chi_{V_0})$, then there is a canonical isomorphims between the Hilbert spaces $\ell^2\big(\mathfrak s^{-1}(\xi_0)\big)$ and $\Ff_N^{(-)}(\Ll)$, which sends $\psi \in \ell^2\big(\mathfrak s^{-1}(\xi_0)\big)$ to $\phi \in \Ff_N^{(-)}(\Ll)$ via the relation
\begin{equation*}
\psi\big ( (\xi_0,\zeta)^{-1}\big ) = \phi (\zeta) = \phi(\Ll,V_\zeta,\chi_\zeta), \quad \forall \ (\xi_0,\zeta) \in \Gg_N.
\end{equation*}
As such, the left regular representations occur naturally on the Fock space and they take the form
\begin{equation}\label{Eq:FinalRep}
\pi_{\xi_0}(f) = \sum_{(\xi,\zeta) \in \mathfrak b_N^{-1}(\Ll)}  f\big (\hat{\mathfrak t}_{\chi_\xi(1)} (\xi,\zeta )\big ) |\xi\rangle \langle \zeta |.
\end{equation}

Consider now an element from the algebra $\mathfrak G_N(\Ll)$. By definition, it can be approximated in norm by the representation of an element $Q_\Ll^{N}$ on the Fock space $\Ff_N^{(-)}(\Ll)$. According to Eq.~\eqref{Eq:GoodRep}, this representation takes the form
\begin{equation*}
\pi_\eta^N(Q_\Ll^N) = \tfrac{1}{N!}   \sum_{(\xi,\zeta)\in \mathfrak b^{-1}_{N}(\Ll)}
  q_{N} (\xi,\zeta) |\xi \rangle \langle \zeta |,
\end{equation*}
where $q_N$ is a continuous equivariant coefficient, as introduced in Definition~\ref{Def:Z2Coeff}. Due to the equivariance against the translations, this expression can be cast in a form very similar with the one in Eq.~\eqref{Eq:FinalRep},
\begin{equation*}
\pi_\eta^N(Q_\Ll^N) = \tfrac{1}{N!}   \sum_{(\xi,\zeta)\in \mathfrak b^{-1}_{N}(\Ll)}
  q_{N} \big(\hat{\mathfrak t}_{\chi_\xi(1)}(\xi,\zeta) \big ) |\xi \rangle \langle \zeta |,
\end{equation*}
and the two will be identical if we can establish a connection between the $f$'s and $q_N$'s in these two expressions. This is established below.\qed

\begin{proposition}\label{Pro:SeedG} The seed of an equivariant  coefficient $q_{N}$ generates an element of the bi-equivariant groupoid sub-algebra $C^{\ast}_{r,\mathcal{S}_N}(\Gg_N, \CM)$, via the relation
$$
f(\xi,\zeta) = \tfrac{1}{N!} \, \hat q_N \big ([(\xi,\zeta)] \big).
$$
\end{proposition}

\proof This is a direct consequence of Corollary~\ref{Co:Equiv}.\qed

\subsection{$\Gg_N$ as a blow-up groupoid}\label{Sec:BlowUp1}
We provide here an alternative description of the groupoids $\mathcal{G}_{N}$ through a topological procedure known as blowing up the unit space. This allows us to describe the structure of the $C^\ast$-algebras $C^\ast_r(\Gg_N)$.
\begin{definition}[\cite{WilliamsBook}~p.~45] Let $\Gg$ be a locally compact Hausdorff groupoid with open range map. Suppose that $Z$ is locally compact Hausdorff and that $f: Z \to \Gg^{(0)}$ is  a continuous open map. Then
$$
\Gg[Z]=\big \{ (z,\gamma,w)\in Z \times \Gg \times Z: 
 f(z)=r(\gamma) \ {\rm and} \ s(\gamma) = f(w) \big \}
 $$
 is a topological groupoid when considered with the natural operations $$(z,\gamma,w) (w,\eta,x)=(z,\gamma \eta,x) \ {\rm and} \ (z,\gamma,w)^{-1}= (w,\gamma^{-1},z),
 $$
 and the topology inherited from $Z \times \Gg \times Z$.
 \end{definition}
 
 \begin{remark}{\rm The space of units for $\Gg[Z]$ is
 $$
 \Gg[Z]^{(0)} = \big \{ (z,f(z),z), \ z \in Z\big \},
 $$
 hence it can be naturally identified with $Z$. For this reason, the map $f$ is called a blow-up of the unit space and the groupoid $\Gg[Z]$ is referred to as the blow-up of $\Gg$ through $f$.
 }$\Diamond$
 \end{remark}
 
 \begin{proposition} Recall the unit space of $\Gg_N$ from Eq.~\eqref{Eq:XiN}. The map 
 $$
 f: \Gg_{N}^{(0)} \to \Gg_{1}^{(0)}, \quad f(\Ll,V,\chi_V) : = (\Ll,\chi_V(1)) = (\Ll,0)
 $$ 
 is a blow-up of the unit space $\Gg_{1}^{(0)}$.
 \end{proposition}
 
 \proof Indeed, the arguments from Proposition~\ref{Pro:Cover2} apply also here, which show that $f$ is a local homeomorphism.\qed
 
 \vspace{0.2cm}
 
The associated blow-up groupoid can be presented as
$$
\widetilde \Gg_{N} = \big \{ \big (\xi, (\Ll,x), \zeta\big), \ (\Ll,x) \in \Gg_1,\  \xi \in \mathfrak a_N^{-1}(\Ll), \ \zeta \in \mathfrak a_N^{-1}(\Ll-x), \ \chi_\xi(1) = \chi_\zeta(1)=0 \big \}.
$$

\begin{proposition} The map
\begin{equation}\label{Eq:LastIso}
\Gg_N \ni (\xi,\zeta) \mapsto \big (\xi, \big (\Ll_\xi,\chi_\zeta(1) \big ), \mathfrak t_{\chi_\zeta(1)}(\zeta) \big )\in \widetilde \Gg_N.
\end{equation}
is an isomorphism $\Gg_N \simeq \widetilde \Gg_N$ of topological groupoids.
\end{proposition}

\proof 
Indeed, the composable elements of $\Gg_N$ come as pairs $(\xi,\zeta)$ and $\mathfrak t_{\chi_\zeta(1)}(\zeta,\eta)$. This pair is mapped to 
$$
\big (\xi, \big (\Ll_\xi,\chi_\zeta(1) \big ), \mathfrak t_{\chi_\zeta(1)}(\zeta) \big ) \quad {\rm and} \quad  \big (\mathfrak t_{\chi_\zeta(1)}(\zeta), \mathfrak t_{\chi_\zeta(1)}\big (\Ll_\zeta,\chi_\eta(1) \big ), \mathfrak t_{\chi_\eta(1)}(\eta) \big )
$$
and the latter are composable in $\widetilde \Gg_N$, with the composition given by
$$
\big (\xi, \big (\Ll_\xi,\chi_\eta(1) \big ), \mathfrak t_{\chi_\eta(1)}(\eta) \big ).
$$
Hence, the map~\eqref{Eq:LastIso} respects the composition. Furtermore, the inverse of $(\xi,\zeta) \in \Gg_N$ is $\mathfrak t_{\chi_\zeta(1)}(\zeta,\xi)$, and this element is mapped to
$$
\big (\mathfrak t_{\chi_\zeta(1)}(\zeta), \mathfrak t_{\chi_\zeta(1)}\big (\Ll_\zeta,\chi_\xi(1) \big ), \mathfrak t_{\chi_\xi(1)}(\xi) \big ) = \big (\mathfrak t_{\chi_\zeta(1)}(\zeta), \big (\Ll_\xi,\chi_\zeta(1) \big )^{-1}, \mathfrak t_{\chi_\xi(1)}(\xi) \big ),
$$
hence the map~\eqref{Eq:LastIso} respects the inversion, too. Lastly, it can be checked that the map~\eqref{Eq:LastIso} is a homeomorphism.\qed

\begin{corollary}[\cite{WilliamsBook},~Th.~2.52] The groupoid $C^{*}$-algebras $C^\ast_r(\Gg_1)$ and $C^\ast_r(\Gg_N)$ are Morita equivalent.
\end{corollary}

Note that this Morita equivalence does not apply to the bi-equivariant subalgebras $C^\ast_{r,\Ss_N}(\Gg_N)$. The equivalences between $\Gg_1$ and its blow-ups $\Gg_N$ and the associated imprimitivity bimodules can be written explicitly \cite[p.~45]{WilliamsBook}. For the present context, it will be extremely useful to also construct the isomorphisms $\KM \otimes C^\ast_r(\Gg_1) \simeq \KM \otimes C^\ast_r(\Gg_N)$, explicitly, and to investigate their relations with the conditional expectations onto $C^\ast_{r,\Ss_N}(\Gg_N)$. Indeed, there is a real possibility of using these isomorphisms to generate the inductive tower~\eqref{Eq:Diag1} from the tensor product of $C^\ast_r(\Gg_1)$ and an almost-finite $C^\ast$-algebra.

The Morita equivalence between $C^\ast_r(\Gg_1)$ and $C^\ast_r(\Gg_N)$ supplies a good start for the computation of the $K$-theory of $C^\ast_{r,\Ss_N}(\Gg_N)$. However, new tools need to be developed in order to address group 2-actions on groupoid algebras and understand their effect on $K$-theory.

\section{Analysis of special cases with disconnected order covers}
\label{Sec:SpecialCases}

The periodic lattices $\Ll_0 = \ZM^d$ and their small perturbations are very special because the order covers over the transversals of such patterns are disconnected. Indeed, as we shall see below, the subsets of these lattices and their translates accept canonical orderings that are compatible with the translations and the topology of the covers. As a consequence, the bi-equivariant groupoid $C^\ast$-algebra $C^\ast_{r,\Ss_N}(\Gg_N)$ reduces to the $C^\ast$-algebra of a blow-up of $\Gg_1$ and we can put forward a very strong statement saying that $K$-theory of $C^\ast_{r,\Ss_N}(\Gg_N)$ is identical to that of $C^\ast_{r}(\Gg_1)$.

\subsection{Canonical ordering for perturbed periodic patterns}

We start by introducing a large class of patterns displaying disconnected order covers. 

\begin{definition} Let $\Pp_\varepsilon = \bigcup_{n \in \ZM^d} B(n,\varepsilon)$, $\varepsilon <1/2$, be a pattern of disconnected closed balls centered at the points of the periodic lattice $\ZM^d$. We call $\Ll_0 \subset \RM^d$ a perturbation of the periodic lattice if, for any $\Ll \in \Xi_{\Ll_0}$, there exists $\varepsilon <1/2$ such that
\begin{enumerate}[{\ \ \rm 1.}]
\item $\Ll \subset \Pp_\varepsilon$;
\item $|\Ll \cap B(n,\varepsilon)| =1$, for all $n \in \ZM^d$.
\end{enumerate}
\end{definition} 

\begin{example}{\rm Let $\Ll_0$ be a point pattern such that $\Ll_0 \subset \Pp_{\varepsilon_0}$ for $\varepsilon_0 <1/4$ and $|\Ll \cap B(n,\varepsilon_0)| =1$ for all $n \in \ZM^d$. Then $\Ll_0$ is a perturbation of $\ZM^d$. Indeed, the points of a translate $\Ll_0-x$, for some arbitrary $x \in \Ll_0$, are all at a distance less than $2 \epsilon_0 <1/2$ from $\ZM^d$, hence the translates $\Ll_0-x$ fulfill the conditions of the above definition.
}$\Diamond$
\end{example}

The following statement highlights one of the main characteristics of a perturbed periodic lattice:

\begin{proposition} If $\Ll_0$ is a perturbation of $\ZM^d$, then, for each $\Ll \in \Xi_{\Ll_0}$, there exists a canonical bijection $ \bm l_\Ll : \Ll \to \ZM^d$ that labels a point $x\in \Ll$ by an element of $\ZM^d$,
$$
\bm l_\Ll(x) = \big (n_1(x),\ldots,n_d(x)\big ) \in \ZM^d.
$$
This bijection is compatible with the translations,
\begin{equation}\label{Eq:LCompat2}
\bm l_{\mathfrak t_{y}(\Ll)}(x-y) = \bm l_\Ll (x)-\bm l_\Ll (y), \quad \forall \ x,y \in \Ll.
\end{equation}
\end{proposition}

\proof Indeed, we can define $\bm l_\Ll$ as the map corresponding to the graph
$$
\big \{(y_n,n) \in \Ll \times \ZM^d, \ \ \{y_n\} = \Ll \cap B(n,\varepsilon) \big \}.
$$
With the stated assumptions, this is the graph of a bijection. The compatibility with the translations is obvious.\qed

\begin{proposition} Let $\Ll_0$ be a perturbation of $\ZM^d$. Then each compact subset $V$ of $\Ll \in \Xi_{\Ll_0}$ accepts a canonical order $\bar \chi_V : \Ii_{|V|} \to V$ compatible with  the translations,
\begin{equation}\label{Eq:LCompat1}
\bar \chi_{\mathfrak t_{x}(V)} = \mathfrak t_x \circ \bar \chi_V, \quad \forall \ x \in \Ll.
\end{equation}
\end{proposition}

\proof Let $ \bm l_\Ll : \Ll \to \ZM^d$ be the canonical bijection defined above. Given a subset $W \subset \Ll$, we define
\begin{equation}\label{Eq:Wj}
W_j =\{x \in W \; |\;  n_j(x) \leq n_j(y) \ \forall \ y \in W\}.
\end{equation}
Then, for compact $V \subset \Ll$, we claim that the following set
\begin{equation}\label{Eq:Vij}
\big ( \cdots \big ( \big (V_1\big )_{2}\big )_{3} \cdots \big )_d
\end{equation}
contains one and only one point. Indeed, the set $W_j$ in Eq.~\eqref{Eq:Wj} is non-empty for any non-empty set $W$. As such, $V_1$, $\big (V_1\big)_2$, etc., are all non-empty. Assume that \eqref{Eq:Vij} contains two points, $x_1$ and $x_2$. Then, by construction,
$$
n_j(x_1) \leq n_j(x_2) \ {\rm and} \ n_j(x_2) \leq n_j(x_1), \quad j =1,\ldots,d,
$$
and, as such, $\bm l(x_1) = \bm l(x_2)$. Since $\bm l$ is a bijection, this is impossible, hence \eqref{Eq:Vij} must contain a single point. We denote this point by $v_1$ and declare $\bar \chi_V(1) := v_1$. By iteration, we devise the algorithm
\begin{equation}\label{Eq:CanonicalOrder}
\bar \chi_V(1) = v_1, \  \bar \chi_V(2) = \bar \chi_{V\setminus \{\bar \chi_V(1)\}}(1), \ \bar \chi_V(3) = \bar \chi_{V\setminus \{\bar \chi_V(1),\bar \chi_V(2)\}}(1), \ {\rm etc.},
\end{equation}
which supplies the desired canonical order. Indeed, since the algorithm uses the natural order of $\ZM$, which is compatible with its group structure, the statement from Eq.~\ref{Eq:LCompat1} follows from Eq.~\eqref{Eq:LCompat2}.\qed

\begin{remark}{\rm By permuting the indices in Eq.~\eqref{Eq:Vij}, one can generate $|V|!$ distinct orders over $V$. In the following, we adopt the convention that $\bar \chi_V$ is always generated by the algorithm with the order seen in Eq.~\eqref{Eq:CanonicalOrder}. Any other possible order is generated as $\bar \chi_V \circ s$, $s \in \Ss_N$.
}$\Diamond$
\end{remark}

\begin{remark}{\rm The pattern from Example~\ref{Ex:CircleXi} is markedly different from any perturbation of periodic lattices. Indeed, that is an example of a pattern with a connected order cover over its transversal.
}$\Diamond$
\end{remark}

\subsection{Reduction of the $C^\ast$-algebra and $K$-theoretic statements}

We now fix a lattice $\Ll_0$ and assume that $\Ll_0$ is a perturbation of $\ZM^d$.

\begin{proposition} The set
\begin{equation}
\bar \Xi_{\Ll_0}^{(N)} : = \Big \{ \xi \in \Xi_{\Ll_0}^{(N)}, \ \chi_\xi = \bar \chi_{V_\xi}\Big \}
\end{equation}
is a clopen subset of $\Xi_{\Ll_0}^{(N)}$.
\end{proposition}

\proof Given the topology of $\Xi_{\Ll_0}^{(N)}$, which is inherited from the topology defined in Proposition~\ref{Pro:Cover2},  the sets
$$
\Big \{ \xi \in \Xi_{\Ll_0}^{(N)}, \ \chi_\xi = \bar \chi_{V_\xi} \circ s\Big \}, \quad s \in \Ss_N,
$$
are open and, furthermore, they are obviously disjoint. The statement then follows because their union is the full $\Xi_{\Ll_0}^{(N)}$. \qed 

\begin{definition} We define the reduced groupoid $\bar \Gg_N$ as the restriction of $\Gg_N$ to $\bar \Xi_{\Ll_0}^{(N)}$. Specifically,
\begin{equation}
\bar \Gg_N : = \mathfrak s^{-1}\Big (\bar \Xi_{\Ll_0}^{(N)}\Big ) \; \cap \; \mathfrak r^{-1}\Big ( \bar \Xi_{\Ll_0}^{(N)}\Big ),
\end{equation}
equipped with the algebraic and topological structure inherited from $\Gg_N$.
\end{definition}

\begin{remark}{\rm Both maps $\mathfrak s$ and $\mathfrak r$ are continuous, hence $\bar \Gg_N$ is a clopen subset of $\Gg_N$. This assures us that $\bar \Gg_N$ is indeed a topological groupoid. As for notation, we will use $\bar{\mathfrak s}$ and $\bar{\mathfrak r}$ for the source and range maps of $\bar \Gg_N$, respectively.
}$\Diamond$
\end{remark}

\begin{proposition} There exists an injective $C^\ast$-homomorphism 
$$
{\rm j} : C^\ast_r(\bar \Gg_N) \rightarrowtail C^\ast_r(\Gg_N).
$$
\end{proposition} 

\proof Any $\bar f, \bar g \in C_0(\bar \Gg_N)$ can be extended continuously over $\Gg_N$ by setting them to be zero on $\Gg_N \setminus \bar \Gg_N$. Then
$$
\big ({\rm j}(\bar f) \ast {\rm j}(\bar g) \big ) (\xi,\zeta) = \sum_{\eta \in \mathfrak a_n^{-1}(\Ll_\xi)} {\rm j}(\bar f)(\xi,\eta) {\rm j}(\bar g)\big(\mathfrak t_{\chi_\eta(1)}(\eta,\zeta)\big ).
$$
The product clearly returns a trivial value if $(\xi,\zeta) \notin \bar \Gg_N$. If the product is evaluated at $(\bar \xi,\bar \zeta) \in \bar \Gg_N$, note that 
both ${\rm j}(\bar f)$ and ${\rm j}(\bar g)$ return trivial values if $\chi_\eta \neq \bar \chi_{V_\eta}$. Hence, the summation can be restricted over $\mathfrak a_n^{-1}(\Ll_{\bar \xi}) \cap \bar \Gg_N$ and the latter contains all $\bar \eta$'s such that $(\bar \xi,\bar \eta) \in \bar{\mathfrak r}^{-1}(\bar \xi)$. The conclusion is that
$$
\big ({\rm j}(\bar f) \ast {\rm j}(\bar g)\big ) (\bar \xi,\bar \zeta) = \sum_{\bar \eta, (\bar \xi,\bar \eta) \in \bar{\mathfrak r}^{-1}(\bar \xi)} {\rm j}(\bar f)(\bar \xi,\bar \eta) {\rm j}(\bar g)\big(\mathfrak t_{\chi_{\bar \eta}(1)}(\bar \eta,\bar \zeta)\big ),
$$
which is the product in $C^\ast_r(\bar \Gg_N)$.\qed

\begin{proposition}\label{Pro:Mor2} $C^\ast_r(\bar \Gg_N) \simeq C^\ast_{r,\Ss_N}(\Gg_N)$.
\end{proposition}

\proof For $\xi \in \Xi_{\Ll_0}^{(N)}$, we simplify the notation by declaring that $\bar \chi_\xi = \bar \chi_{V_\xi}$, the canonical order of the subset $V_\xi \in \Ll_\xi$. We define
\begin{equation}
\Phi : C^\ast_{r,\Ss_N}(\Gg_N)  \to C^\ast_r(\bar \Gg_N), \quad \Phi(f)(\bar \xi,\bar \zeta) = N! \, f(\bar \xi,\bar \zeta),
\end{equation}
and
\begin{equation}
\begin{aligned}
&\qquad \qquad \quad \Phi^{-1} : C^\ast_r(\bar \Gg_N) \to C^\ast_{r,\Ss_N}(\Gg_N), \\
& \Phi^{-1}(\bar f)(\xi,\zeta) = \tfrac{1}{N!} \, (-1)^{\chi_\xi^{-1} \circ \bar \chi_\xi}  (-1)^{\bar \chi_\zeta^{-1} \circ \chi_\zeta} \; f\big (\bar \chi^{-1}_\xi \circ \chi_\xi \cdot (\xi,\zeta) \cdot \chi^{-1}_\zeta \circ \bar \chi_\zeta \big ),
 \end{aligned}
\end{equation}
which are obviously inverse to each other as maps between sets. We will show that they are algebra morphisms. Indeed, for $(\bar \xi,\bar \zeta) \in \bar \Gg_N \subset \Gg_N$ and $f,g \in C^\ast_{r,\Ss_N}(\Gg_N)$,
$$
\Phi (f \ast g)(\bar \xi,\bar \zeta) = N! \sum_{\eta, (\bar \xi,\eta) \in {\mathfrak r}^{-1}(\bar \xi)} f(\bar \xi,\eta) g(\mathfrak t_{\chi_\eta(1)}(\eta,\bar \zeta))
$$
and we can use the bi-equivariant property of $f$ and $g$ to continue
$$
\Phi (f \ast g)(\bar \xi,\bar \zeta) = N! \sum_{\eta, (\bar \xi,\eta) \in {\mathfrak r}^{-1}(\bar \xi)} f\big ((\bar \xi,\eta) \cdot s^{-1}\big ) g\Big(s \cdot \big(\mathfrak t_{\chi_\eta(1)}(\eta),\mathfrak t_{\chi_\eta(1)}(\bar \zeta)\big)\Big )
$$
where $s = \bar \chi_\eta^{-1} \circ \chi_\eta \in \Ss_N$. We have
$$
s \cdot \big (\mathfrak t_{\chi_\eta(1)}(\eta),\mathfrak t_{\chi_\eta(1)}(\bar \zeta) \big ) =\mathfrak t_{[\chi_{\mathfrak t_{\chi_\eta(1)}(\eta)} \circ s^{-1}](1)}\big (\Lambda_s(\mathfrak t_{\chi_\eta(1)}(\eta)),\mathfrak t_{\chi_\eta(1)}(\bar \zeta)\big )
$$
and, recalling the definition of $s$,
$$
\chi_{\mathfrak t_{\chi_\eta(1)}(\eta)} \circ s^{-1}= \chi_{\mathfrak t_{\chi_\eta(1)}(\eta)} \circ  \chi_\eta^{-1} \circ \bar \chi_\eta.
$$
At this step we use the identity $\chi_{\mathfrak t_y(\eta)}=\mathfrak t_y \circ \chi_\eta$, to write
$$
\chi_{\mathfrak t_{\chi_\eta(1)}(\eta)} \circ  \chi_\eta^{-1} \circ \bar \chi_\eta = \mathfrak t_{\chi_\eta(1)} \circ \bar \chi_\eta.
$$
Therefore, using the identity $\mathfrak t_{\mathfrak t_x (y)} = \mathfrak t_y \circ \mathfrak t_x^{-1}$,
$$
\mathfrak t_{[\chi_{\mathfrak t_{\chi_\eta(1)}(\eta)} \circ s^{-1}](1)} = \mathfrak t_{\mathfrak t_{\chi_\eta(1)} ( \bar \chi_\eta(1))}=\mathfrak t_{\bar \chi_\eta(1)}\circ \mathfrak t^{-1}_{\chi_\eta(1)} = \mathfrak t_{\chi_{\Lambda_s(\eta)}(1)} \circ \mathfrak t^{-1}_{\chi_\eta(1)}.
$$
The conclusion is 
$$
\Phi (f \ast g)(\bar \xi,\bar \zeta) = N! \sum_{\eta, (\bar \xi,\eta) \in {\mathfrak r}^{-1}(\bar \xi)} f\big (\bar \xi,\Lambda_s(\eta)\big ) g\big(\mathfrak t_{\chi_{\Lambda_s(\eta)}(1)}(\Lambda_s(\eta),\bar \zeta )\big ).
$$
Our last observations are that $\bar \eta=\Lambda_s(\eta)$ belongs to $\bar \Gg_N$ and, for each $\bar \eta \in \bar\Gg_N$, there are precisely $N!$ elements of $\Gg_N$ that are mapped in this way into $\bar \eta$. As such
 $$
\Phi (f \ast g)(\bar \xi,\bar \zeta) = \sum_{\bar \eta, (\bar \xi,\bar \eta) \in \bar {\mathfrak r}^{-1}(\bar \xi)} N!  f\big ((\bar \xi,\bar \eta \big ) \; N! g\big(\mathfrak t_{\chi_{\bar \eta}(1)}\big(\bar \eta,\bar \zeta)\big ),
$$
hence $\Phi$ is an algebra morphism.\qed 

\vspace{0.2cm}

We recall that $\bar \Xi_{\Ll_0}^{(N)}$ serves as the space of units for the restricted groupoid. Then, by the same arguments from subsection~\ref{Sec:BlowUp1}, we arrive at the following statement: 

\begin{proposition} The map 
 $$
 \bar f: \bar \Xi_{\Ll_0}^{(N)} \to \Xi_{\Ll_0}, \quad \bar f(\Ll,V,\bar \chi_V) : = (\Ll,\bar \chi_V(1)) = (\Ll,0)
 $$ 
 is a blow-up of the unit space of $\Gg_{1}$.
 \end{proposition}
 
Theorem~2.52 from \cite{WilliamsBook} and Proposition~\ref{Pro:Mor2} then give: 
 
 \begin{corollary} Let $\mathcal{L}_0$ be a perturbation of $\mathbb{Z}^{d}$. Then the $C^\ast$-algebras $C^\ast_{r,\Ss_N}(\Gg_N)$ and $C^\ast_r(\Gg_1)$ are Morita equivalent. As a consequence, their $K$-theories coincide.
 \end{corollary}

\section{Discussion and Outlook}
\label{Sec:Discussion}

Our work proposes a framework for the dynamics of many fermions, but several other tools are yet to be developed in order to demonstrate its effectiveness. In this section, we point to several directions where, in our opinion, swift progress with such tools can be made.

\subsection{Self-binding versus scattered dynamics}
\label{Sec:SBvsS}

Depending on the energy infused by an external excitation and on the type of self-interaction, the dynamics of $N$-fermions can display qualitatively distinct behaviors. Indeed, the $N$ fermions can evolve together as one cluster or they can be split into several clusters. In the first scenario, one will say that the $N$-fermions are in a self-binding state, while for the second scenario that they are in a scattered state. For example, for the many-body quantum system simulated in Fig.~11 of \cite{LiuPRB2022}, the energy spectrum separates into several spectral islands and \cite{LiuPRB2022} found that the states associated with top spectral island display all the trades of self-binding dynamics. On the other hand, the states associated with the other spectral islands display all the trades of a scattered dynamics (see Fig.~23 in \cite{LiuPRB2022} for direct evidence of the self-binding/scattered characters). Furthermore, the spectral projection onto the top spectral island belongs to the groupoid $C^\ast$-algebra, while the spectral projections onto the remaining spectral islands belong to the corona algebra. 

We conjecture that any self-bound dynamics of $N$ fermions can be generated from the bi-equivariant groupoid algebra $C^\ast_{r,\Ss_N}(\Gg_N)$, while the scattering dynamics can be generated from the corona algebra $\Mm\big ( C^\ast_{r,\Ss_N}(\Gg_N) \big ) /C^\ast_{r,\Ss_N}(\Gg_N)$. Given the statement of Proposition~\ref{Pro:Mor10} and Remark~\ref{Re:Pullback}, one can focus entirely on the separable $C^\ast$-algebra $C^\ast_{r,\Ss_N}(\Gg_N)$. In particular, computing various $K$-theories of $C^\ast_{r,\Ss_N}(\Gg_N)$ will shed light on the possible dynamical features that can be observed for $N$ self-interacting fermions hopping on an aperiodic lattice. Given the statements from section~\ref{Sec:SpecialCases}, the interesting cases are represented by the patterns with connected order covers.

\subsection{Thermodynamic limit}

As we already indicated, a Hamiltonian with finite interaction range will contain a finite number of many-body potentials. Inherently, such Hamiltonians will fall into the corona algebra when sectors with large number of fermions are considered. At first sight, this seems to be an unavoidable phenomenon when considering the thermodynamic limit $N \to \infty$. However, a Hamiltonian $H_{\Ll_0}$ can be brought back into $C^{\ast}_{r,\mathcal{S}_N}(\Gg_N) $ by sandwiching with the approximate unit introduced in Proposition~\ref{Pro:UnitApprox}. According to our previous discussion, $\widetilde H_{\Ll_0}(\epsilon) = 1_N^\epsilon H_{\Ll_0} 1_N^\epsilon$ generates the dynamics of a ``self-binding'' cluster of $N$ fermions and the size of the cluster can be adjusted by tuning the profile of the approximate unit. As such, we have a very convenient tool to enforce a desired density of fermions for such clusters. We recall that, traditionally, the fermion density is enforced during the thermodynamic limiting process by restricting the $N$ fermions to a finite lattice of proper volume. The latter, however, destroys the Galilean invariance of the models and the motion of the center of mass.

The important conclusion is that we have put forward a new paradigm for taking the thermodynamic limit and the process can be set up entirely inside the separable $C^\ast$-algebras $C^{\ast}_{r,\mathcal{S}_N}(\Gg_N) $. An interesting possibility worth investigating  is generating thermodynamic limits using directed towers of $C^\ast$-algebras
\begin{equation}\label{Eq:Diag1}
\begin{diagram}
\cdots & \rightarrowtail & \Bb_n \subset C^\ast_{r,\Ss_{N_n}}(\Gg_{N_n})   & \rightarrowtail & \Bb_{n+1} \subset C^\ast_{r,\Ss_{N_{n+1}}}(\Gg_{N_{n+1}}) & \rightarrowtail & \cdots
\end{diagram}
\end{equation}
where $N_n$, $n \in \NM$, is a sequence of increasing fermion numbers. Examples of such towers are yet to be supplied, but, if that happens, then any self-adjoint element $H$ from the algebra defined by the direct limit can be used to generate a dynamics or an equilibrium state. Furthermore, if the algebras $\Bb_n$ are invariant against the $\RM^d$-action induced by the $U(1)$-twists of the ${\rm CAR}(\Ll_0)$ generators, $a_x \mapsto e^{\imath k \cdot x} a_x$, $x\in \Ll_0$, $k \in \RM^d$, then the generator $\nabla$ of this action supplies the current operator $\nabla (H)$, which also belongs to the direct limit. In such cases, the formalism supplies the means to investigate the transport coefficients for the many-fermions models.

\end{document}